\documentclass[12pt]{report}
\usepackage{amsfonts}
\usepackage{amsmath}
\usepackage{amssymb}
\usepackage{graphics}
\usepackage[pagebackref]{hyperref}
\usepackage{amsmath,amscd}
\usepackage[demo]{graphicx}

\newtheorem{theorem}{Theorem}[chapter]
\newtheorem{corollary}[theorem]{Corollary}
\newtheorem{lemma}[theorem]{Lemma}

\newtheorem{lemma and definition}[theorem]{Lemma and Definition}
\newtheorem{proposition}[theorem]{Proposition}
\newtheorem{definition}[theorem]{Definition}
\newtheorem{notation}[theorem]{Notation}
\newtheorem{remark}[theorem]{Remark}

\newtheorem{discussion}[theorem]{Discussion}

\newtheorem{the construction}[theorem]{THE CONSTRUCTION}
\DeclareMathOperator{\Ann}{Ann} 
\DeclareMathOperator{\Ker}{Ker} \DeclareMathOperator{\Ima}{Im}
\DeclareMathOperator{\Spec}{Spec}

\DeclareMathOperator{\Sup}{sup} 
\DeclareMathOperator{\Cok}{coker} \DeclareMathOperator{\depth}{depth}
\DeclareMathOperator{\pd}{pd}\DeclareMathOperator{\height}{ht}
\DeclareMathOperator{\charact}{char} \DeclareMathOperator{\Hom}{Hom}
\DeclareMathOperator{\Mat}{Mat}
\def\1{{\rm (1)}}
\def\2{{\rm (2)}}
\def\3{{\rm (3)}}
\def\4{{\rm (4)}}
\def\5{{\rm (5)}}
\def\6{{\rm (6)}}
\begin{document}
\pagestyle{headings}
\baselineskip=16pt
%%%%%%%%%%%%%%%%%%%%%%%%%%%%%%%%%%%%%%%%%%%%%%%%%%%%%%%%%%%%
\thispagestyle{empty}
%%%%%%%%%%%%%%%%%%%%%%%%%%%%%%%%%%%%%%%%%%%%%%%%%%%%%%%%%%%%

\begin{center}

\hrule\bigskip

{\Large\sc On the Finite F-representation type and F-signature of  hypersurfaces }

\bigskip\hrule\vspace{2.5cm}

{\bf By}\bigskip

{\Large Khaled Alhazmy }\vspace{2cm}

{\Large\tt\ A thesis submitted in partial fulfilment of the requirements for the degree of
Doctor of Philosophy
}\vspace{2.5cm}

{\Large The University of Sheffield}
\bigskip\bigskip

{\Large School of Mathematics and Statistics} \vspace{2cm}

%{\tt Supervised by}\bigskip

%{\large Dr. Moty Katzman}\\
%{\large School of Mathematics and Statistics } \vspace{2cm}
July
 2017
\end{center}
\newpage

\chapter*{Abstract}

Let  $S=K[x_1,...,x_n]$  or $S=K[\![x_1,...,x_n]\!]$  be either a
polynomial or a formal power series ring in a finite number of variables
over a field $K$ of characteristic $p > 0$ with $[K:K^p] < \infty$. Let $R$ be the hypersurface $S/fS$ where  $f$ is a nonzero nonunit element of $S$.  If  $e$ is a positive integer, $F_*^e(R)$ denotes the $R$-algebra structure induced on $R$ via the $e$-times iterated  Frobenius map ( $r\rightarrow r^{p^e}$ ). We describe a matrix factorizations  of $f$  whose cokernel is isomorphic to  $F_*^e(R)$  as $R$-module. The presentation of $F_*^e(R)$ as the cokernel of a matrix factorization of $f$ enables us to find a characterization from which we can decide when  the ring $S[\![u,v]\!]/(f+uv)$  has finite F-representation type (FFRT) where $S=K[\![x_1,...,x_n]\!]$. This allows us to create  a class of rings that have finite F-representation type but  not  finite CM type. For $S=K[\![x_1,...,x_n]\!]$, we use this presentation to show that the ring $S[\![y]\!]/(y^{p^d} +f)$ has finite F-representation type for any $f$ in $S$. Furthermore,  we prove that $S/I$ has finite F-representation type when $I$ is a monomial ideal in either $S=K[x_1,...,x_n]$  or $S=K[\![x_1,...,x_n]\!]$. Finally, this presentation enables us to  compute the F-signature of  the rings $S[\![u,v]\!]/(f+uv)$ and $S[\![z]\!]/(f+z^2)$ where $S=K[\![x_1,...,x_n]\!]$ and  $f$ is a monomial in the ring  $S$. When $R$ is a Noetherian ring of prime characteristic that has FFRT, we prove that $R[x_1,...,x_n]$ and $R[\![x_1,...,x_n]\!]$ have FFRT.
We prove also that over local ring of prime characteristic  a module has FFRT if and only it has FFRT by a FFRT system. This enables us to show that if $M$ is a finitely generated module over  Noetherian ring $R$ of prime characteristic $p$,  then the  set of all prime  ideals $Q$ such that $M_Q$ has FFRT over $R_Q$ is an open set in the Zariski topology on $\Spec(R)$.

\newpage

\chapter*{Acknowledgments}

First of all, I praise Allah (God), the almighty, merciful and passionate, for granting me the capability and  providing me this opportunity to complete my PhD studies.

I owe my deepest gratitude to my PhD supervisor, Dr Moty Katzman.  Without
his inspiring guidance, his unselfish help, his encouragements, his valuable advice,   this thesis would
not have been possible. I have been  extremely  blessed to have
a supervisor who has tended to my work greatly  and responded to my questions and enquiries so
promptly. I would like to express my gratefulness and appreciation for his patience and kindness during my research which has made my PhD experience an enjoyable one.  He is an example  that I will strive to follow in my
career.

I would like to thank  the members of my viva examination  Prof. Holger Brenner and Prof. Tom Bridgland for offering their time reading and reviewing this thesis and for their suggestions and helpful discussions during and after my viva.

 It is beyond words to express my deep sense of gratitude to my father  Abdullah, and my mother  Obaidah, for their immeasurable love, continuous support  and countless prayers that stand behind me in  my life throughout. I am also so thankful for my brothers and sisters for their encouragements and help.  Their  love, assistance  and confidence in me  have sustained and inspired me throughout.

I would like to express my sincerest gratitude to my  beloved wife, Hanan Alqaisi, for her  gracious sacrifices and unconditional love. Her strong encouragement and unselfish support have been the essential reasons and the powerful  source of energy  and inspiration   behind most of the successful achievements  I have had so far. I truly believe that without her, I would not be able to achieve  or even  plan for my goals. There is no way to express properly how much I appreciate what she has done for me and I will forever be indebted to her for her  support and effort. I thank God for her presence in my life. I also thank my wonderful children: Abdullah, Abdulaziz, Faisal, and Renad, for always making me smile and for understanding on those days when I was studying instead of playing  with them. Thanks for putting up with me.

A special thanks goes to Prof. Salah-Eddine Kabbaj, Prof. Abdeslam Mimouni and Prof. Jawad Abuihlail for the courses in commutative algebra I took with them at King Fahd University of Petroleum and Minerals of Saudi Arabia. Thanks for their encouragement for me to pursue my PhD. I would like  to also thank Prof. Craig Huneke for the other  courses in commutative algebra that I took with him at the University of Kansas while I was studying my master degree.

I am very grateful to King Khalid University of  Saudi Arabia for funding my studies abroad. Without their financial support, this would not have been accomplished.

Last but not least, I am so grateful for my best friend Nabil Alhakamy for his constant support and encouragement throughout my journey. I would always remember his supportive friendship. I must also   thank my classmate  Mehmet Yesil for being a nice friend. I would like to thank my friends in Sheffield for the treasuring moments we shared together. I find myself lucky of having them during my stay in Sheffield.

%%%%%%%%%%%%%%%%%%%%%%%%%%%%%%%%%%%%%%%%%%%%%%%%%%%%%%%%%%%%
\markright{}
\tableofcontents
\newpage

%%%%%%%%%%%%%%%%%%%%%%%%%%%%%%%%%%%%%%%%%%%%%%%%%%%%%%%%%%%%
%%%%%%%%%%%%%%%%%%%%%%%%%%%%%%%%%%%%%%%%%%%%%%%%%%%%%%%%%%%%
%%%%%%%%%%%%%%%%%%%%%%%%%%%%%%%%%%%%%%%%%%%%%%%%%%%%%%%%%%%%
%\noindent {\huge\bf Acknowledgments} \markright{}
%\addcontentsline{toc}{section}{\protect\numberline{}{Acknowledgments}}

%\newpage
%%%%%%%%%%%%%%%%%%%%%%%%%%%%%%%%%%%%%%%%%%%%%%%%%%%%%%%%%%%%
%%%%%%%%%%%%%%%%%%%%%%%%%%%%%%%%%%%%%%%%%%%%%%%%%%%%%%%%%%%%
%%%%%%%%%%%%%%%%%%%%%%%%%%%%%%%%%%%%%%%%%%%%%%%%%%%%%%%%%%%%
%\noindent {\huge\bf Summary} \markright{}
%\addcontentsline{toc}{section}{\protect\numberline{}{Summary}}
%\bigskip\bigskip

%%%CHAPTER1%%%%%%%%%%%%%%%%%%%%%%%%%%%%%%%%%%%%%%%%%%%%%%%%%
%%%%%%%%%%%%%%%%%%%%%%%%%%%%%%%%%%%%%%%%%%%%%%%%%%%%%%%%%%%%
%%%%%%%%%%%%%%%%%%%%%%%%%%%%%%%%%%%%%%%%%%%%%%%%%%%%%%%%%%%%
\chapter{Introduction}

In this chapter, we give an overview of the results in this thesis. The next three chapters will introduce the necessary concepts and any unexplained terminology in this chapter will be explained there.

\section{Notation}
\label{section:Notation}

Throughout this thesis, we shall assume all rings are commutative with a unit   unless otherwise is stated.  $\mathbb{N}$ denotes the set of the positive integers and $\mathbb{Z}_{+}$ denotes the set of non-negative integers. Let $R$ be a ring. If $I$ is an ideal of  $R$ generated by the elements $r_1,...,r_n$, we write $I=(r_1,...,r_n)$. Similarly, if $M$ is a finitely generated $R$-module generated by $m_1,...,m_n$, we write $M=(m_1,...,m_n)$.  $R$ is called a  Noetherian ring if every ideal of $R$ is finitely generated. The set of all prime ideals of a  $R$ is denoted by $\Spec(R)$. $R$ is a local ring if $R$ has only one  maximal ideal, and we write $(R,\mathfrak{m})$ to mean that $R$ is a local ring with the unique maximal ideal $\mathfrak{m}$.   If $I$ is an ideal of $R$, we define $V(I)=\{ Q \in \Spec(R) \; | \; I \subseteq Q \}$ and the radical of $I$, denoted  $\sqrt{I}$, is intersection of all prime ideals in $V(I)$. If $\sqrt{I}=Q$ for some prime ideal $Q$, $I$ is said to be $Q$-primary. The collection $\{ V(I) \; | \; I \text{   is an ideal of } R \}$ defines a topology on $\Spec(R)$ that is called the Zariski topology on $\Spec(R)$ in which $V(I)$ is a closed set for any ideal $I$ of $R$.   A chain in $R$ or $\Spec(R)$ is a sequence $Q_0 \subseteq Q_1 \subseteq ... \subseteq Q_n$  of
prime ideals of $R$. If $Q$ is a prime ideal, the height of $Q$ denoted $\height(Q)$ is defined as
  $$\height(Q)= \Sup \{ n | \text{  there exists a chain  } Q_0 \subsetneqq Q_1 \subsetneqq ... \subsetneqq Q_n=Q \text{  in  } \Spec(R)\} $$
  and consequently the Krull dimension  or simply the dimension of $R$,  denoted   $\dim(R)$, is given by
  $$\dim(R)= \Sup\{ \height(\mathfrak{m}) | \mathfrak{m} \text{ is a maximal ideal in } R \} $$
  Furthermore,the dimension of an $R$-module $M$,  denoted   $\dim_RM$ or $\dim M$, is the dimension of the ring $R/\Ann M$ where $\Ann M$ is the ideal $\Ann M = \{r \in R | rM=0 \} $.
The characteristic of $R$, denoted $\charact(R)$, is the smallest positive integer $n$ such that $\sum_{i=1}^n 1 = 0$; if no such $n$ exists, then the characteristic is zero.

\section{Outline of Thesis}
\label{section:Outline of Thesis}
In Chapter \ref{chapter:Preliminaries}, we gather in section \ref{section:General Background} the necessary concepts making the general  background of this thesis. This includes the definitions and the theorems that we need in the subsequent chapters. We state in this section the theorems, propositions and lemmas that have references without proof.  In section \ref{section:Technical Lemmas} of this chapter, we prove technical lemmas that are essential for  some of the   main results in the chapters \ref{chapter:The finite F-representation type over hypersurfaces} and \ref{chapter:F-signature of specific hypersurfaces}.

Chapter \ref{chapter:Matrix Factorization} is devoted to the concept of matrix factorization. In section \ref{section:Definitions and Properties}, we provide the definition and the main properties of this concept and fix notations for specific hypersurfaces. Section \ref{section:Matrix factorization and Maximal Cohen Maculay modules} explains how the concept of matrix factorization can be related to the subject of maximal Cohen Maculay (MCM) modules over specific hypersurfaces.

In 1997, K. Smith and M. Van den Bergh \cite{SV} introduced the notion of finite F-representation type (FFRT) over a class of rings for which the Krull-Remak-Schmidit Theorem holds as a characteristic $p$ extension of the notion of finite Cohen-Macaulay representation
type.  They then showed that rings with finite
F-representation type play a role in  developing the theory of differential operators
on the rings. However, Y.Yao  in his paper \cite{Y} studied the notion of FFRT in more general settings.
 T.Shibuta summarized in \cite{TS} several nice properties  satisfied for rings of finite F-representation type.
For example,  the Hilbert-Kunz multiplicities of
such rings are proved to be  rational numbers by G.Seibert \cite{S}. Y.Yao in his paper \cite{Y} proved that tight closure commutes
with localization in such rings.

 Chapter \ref{Chapter: Modules of finite Frepresentation type} provides in section \ref{section:Definition and examples} the definition and examples of the notion of finite
 F-representation type (FFRT). Y.Yao in \cite{Y} observed that the localization and the completion of modules with FFRT have also FFRT. In section \ref{section:Several FFRT extensions of FFRT rings} we prove this observation and that both of $R[x]$ and $R[\![x]\!]$ have FFRT whenever $R$ has FFRT (Theorem \ref{C4.6}). The rings that have FFRT by FFRT system were introduced by  Y.Yao in \cite{Y}. We prove in section \ref{section: Rings of FFRT by FFRT system} that over local ring of prime characteristic  a module has FFRT if and only if it has FFRT by a FFRT system (Theorem \ref{P1}). This result enables us in section \ref{section: FFRT locus of a module is an open set} to prove  that the FFRT locus of a module is an open set in the Zariski topology on $\Spec(R)$ (Theorem \ref{P4.11}).

We fix in chapter \ref{chapter:The finite F-representation type over hypersurfaces} the notation that  $S=K[x_1, \ldots , x_n]$ or $S=K[\![x_1, \ldots , x_n]\!]$ where $K$ is a  filed of prime characteristic $p$ with $[K:K^p] < \infty $, and $f$ is a nonzero nonunit element in $S$. In  section \ref{section:The presentation of  Fe as a cokernel of a Matrix Factorization of f}, we describe  a presentation of  $F_*^e(S/fS)$ as a cokernel of a matrix factorization of $f$ (Theorem \ref{P5.11}).  This presentation of  $F_*^e(S/fS)$  as a cokernel of a matrix factorization of $f$  allows us  in section \ref{section:When does the ring fuv have finite F-representation type?} to find a characterization of when the ring $S[\![u,v]\!]/(f+uv)$, where $K$ is an algebraically closed field of prime characteristic $p>2$,  has FFRT (Theorem \ref{P30}) and hence in section \ref{section:Class of rings that have FFRT but not finite CM type} we provide  a class of rings that have finite F- representation type but it does not have finite Cohen-Macaulay type (Theorem \ref{P7.10}). In section \ref{section:Qutient by   a monomial ideal}, we use this presentation to prove that $S/I$ has FFRT for any monomial ideal $I$ in $S$ (Theorem \ref{P5.25}). Furthermore, if  $S=K[\![x_1,...,x_n]\!]$, we use this presentation to show in section \ref{section:The ring pd has finite Frepresentation type} that the ring $S[\![y]\!]/(y^{p^d} +f)$ has finite F-representation type for any $f$ in $S$ (Theorem \ref{P5.1}).

%If $M$ is a finitely generated module over a local ring $(R, \mathfrak{m})$, then  by \cite[Proposition 2.1]{NW} $M$ can be decomposed as $ M \cong R^a \oplus N$  where $a$ is a nonnegative integer and $N$ is an  $R$-module that has no free direct summand. The number $a$ is unique and independent of the decomposition as when we take the $\mathfrak{m}$-adic completion, by the Krull-Remak-Schmidt theorem \cite[Theorem 2.13.]{NW} and the theorem \cite[Theorem 3.6.]{NW} $a$ is uniquely determined. The number $a$ is the maximal rank of a free direct summands of $M$ and is denoted by $\sharp (M,R)$.

%As a result,

 The notion of the $F$-signature (Definition \ref{D6.1})  was   introduced and defined  by C. Huneke and G. Leuschke in \cite{HL} on an F-finite local ring of prime
characteristic with a perfect residue field.  Y. Yao in his work \cite{Y2} has defined the F-signature to arbitrary local rings $R$ without the assumptions that the residue field is perfect. K.Tucker in his paper \cite{KT} proved that this limit exists. The $F$-signature  seems to give subtle information on the singularities of  $R$. For example I. Aberbach and G.
Leuschke \cite{AL} have proved  that the $F$-signature is positive if and only if $R$ is strongly F-regular. Furthermore, We have $\mathbb{S}(R) \leq 1$ with equality if and only if R is regular \cite[Theorem 4.16.]{KT} or \cite[Corollary 16]{HL}.
%, the $F$-signature of any of the two-dimensional
%quotient singularities ($A_n$), ($D_n$), ($E_6$), ($E_7$), ($E_8$) is the reciprocal of the order of the group
%$G$ defining the singularity \cite[Example 18]{HL}.

 In Chapter \ref{chapter:F-signature of specific hypersurfaces}, we compute the $F$-signature of specific hypersurfaces. Indeed,  if  $K$ is a  field of prime characteristic $p$ with $[K:K^p] < \infty$ and  $S=K[\![x_1,...,x_n]\!]$, we compute in the sections \ref{section:The F-signature of uv is a monomial} and \ref{section:The Fsignature of z when  is a monomial} the $F$-signatures of the rings $S[\![u,v]\!]/(f+uv)$ and $S[\![z]\!]/(f+z^2)$ where $f$ is a monomial and  $u,v,z$ are variables over $S$. The presentation of $F_*^e(S/fS)$ as a cokernel of a matrix factorization of $f$ playes a role in those computations.   We proved also in  section \ref{section:The Fsignatures are the same} that,  for any $f\in S=K[\![x_1,...,x_n]\!]$, the rings $S/fS$ and $S[\![y_1,...,y_n]\!]/fS[\![y_1,...,y_n]\!]$ have the same $F$-signature and we give in section \ref{section:The Fsignature of pd} a characterization of  the F-signature of the ring $S[\![y]\!]/(y^{p^d}+f)$ for any $f \in S$ and $d \in \mathbb{N}$ .

%%%CHAPTER3%%%%%%%%%%%%%%%%%%%%%%%%%%%%%%%%%%%%%%%%%%%%%%%%%
%%%%%%%%%%%%%%%%%%%%%%%%%%%%%%%%%%%%%%%%%%%%%%%%%%%%%%%%%%%%
%%%%%%%%%%%%%%%%%%%%%%%%%%%%%%%%%%%%%%%%%%%%%%%%%%%%%%%%%%%%
\chapter{Preliminaries}
\label{chapter:Preliminaries}
%%%%%%%%%%%%%%%%%%%%%%%%%%%%%%%%%%%%%%%%%%%%%%%%%%%%%%%%%%%%

  In the first section of this chapter,  we review and provide the basic background from commutative algebra that is needed for this thesis and the theorems, propositions and lemmas that have references are stated in this section   without proof. The second section contains  some technical lemmas that are essential for obtaining the main results of this thesis.
\section{General Background}
\label{section:General Background}

\subsection{Graded modules and Rings}
\label{subsection:Graded modules and Rings}
Let $G$ be an abelian group with an additive operation $+$ and identity element $e \in G$ and let $R$ be a ring.
\begin{definition}\cite[Section 1.1]{NV}
 \emph{ $R$ is a $G$-graded ring if  there exists  a family $\{ R_{\alpha} \, | \, \alpha \in G \}$ of additive subgroups $R_{\alpha}$  of $R$ such that $R=\bigoplus_{\alpha \in G}R_{\alpha}$ (as groups)  and $R_{\alpha}R_{\beta} \subseteq R_{\alpha + \beta }$ for all $\alpha , \beta \in G$. A nonzero element $x \in R$ is a homogeneous of degree $\alpha\in G$ and we write $\deg(x)=\alpha$ if $x \in R_{\alpha}$}.
\end{definition}
\begin{definition}\cite[Section 2.1]{NV}
  \emph{Let $M$ be an $R$-module. If $R=\bigoplus_{\alpha \in G}R_{\alpha}$ is a $G$-graded ring  , then $M$ is a $G$-graded module  if  there exists  a family $\{ M_{\alpha} \, | \, \alpha \in G \}$ of additive subgroups $M_{\alpha}$  of $M$ such that $M=\bigoplus_{\alpha \in G}M_{\alpha}$ (as groups)  and $R_{\alpha}M_{\beta} \subseteq M_{\alpha + \beta }$ for all $\alpha , \beta \in G$. A nonzero element $x \in M$ is a homogeneous of degree $\alpha\in G$ and we write $\deg(x)=\alpha$ if $x \in M_{\alpha}$}.
\end{definition}
\begin{remark}\cite[Section 2.1]{NV}\label{R2.3}
  Let $R=\bigoplus_{\alpha \in G}R_{\alpha}$ be a $G$-graded ring and let $M=\bigoplus_{\alpha \in G}M_{\alpha}$ be a $G$-graded $R$-module. If $N$ is a submodule of $M$ and $ N_{\alpha}=N \cap M_{\alpha}$ for each $\alpha \in G$, then $N$ is  a graded submodule of $M$ if $N=\bigoplus_{\alpha \in G}N_{\alpha}$. Furthermore, if $N$ is  a graded submodule of the $G$-graded module $M=\bigoplus_{\alpha \in G}M_{\alpha}$, then the quotient module $M/N$ is also a $G$-graded $R$-module as follows:  $M/N=\bigoplus_{\alpha \in G}(M/N)_{\alpha} $ where $(M/N)_{\alpha}=(M_{\alpha}+N)/N= M_{\alpha}/N_{\alpha}$ for all $\alpha \in G$.
\end{remark}
\subsection{Projective and flat Modules}
\label{subsection:Projective Modules}
Let $R$ be a ring  and  $M$ be an $R$-module throughout this subsection.
\begin{definition}
  \emph{$M$ is said to be projective if  for every surjective module homomorphism  $g: L \twoheadrightarrow N$ and every module homomorphism $h: M\rightarrow N$, there exists a module  homomorphism $f:M\rightarrow L $ such that $ gf=h $}
\end{definition}
\begin{definition}\cite[Section 2.4]{BCA}\label{Def2.5}
 \emph{ A system $\mathfrak{B}=\{m_i\}_{i\in I}$ of elements of $M$ is said to be linearly independent
(over $R$) if the condition $\sum_{i\in I}r_im_i=0$  with $r_i \in R$ for every $i$ and $r_i=0$ for
almost all $i$ implies that $r_i=0$ for every $i$. We say that $\mathfrak{B}$ is a basis of $M$ (over
$R$) if $\mathfrak{B}$ generates $M$ as an $R$-module and is linearly independent over $R$. An
$R$-module is said to be free (or $R$-free) if it has a basis. Furthermore, if an $R$-module $M$ has a  finite free basis, $M$ is said to be a free module of finite rank}.
\end{definition}

\begin{remark}\cite[Proposition 4.7.6]{BCA}
 Every free module is a projective module.
\end{remark}

\begin{definition}\cite[Section 17.4]{BCA}

 \emph{ A projective resolution of $M$ is an exact sequence
$$ ... \rightarrow P_{n+1}\xrightarrow{d_{n+1}}P_n \rightarrow \cdots \rightarrow   P_1 \xrightarrow{d_1}P_0 \xrightarrow{\epsilon}M \rightarrow 0   $$
where $P_j$ is projective $R$-module for every $j$. If there exists $n$ such that $P_j=0$ for all $j \geq n+1 $, then we say that $M$ has a finite resolution of length $\leq n$ }.

\end{definition}

It is well known \cite[Proposition 17.4.4]{BCA} that every module has a  projective resolution not necessarily finite.

\begin{definition}\cite[Section 18.2]{BCA}
 \emph{The projective dimension of $M$, denoted $\pd M$  or $\pd_RM$, is defined by
\begin{equation*}
  \pd M=\inf \{ n \, |\, M \text{ has a projective resolution of length } \leq n \}
\end{equation*}
when  $M$ has
no finite projective resolution, we write $\pd M= \infty$}.
\end{definition}
\begin{definition}\cite[Section 4.7]{BCA}\label{D2.7}
  \emph{ $M$ is said to be flat if  for every injective  module homomorphism  $g: L \longrightarrow N$, the induced module homomorphism  $1\otimes g: M\otimes_RL \longrightarrow M\otimes_RN$ is injective. If $M$ is a flat $R$-module, $M$ is said to be faithfully flat if for every
nonzero $R$-module $N$, $ M\otimes_RN \neq 0$. Furthermore,  If $S$ is an $R$-algebra,i.e. $S$ is a ring and there exists a ring homomorphism $\phi : R \rightarrow S$, then $S$ is flat (respectively faithfully flat) algebra over $R$ if $S$ is flat (respectively faithfully flat) as $R$-module}.
\end{definition}

\subsection{ Cokernel of  Matrix}
\label{subsection:The Cokernl of a Matrix}
\begin{notation}
 \emph{If  $m,n \in \mathbb{N}$, then  $M_{m \times n}(R)$ (and $M_{ n}(R)$)   denote the set  of all $m \times n$  (and $n \times n$)  matrices over a ring $R$ (where $R$ is not necessarily commutative in this notation). If $A \in M_{n\times m}(R)$ is the matrix representing the $R$-linear map  $\phi : R^{n}\longrightarrow R^{m}$  given by $ \phi (X)= AX$ for all $ X \in R^{ ^n}$, then we write $A:R^{n}\longrightarrow R^{m}$ to denote the $R$-linear map $\phi$ and $\Cok_R(A)$ denotes the cokernl of $\phi$ while   $ \Ima_R(A)$ denotes the image of $\phi $. We write $\Cok(A)$ and $\Ima(A)$ if $R$ is known from the context}.
\end{notation}

\begin{definition}
 \emph{If $A,B \in M_n(R)$ where $R$ is a commutative ring, we say that $A$ is equivalent to $B$ and we write $A\sim B$ if there exist invertible matrices $U,V \in M_n(R)$ such that $A =UBV$}.
\end{definition}
We can observe the following remark.
\begin{remark}\label{R18}
\begin{enumerate}
  \item [(a)] If $A \in M_{n\times m}(R)$ and $B \in M_{s \times t}(R)$, then  we define $A\oplus B$ to be  the matrix in $M_{(n+s) \times (m+t)}(R)$ that is given by $A\oplus B = \left[
                                                                                                                     \begin{array}{cc}
                                                                                                                       A &  \\
                                                                                                                         & B \\
                                                                                                                     \end{array}
                                                                                                                   \right]$. In this case,  $ \Cok_R(A\oplus B)=\Cok_R(A)\oplus \Cok_R(B).$
  \item [(b)] If $A, B \in M_n(R)$ are equivalent matrices, then $\Cok_R(A)$ is isomorphic to  $\Cok_R(B)$ as $R$-modules.
\end{enumerate}
\end{remark}
If $A$ and $B$ are matrices in $M_n(R)$ such that $\Cok_R(A)$ is isomorphic to  $\Cok_R(B)$, it is not true in general that this implies that $A$ is equivalent to $B$ \cite[Section 4]{LR}. However, the following Proposition gives a partial converse of \ref{R18}(b).
\begin{proposition}\label{L3.5}
Let $(R,\mathfrak{m})$ be a Noetherian local ring and  let $A\in M_s(R)$ and $B\in M_t(R)$  be two matrices determining the $R$-linear maps $A:R^s \rightarrow R^s$ and  $B:R^t \rightarrow R^t$ such that  $A$ and $B$ are injective and all entries of $A$ and $B$ are in $\mathfrak{m}$. If $\Cok_R(A)$ is isomorphic to  $\Cok_R(B)$ as $R$-modules, then $s=t$ and $A$ is equivalent to $B$.
\end{proposition}
\begin{proof}
Consider the following diagram with exact rows\\
\begin{center}
$\begin{CD}
0 @>>> R^t @>B>> R^t @>\delta>> \Cok_R(B) @>>> 0 \\
 @.         @V \beta VV     @V\alpha VV          @V\mu V \approx V \\
0 @>>> R^s@>A>> R^s @> \pi >> \Cok_R(A) @>>> 0
\end{CD} $
\end{center}

 The projectivity  of  $R^t$ induces the $R$-linear map $\alpha$ that makes the right square of the diagram commute. Since $\Ima( \alpha B ) \subseteq \Ima(A)$, the projectivity of $R^t$ induces $\beta$ that makes the diagram commutative. The fact that all entries of $A$  are in $\mathfrak{m}$ yields that  $\Ima(A) \subseteq \mathfrak{m}R^s $. Now, if $y \in R^s$, there exists $x \in R^t$ such that $\pi(y)= \mu \delta (x) =\pi \alpha (x)$. Therefore, $y- \alpha(x) \in \Ker (\pi) = \Ima( A )$ and then $y \in   \Ima (\alpha) +\Ima( A) \subseteq \Ima (\alpha) + \mathfrak{m}R^{s} $. This implies that $R^s= \Ima (\alpha) + \mathfrak{m}R^{s} $ and hence by Nakayamma lemma \cite[Lemma 2.1.7]{BCA} it follows that $\alpha$ is  surjective. The surjectivity of $\alpha$ shows that  $t \geq s$. By similar argument we show that $s \geq t$ and hence $s=t$. Since $\alpha:R^s \rightarrow R^s $ is surjective, $\alpha$ is isomorphism \cite[Theorem 2.4]{Mat}. The five lemma \cite[Proposition 2.72]{JR} confirms that $\beta$ is isomorphism too. If $U$ and $V$ are the invertible matrices defining $\alpha$ and $\beta$  (see Remark \ref{Rem 2.9}), then $UB=AV$ as desired.
\end{proof}

\subsection{Presentation of finitely generated  modules}
\label{subsection:Presentation of a module}
In this subsection,  $R$ is a  ring and $M$ is an $R$-module.
\begin{definition}\cite[Section 4.2]{BCA}
 \emph{$M$ is said to be finitely presented if there exists an exact sequence
$G \xrightarrow{\psi} F \xrightarrow{\phi} M \xrightarrow{} 0$ where $G$ and $F$ are free modules of finite rank}.
\end{definition}
\begin{remark}\label{Rem 2.9}
If $R$ is Noetherain and $M$ is a finitely generated $R$-module, there exists an exact sequence
$G \xrightarrow{\psi} F \xrightarrow{\phi} M \xrightarrow{} 0$ where $G$ and $F$ are free modules of finite rank. If $\{g_1,...,g_m\}$ and $\{f_1,...,f_n\}$  are bases for $G$ and $F$ respectively and $ \psi(g_j)= \sum_{i=1}^n a_{ij}f_i$ for all $1 \leq j \leq m$, then  the matrix $A=[a_{ij}]$ is said to represent $\psi$ with $\Cok A = M$ and we write $\Mat(\psi)=A$. If $G \xrightarrow{\theta} G $ and $G \xrightarrow{\tau} G $ are $R$-linear maps where $G$ is a free module of finite rank, then $\Mat(\theta + \tau)= \Mat(\theta) + \Mat(\tau)$ and $\Mat(\theta  \tau)= \Mat(\theta)  \Mat(\tau).$
\end{remark}

\subsection{Localization of modules and rings}
\label{subsection:Localization of modules and rings}
Let $R$ be a ring and $M$ be an $R$-module throughout this subsection.
\begin{definition}\cite[Section 2.7]{BCA}
 \emph{If $W$ is a multiplicative closed set, i.e  $1\in W$ and $st \in W$ for all $s,t \in W$,   define the relation $\sim $ on $W \times M$ by $(s,m) \sim (t,n)$, where $(s,m) , (t,n) \in W \times M$, if and only if there exists $w \in W$ such that $wtm=wsn$. The relation $\sim$ is an equivalence relation and for every $(s,m) \in W \times M$ let $ \frac{m}{s}$ denote the  equivalence class of $(s,m)$. Let $W^{-1}M$ denote the set of all equivalence classes  $ \frac{m}{s}$ for all $(s,m) \in W \times M$}.
\end{definition}
It is straightforward to verify the following well known fact.
\begin{proposition}
  If $W$ is a multiplicative closed set, then:

  \begin{enumerate}

    \item [(a)]$W^{-1}R$ is a commutative ring for which the addition and the  multiplication  are given by $\frac{a}{s} + \frac{b}{t}=\frac{ta+sb}{st}$ and $\frac{a}{s} . \frac{b}{t}=\frac{ab}{st}$ for all $\frac{a}{s} , \frac{b}{t} \in W^{-1}R$.
    \item [(b)]$W^{-1}M$ is an $W^{-1}R$-module for which  the addition and the scalar multiplication  are given by $\frac{m}{s} + \frac{n}{t}=\frac{tm+sn}{st}$ and $\frac{a}{u} \frac{m}{s}=\frac{am}{us}$ for all $\frac{a}{u}\in W^{-1}R$ and $\frac{m}{s} , \frac{n}{t} \in W^{-1}M$.
  \end{enumerate}
\end{proposition}

\begin{remark}
  \begin{enumerate}
    \item [(a)] If $P$ is a prime ideal in $R$, then $W= R \setminus P$ is a multiplicative closed set and we write $M_P$ (or $R_P$) to denote $W^{-1}M$ (or $W^{-1}R$).
    \item [(b)] If $u \in R\setminus \{0\}$, then $W=\{1,u,u^2,...,u^n,...\}$ is a multiplicative closed set and we write $M_u$ (or $R_u$) to denote $W^{-1}M$ (or $W^{-1}R)$.
  \end{enumerate}
\end{remark}

Recall from Corollary4.20, Rule4.22, and Example 4.18 in ChapterIII of   \cite{Kunz} and \cite[Theorem 4.4]{Mat}   the following property of  localizations.

\begin{proposition}\label{LLL2}
Let $S$ and $T$ be a multiplicative closed sets of a ring $R$ and let $\hat{T}$ be the image of $T$ in $S^{-1}R$.  For every $R$-module $M$, it follows that:
\begin{enumerate}
  \item [(a)] If $S \subseteq T$, then $\hat{T}^{-1}(S^{-1}R) \cong T^{-1}R$ as rings and $\hat{T}^{-1}(S^{-1}M) \cong T^{-1}M$ as $T^{-1}R$-modules.
  \item [(b)] If $f,g \in R$ and $\frac{g}{1}$ is the image of $g$ in $R_f$, then $(R_f)_{\frac{g}{1}}\cong R_{fg}$ as rings and $(M_f)_{\frac{g}{1}}\cong M_{fg}$ as $R_{fg}$-modules.
  \item [(c)] Let $I$ be an ideal of $R$, let $P \in \Spec(R)$ contain $I$, and let $\bar{P}$ be the image of $P$ in $R/I$. If $\rho(\frac{r+I}{s+I})=\frac{r}{s}+I_P$ for all $\frac{r+I}{s+I} \in(R/I)_{\bar{P}}$, then the map $\rho:(R/I)_{\bar{P}}\rightarrow R_P/I_P$ is a ring isomorphism.
  \item [(d)] $S^{-1}M$ is isomorphic to $S^{-1}R\otimes_{R}M$ as $S^{-1}R$-modules.
\end{enumerate}
\end{proposition}
Let $W$ be a multiplicative closed set of a ring $R$. If $M$ and $N$ are $R$-modules, there exists an $W^{-1}R$-linear map $W^{-1}\Hom_R(M,N) \rightarrow \Hom_{W^{-1}R}(W^{-1}M,W^{-1}N)$ given by $\frac{f}{w}\mapsto \frac{1}{w}W^{-1}f$ for every $f \in \Hom_R(M,N)$ and $w \in W$ where $W^{-1}f \in \Hom_{W^{-1}R}(W^{-1}M,W^{-1}N)$ is given by $W^{-1}f(\frac{m}{w})= \frac{f(m)}{w}$ for all $\frac{m}{w} \in W^{-1}M$. However, this $W^{-1}R$-linear map is isomorphism in the following case.

\begin{lemma}\cite[Lemma 4.87.]{JR}\label{LL1}
Let $W$ be a multiplicative closed set of a ring $R$. If $M$ is finitely presented, then  $W^{-1}\Hom_R(M,N)$ is isomorphic to $\Hom_{W^{-1}R}(W^{-1}M,W^{-1}N)$ as $W^{-1}R$-module for every $R$-module $N$.  In particular case, if $R$ is Noetherian and $M$ is finitely generated, then $W^{-1}\Hom_R(M,N)$ is isomorphic to $\Hom_{W^{-1}R}(W^{-1}M,W^{-1}N)$ as $W^{-1}R$-module for every $R$-module $N$.
\end{lemma}

We need the following lemma later in section \ref{section: FFRT locus of a module is an open set}.
\begin{lemma}\label{LL2}
Let $M$ and $N$ be finitely generated modules over a Noetherian ring $R$. If $W$ is a multiplicative closed set of $R$ such that $W^{-1}M$ is isomorphic to $W^{-1}N$ as $W^{-1}R$-module, then there exists $u \in W$ such that $M_u$ is isomorphic to $N_u$ as $R_u$-module.
\end{lemma}
\begin{proof}
Let $ \phi \in \Hom_{W^{-1}R}(W^{-1}M,W^{-1}N)$ be an isomorphism. By lemma \ref{LL1}, there exists $f \in  \Hom_R(M,N)$ and $t\in W$ such that $\phi = \frac{1}{t}W^{-1}f$. Therefore, $f$ induces the following exact sequence resulted from the exactness of localizations

\begin{equation*}
 0 \rightarrow W^{-1}(\Ker f ) \rightarrow W^{-1}M \xrightarrow{W^{-1}f} W^{-1}N \rightarrow W^{-1}(\Cok f ) \rightarrow 0.
\end{equation*}

Since $W^{-1}(\Cok f ) = \Cok W^{-1}f = \Cok \frac{1}{t}W^{-1}f= \Cok \phi $  and  $W^{-1}(\Ker f ) = \Ker W^{-1}f = \Ker \frac{1}{t}W^{-1}f= \Ker \phi  $, it follows $ W^{-1}(\Cok f )  = 0 = W^{-1}( \Ker f )$ and hence there exists $u \in W$ such that $ u \Cok f = 0 = u \Ker f$. Such $u$ exists as $\Cok f$ and $\Ker f$ are finitely generated $R$-modules. As a result, we get that $ \Cok f_u= (\Cok f )_u  = 0$ and $ \Ker f_u = ( \Ker f )_u= 0 $ which proves that $M_u$ is isomorphic to $N_u$ as $R_u$-module.
\end{proof}
\subsection{Completion of modules and  rings}
\label{subsection:The Completion of modules and  rings}
In this subsection,  $R$ is a ring and $M$ is an $R$-module.

\begin{definition}\cite[Chapter 8]{BCA}
 \emph{A filtration on $R$ is a sequence $\{ I_n  \}_{n\geq 0}$ of ideals of $R$ such that $I_0=R$, $I_n \supseteq I_{n+1}$
 and $I_nI_m \subseteq I_{n+m}$ for all $m, n$. A ring with a filtration is
called a filtered ring.
If $R$ is a filtered ring with filtration $\{ I_n  \}_{n\geq 0}$, a filtration on the $R$-module $M$ is a sequence $F=\{M_n\}_{n \geq 0}$ of submodules of $M$ such that $M_0=M$, $M_n \supseteq M_{n+1}$
and $I_mM_n \subseteq M_{m+n}$ for all m, n. In this case $M$ is called a
filtered $R$-module. The condition $I_mM_n \subseteq M_{m+n}$ is sometimes expressed by
saying that the filtration $\{M_n\}_{n \geq 0}$ is compatible with the filtration $\{ I_n  \}_{n\geq 0}$.
An example of a filtration  is the $I$-adic filtration
corresponding to an ideal $I$ of $R$. This is the filtration given by $I_n = I^n$
and $M_n = I^nM$ for all $n\geq 1$  , $M_0= M$ and $I_0=R$}.
\end{definition}
\begin{definition}
   \emph{Let $R$ be a filtered ring with filtration $A=\{ I_n  \}_{n\geq 0}$ and  let $M$ be a
filtered $R$-module with (compatible) filtration $F=\{ M_n  \}_{n\geq 0}$.
\begin{enumerate}
  \item [(a)] We use $\widehat{R}_A$ (or $\widehat{R}$ if there is no ambiguity) to denote the completion of $R$ with respect to the filtration $A$  that  is defined as $$ \widehat{R}_A= \varprojlim R/I_n=\{(x_n+I_n)_{n\geq 1}\in \prod_{n\geq 1}R/I_n \mid x_{n+1}-x_n \in I_n, \forall n\geq 1 \}.$$
  \item [(b)] We use $\widehat{M}_F$ (or $\widehat{M}$ if there is no ambiguity) to denote the completion of $M$ with respect to  the filtration $F$  that  is defined as $$ \widehat{M}_F= \varprojlim M/M_n=\{(x_n+M_n)_{n\geq 1}\in \prod_{n\geq 1}M/M_n \mid x_{n+1}-x_n \in M_n, \forall n\geq 1 \}.$$
  \item [(c)]If $I$ is an ideal of $R$, the $I$-adic completion of $R$ (and $M$) is denoted by $\widehat{R}_I$ (and $\widehat{M}_I$) and consequently are defined as $$ \widehat{R}_I= \varprojlim R/I^n=\{(x_n+I^n)_{n\geq 1}\in \prod_{n\geq 1}R/I^n \mid x_{n+1}-x_n \in I^n, \forall n\geq 1 \}$$ and $$ \widehat{M}_I=\varprojlim M/I^nM=\{(x_n+I^nM)_{n\geq 1}\in \prod_{n\geq 1}M/I^nM \mid x_{n+1}-x_n \in I^nM \forall, n\geq 1 \}.$$
\end{enumerate}
}
\end{definition}
We know from \cite[Section 8]{Mat} , \cite[Chapter 10]{AM}, or \cite[Section 8.2]{BCA} that $\widehat{R}$ is a ring and $\widehat{M}$ is $\widehat{R}$-module as follows:

\begin{proposition}
  If $R$ is a filtered ring with filtration $A=\{ I_n  \}_{n\geq 0}$ and  $M$ is a
filtered $R$-module with (compatible) filtration $F=\{ M_n  \}_{n\geq 0}$, then:
\begin{enumerate}
  \item [(a)]$\widehat{R}_A$   is a ring  with addition and multiplication given by $(r_n+I_n)_{n\geq 1}+(s_n+I_n)_{n\geq 1}= ((r_n + s_n)+I_n)_{n\geq 1}$ and $(r_n+I_n)_{n\geq 1}(s_n+I_n)_{n\geq 1}= (r_n  s_n+I_n)_{n\geq 1}$ for all $(r_n+I_n)_{n\geq 1},(s_n+I_n)_{n\geq 1}\in \hat{R}$
  \item [(b)]$\widehat{M}_F$ is an $\widehat{R}_A$-module with addition and scalar multiplication given by $(x_n+M_n)_{n\geq 1}+(y_n+M_n)_{n\geq 1}= ((x_n + y_n)+M_n)_{n\geq 1}$ and $(r_n+I_n)_{n\geq 1}(y_n+M_n)_{n\geq 1}= (r_n  y_n + M_n)_{n\geq 1}$  for all $(x_n+M_n)_{n\geq 1},(y_n+M_n)_{n\geq 1}\in \widehat{M}_F$  and $(r_n+I_n)_{n\geq 1}\in \widehat{R}_A$
\end{enumerate}
\end{proposition}
An example of the completion is the following:

\begin{proposition}\cite[Section 7.1]{E}\label{EX1}
Let $A=R[x_1,..,x_t]$ be a polynomial ring over the ring $A$. If $I=(x_1,...,x_t)$, then  $\widehat{A}_I=R[\![x_1,...,x_t]\!]$.
\end{proposition}

\begin{definition}
 \emph{Let $R$ be a filtered ring with filtration $A=\{ I_n  \}_{n\geq 0}$, and let $F=\{ M_n  \}_{n\geq 0}$ and $\tilde{F}=\{ \tilde{M}_n  \}_{n\geq 0}$
be two compatible filtrations  on $M$. The filterations $F=\{ M_n  \}_{n\geq 0}$ and $\tilde{F}=\{ \tilde{M}_n  \}_{n\geq 0}$ are equivalent  if the following holds: Given $r \geq 0$, there exist $n(r) \geq 0$ and
$\tilde{n}(r) \geq  0$ such that $M_{n(r)} \subseteq \tilde{M}_r$
and $\tilde{M}_{\tilde{n}(r)} \subseteq M_r$}.
\end{definition}
We need the following Proposition later in the subsection \ref{subsection:Modules over rings of prime characteristic p}.
\begin{proposition}\cite[Lemma 8.2.1]{BCA}\label{P2.20}
Let $R$ be a filtered ring with filtration $A=\{ I_n  \}_{n\geq 0}$. If  $F=\{ M_n  \}_{n\geq 0}$ and $\tilde{F}=\{ \tilde{M}_n  \}_{n\geq 0}$
are two compatible equivalent  filtrations  on $M$, then
$\varprojlim M/M_n$ is isomorphic  $\varprojlim M/\tilde{M}_n$ as $\hat{R}$-modules.
\end{proposition}
When $R$ is Noetherian and $M$ is a finitely generated $R$-module, we have the following useful theorem

\begin{theorem}\cite[Theorems 8.7, 8.8, and 8.14]{Mat}\label{thm 2.21}
 Let $R$ be a Noetherian ring and let $M$ be a finitely generated $R$-module.  If $I$ is an ideal of $R$, then

 \begin{enumerate}
   \item [(a)] $\widehat{M}_I$ is isomorphic to $M\otimes_R \widehat{R}_I$ as $\widehat{R}_I$-module.
   \item [(b)] $\widehat{R}_I$ is flat $R$-algebra. Furthermore, if $I$ is contained in the intersection of all maximal
ideals of $R$, then $\widehat{R}_I$ is faithfully flat $R$-algebra.
 \end{enumerate}

\end{theorem}

If $M$ and $N$ are $R$-modules, recall that $M$ is said to be a direct summand of $N$ if there exists an $R$-module $L$ such that $N= M \oplus L$.
\begin{proposition}\cite[Theorem 3.6 and Corollary 3.7]{NW}\label{Pro2.22}
  Let $(R,\mathfrak{m})$ be a  Noetherian local ring. If $M$ and $N$ are finitely generated $R$-modules, then:

  \begin{enumerate}
    \item [(a)] $M$ is a direct summand of $N$ if and only if $\widehat{M}_{\mathfrak{m}}$ is a direct summand of $\widehat{N}_{\mathfrak{m}}$.
    \item [(b)] $M \cong N$ as $R$-modules if and only if $\widehat{M}_{\mathfrak{m}}\cong \widehat{N}_{\mathfrak{m}}$ as $\widehat{R}_{\mathfrak{m}}$-modules.
  \end{enumerate}
\end{proposition}

\subsection{Modules over rings of prime characteristic p }
\label{subsection:Modules over rings of prime characteristic p}

Throughout this section, all rings are  of prime characteristic $p$ unless otherwise  stated,  $e \in \mathbb{N}$, and $q=p^e$. Let  $R$ be  a ring of prime characteristic $p$. For every  $e\in \mathbb{N}$ , and $a,b \in R$, it follows that $(a+b)^{p^e}=a^{p^e}+ b^{p^e}$ and consequently the map $F^e: R \rightarrow R$ that is given by $F^e(a)= a^{p^e}$ is a ring homomorphism. This homomorphism is called the $e$-th Frobenius iterated map on the  ring $R$.
\begin{definition}
 \emph{If $M$ is an $R$ module, $F_*^e(M)$ denotes the  $R$-module obtained via the restriction under the Frobenius homomorphism $F^e:R\rightarrow R$. Thus,  $F_*^e(M)$ is the  $R$-module that is  the same as
$M$ as an abelian group but for every $m\in M$   we set  $F_*^e(m)$ to  represent the corresponding element  in $F_*^e(M)$  and the $R$-module structure of $F_*^e(M)$ is given by
\begin{center}
$rF_*^{e}(m)=F_*^{e}(r^{p^e}m)$ for all $m \in M$ and $r\in R$
\end{center}
In a particular case, $F_*^e(R)$ is the abelian group $R$ that has an $R$-module structure via
\begin{center}
$rF_*^{e}(a)=F_*^{e}(r^{p^e}a) $ for all $a ,r\in R$.
\end{center}
}
\end{definition}
If $I$ is an ideal of  $R$, then $I^{[q]}$ denotes the ideal generated by the set $\{r^q | r \in I \}$. As a result, if $I= (r_1,...,r_n)$, then $I^{[q]}= (r_1^q,...,r_n^q)$.

%\begin{definition}\cite[Section 1]{TT}
%If $R$ and $M$ are $\mathbb{Z}$-graded, then
%$F_*^e(M)$ carries a  $ \mathbb{Q}$-graded R-module structure: We grade $F_*^e(M)$ by setting $[F_*^e(M)]_{\alpha} = [M]_{p^e\alpha}$
%if $ \alpha \in \frac{1}{p^e}\mathbb{Z}$, and  $[F_*^e(M)]_{\alpha} = 0$ for all $\alpha \not\in \frac{1}{p^e}\mathbb{Z} $.
%\end{definition}

 We can  observe the following:
 \begin{remark}\label{R2.29}

 Let $M$ and $N$  be  modules over a ring $R$. If $e\in \mathbb{N}$, it follows that:
 \begin{enumerate}
   \item [(a)] $F_*^{e+d}(M)=F_*^e(F_*^d(M))$ for all $d \in \mathbb{N}$.
   \item [(b)] $F_*^e(R)$ is a ring itself,  isomorphic to $R$, with an addition  and a multiplication  given by $F_*^e(a)+F_*^e(b)= F_*^e(a+b)$ and $F_*^e(a) F_*^e(b)= F_*^e(a b)$ for all $a$ and $b$ in $R$.
   \item [(c)] $F_*^e(M)$ is also $F_*^e(R)$-module via $F_*^e(r)F_*^{e}(m)=F_*^{e}(rm)$ for all $m \in M$ and $r\in R$.
   \item [(d)] If $I$ is an ideal, then $IF_*^e(M)=F_*^e(I^{[q]}M)$.
   \item [(e)] If $N$ is a submodule of $M$, then $F_*^e(N)$ is a submodule of $F_*^e(M)$.
   \item [(f)] If $N$ is a submodule of $M$, then $F_*^e(M)/F_*^e(N)$ is isomorphic to   $F_*^e(M/N)$ as $R$-module.
   \item [(g)] If $\phi:N\rightarrow M$ is an $R$-module homomorphism, then so is the map $F_*^e(\phi):F_*^e(N)\rightarrow F_*^e(M)$ that is given by $F_*^e(\phi)(F_*^e(m))=F_*^e(\phi(m))$ for each $m\in M$.
   \item [(h)] $F_*^e(-)$ is an exact functor on the category of $R$-modules.
   \item [(i)] If $R^q=\{r^q \, | \, r \in R \}$, then $R^q$ is a subring of $R$ and consequently $R$ is an $R^q$-module.
   \item [(j)] If $\{M_i\}_{i \in I}$ is a family of $R$-modules, then $F_*^e(\prod_{i\in I}M_i)$ is isomorphic to $ \prod_{i\in I}F_*^e(M_i)$ and $F_*^e(\bigoplus_{i\in I}M_i)$ is isomorphic to $ \bigoplus_{i\in I}F_*^e(M_i)$ as $R$-modules.
 \end{enumerate}
 \end{remark}
 \begin{proposition}\label{L2.3}
 Let $M$ be an $R$-module. If $W$ is a multiplicative closed set of $R$, then $F_*^e(W^{-1}M)$  is isomorphic to  $W^{-1}F_*^e(M)$ as $W^{-1}R$-module.
 \end{proposition}
 \begin{proof}
 For every $m \in M$ and $w \in W$, define $\phi(F_*^e(\frac{m}{w}))=\frac{F_*^e(w^{q-1}m)}{w}$. If $n \in M$ and $y \in W$ satisfy that $F_*^e(\frac{m}{w})=F_*^e(\frac{n}{y})$, then $\frac{m}{w}=\frac{n}{y}$ and hence $uym=uwn$ for some $u\in W$. Therefore, $(uwy)^{q-1}uym=(uwy)^{q-1}uwn$ implies that $F_*^e(u^qy^qw^{q-1}m)=F_*^e(u^qw^qy^{q-1}n)$ and hence $uyF_*^e(w^{q-1}m)= uwF_*^e(y^{q-1}n)$. This shows that $\frac{F_*^e(w^{q-1}m)}{w}=\frac{F_*^e(y^{q-1}n)}{y}$ and hence $\phi:F_*^e(W^{-1}M)\rightarrow W^{-1}F_*^e(M) $ is well defined. Notice for any $m\in M$ and $w\in W$ that
 \begin{equation*}
   \frac{F_*^e(m)}{w}=\frac{w^{q-1}F_*^e(m)}{w^q}=\frac{F_*^e((w^q)^{q-1}m)}{w^q}=\phi(F_*^e(\frac{m}{w^q}))
 \end{equation*}
 This shows that $\phi$ is surjective. One can verify that $\phi$ is also injective module  homomorphism over $W^{-1}R$  and hence
 $\phi:F_*^e(W^{-1}M)\rightarrow W^{-1}F_*^e(M) $ is an isomorphism as $W^{-1}R$-module isomorphism.
 \end{proof}
  \begin{proposition}\label{L2.4}
 Let $M$ be an $R$-module and let $I$ be a finitely generated  ideal of $R$. If $\widehat{M}_I$ is the $I$-adic completion of $M$, then $F_*^e(\widehat{M}_I)$ is isomorphic to $\widehat{F_*^e(M)}_I$ as $\hat{R}_I$-modules.
 \end{proposition}
 \begin{proof}
 For each $n \in \mathbb{N} $, let $J_n=I^n $. It follows that
 \begin{eqnarray*}
  \widehat{F_*^e(M)}_I &=& \{ (F_*^e(x_n)+I^nF_*^e(M))_{n=1}^{\infty}\in \prod_{n=1}^{\infty}\frac{F_*^e(M)}{I^nF_*^e(M)}\, |\, F_*^e(x_{n+1})- F_*^e(x_n) \in I^nF_*^e(M) \} \\
                       &=& \{ (F_*^e(x_n)+F_*^e(J_n^{[q]}M))_{n=1}^{\infty}\in \prod_{n=1}^{\infty}\frac{F_*^e(M)}{F_*^e(J_n^{[q]}M)} \, | \, F_*^e(x_{n+1})- F_*^e(x_n) \in F_*^e(J_n^{[q]}M) \} \\
   &=& \varprojlim\frac{F_*^e(M)}{F_*^e(J_n^{[q]}M)}.
 \end{eqnarray*}
 On the other hand, we have
 \begin{eqnarray*}
  F_*^e(\widehat{M}_I) &=& \{ F_*^e((x_n +I^nM)_{n=1}^{\infty})\in F_*^e(\prod_{n=1}^{\infty}\frac{M}{I^nM})\, | \, x_{n+1} -  x_n  \in I^nM  \} \\
                       &\cong& \{ (F_*^e(x_n)+F_*^e(I^nM))_{n=1}^{\infty}\in \prod_{n=1}^{\infty}\frac{F_*^e(M)}{F_*^e(I^nM)}\,|\, F_*^e(x_{n+1})- F_*^e(x_n) \in F_*^e(I^nM) \} \\
   &=&\varprojlim \frac{F_*^e(M)}{F_*^e(I^nM)}.
 \end{eqnarray*}

 For $m$ sufficiently larger than $n$, we find that $F_*^e(I^mM) \subseteq F_*^e(J_n^{[q]}M)$ and it is obvious that $F_*^e(J_n^{[q]}M)\subseteq F_*^e(I^nM)$. This shows by Proposition \ref{P2.20} that  $\varprojlim \frac{F_*^e(M)}{F_*^e(I^nM)}$ is isomorphic to $\varprojlim \frac{F_*^e(M)}{F_*^e(J_n^{[q]}M)}$ and hence $F_*^e(\widehat{M}_I)$ is isomorphic to $\widehat{F_*^e(M)}_I$ as $\hat{R}_I$-modules.
 \end{proof}
\begin{definition}
 \emph{  Let $M$ be an $R$-module where $R$ is a ring not necessarily of prime characteristic in this definition.  $M[x]$ denotes the set of all polynomials in $x$ with coefficients in $M$, i.e. every element in $M[x]$ has the form $\sum_{j=0}^tm_jx^j $ where $t \in \mathbb{Z}_{+}$ and $m_j \in M$ for each $0\leq j \leq t$. The zero polynomial  and the addition between two polynomials in $M[x]$ can be defined similarly as in the case of polynomial rings and for $f= \sum_{i=0}^sr_ix^i \in R[x]$ and  $m=\sum_{j=0}^tm_j x^j \in M[x]$ we define
  $$fm=\sum_{i=0}^s\sum_{j=0}^tr_im_jx^{i+j}.$$}
\end{definition}
One can check the following remark.
\begin{remark}\label{R4.3}
If $M$ is an $R$-module where $R$ is a ring not necessarily of prime characteristic in this remak, then
\begin{enumerate}
  \item [(a)] $M[x]$ is an $R[x]$-module(\cite[Chapter 2]{AM}).
  \item [(b)] $M[x]$ is isomorphic to $M \otimes_RR[x]$ as $R[x]$-modules (\cite[Chapter 2]{AM}).
  \item [(c)] If $M$ is a finitely generated $R$-module that is generated by $\{m_1,...,m_n\}$, then $M[x]$ is a finitely generated $R[x]$ -module that is generated by $\{m_1\otimes_R1,...,m_n\otimes_R1\}$.
\end{enumerate}
\end{remark}

\begin{remark}
 If $R$ is a  ring,  $M$ is an $R$-module, $e \in \mathbb{N} $ and $q=p^e$, notice that:
\begin{enumerate}
  \item [(a)] $F_*^e(M)[x]$ is $R[x]$-module with scalar multiplication given as follows:\\
 If $f= \sum_{i=0}^sr_ix^i \in R[x]$and  $m=\sum_{j=0}^tF_*^e(m_j)x^j \in F_*^e(M)[x]$ and , then

 \begin{equation}\label{Equat1}
   fm=\sum_{i=0}^{s}\sum_{i=0}^tF^e(r_i^qm_j)x^{i+j}.
 \end{equation}

  \item [(b)]$F_*^e(M[x])$ is $R[x]$-module with scalar multiplication given as follows:
 If $f= \sum_{i=0}^sr_ix^i \in R[x]$ and $F_*^e(m)=\sum_{j=0}^tF_*^e(m_jx^j) \in F_*^e(M[x])$ and , then
 \begin{equation}\label{Equat2}
   fF_*^e(m)=\sum_{i=0}^{s}\sum_{j=0}^tF_*^e(r_i^qm_jx^{qi+j}).
 \end{equation}
 \end{enumerate}
\end{remark}

\begin{proposition}\label{P2.41}
 Let $R$ be a  ring  and let $M$ be an $R$-module. \\ If $M_k=\{\sum_{j=0}^{t}F_*^e(m_jx^{qj+k})\,|\,m_j \in M \text{ and } t\in \mathbb{Z}_{+}\}$ for each $0\leq k \leq q-1$, then:
\begin{enumerate}
  \item [(a)] $M_k$ is $R[x]$-submodule of $F_*^e(M[x])$.
  \item [(b)]  $F_*^e(M[x])=\bigoplus_{k=0}^{q-1}M_k$.
  \item [(c)] $F_*^e(M)[x]$ is isomorphic to $M_k$ as $R[x]$-modules.
  \item [(d)] $F_*^e(M[x])$ is isomorphic to $ (F_*^e(M)[x])^{\oplus q}$ as $R[x]$-modules.
\end{enumerate}
\end{proposition}
\begin{proof}
  (a) If $r=\sum_{i=0}^{s}r_ix^i \in R[x]$ , $m=\sum_{j=0}^{t}F_*^e(m_jx^{qj+k})\in M_k$ and $n=\sum_{j=0}^{t}F_*^e(n_jx^{qj+k})\in M_k$, then $m+n=\sum_{j=0}^{t}F_*^e((m_j+n_j)x^{qj+k})\in M_k$ and, by the scalar multiplication given by  equation \ref{Equat2}, $$rm=(\sum_{i=0}^{s}r_ix^i)(\sum_{j=0}^{t}F_*^e(m_jx^{qj+k}))=\sum_{i=0}^{s}\sum_{j=0}^{t}F_*^e(r_i^qm_jx^{q(i+j)+k}))\in M_k.$$
  This shows that $M_k$ is $R[x]$-submodule of $F_*^e(M[x])$.

  (b) If $m = \sum_{j=0}^t m_j x^j\in M[x]$, then $F_*^e(m) = \sum_{j=0}^t F_*^e(m_jx^j)$.
For each $j\in \mathbb{Z}_{+}$, there exist unique $ 0\leq c_j \leq q-1 $ and $b_j \in \mathbb{Z}_{+}$ such that $j=qb_j + c_j$. Therefore, $F_*^e(m) = \sum_{j=0}^t F_*^e(m_jx^{qb_j + c_j})$ and accordingly we can write $F_*^e(m)= \sum_{k=0}^{q-1}f_k $  where $f_k\in M_k$ for each $0\leq k \leq q-1 $. Now let $f_k=\sum_{j=0}^{u}F_*^e(m_{k,j}x^{qj+k})\in M_k$ for all $0\leq k \leq q-1$.  Notice that $\sum_{k=0}^{q-1}f_k=0$ if and only if $\sum_{k=0}^{q-1}\sum_{j=0}^{u}F_*^e(m_{k,j}x^{qj+k})=0$ if and only if $\sum_{k=0}^{q-1}\sum_{j=0}^{u}m_{k,j}x^{qj+k}=0$. For all $0 \leq k, l \leq q-1$ and $j,i\in \mathbb{Z}_{+}$ notice that $qj+k=qi+l$ if and only if $k=l$ and $j=i$. This makes $\sum_{k=0}^{q-1}\sum_{j=0}^{u}m_{k,j}x^{qj+k}=0$ if and only if $m_{k,j}=0$ for all $0 \leq k \leq q-1$ and $0 \leq j \leq u$. Therefore $\sum_{k=0}^{q-1}f_k=0$ if and only if $f_k=0$ for all $0 \leq k \leq q-1$. This shows that  $F_*^e(M[x])=\bigoplus_{k=0}^{q-1}M_k$.

  (c) If $\sum_{j=0}^{t}F_*^e(m_j)x^{j}  \in F_*^e(M)[x]$, define $\phi(\sum_{j=0}^{t}F_*^e(m_j)x^{j}) =\sum_{j=0}^{u}F_*^e(m_jx^{qj+k})$. One can check that   $\phi(m+n)=\phi(m)+\phi(n)$ for all $m,n\in F_*^e(M)[x]$ and  $\phi:F_*^e(M)[x]\rightarrow M_k$ is a group isomorphism. Furthermore, if $r=\sum_{i=0}^{s}r_ix^i \in R[x]$ and  $m=\sum_{j=0}^{t}F_*^e(m_j)x^{j}\in F_*^e(M)[x]$, we get by the scalar multiplications  given by equations  \ref{Equat1} and \ref{Equat2} that
  \begin{equation*}
    \phi(rf)=\phi(\sum_{i=0}^{s}\sum_{j=0}^{t}F_*^e(r_i^qm_j)x^{i+j})=\sum_{i=0}^{s}\sum_{j=0}^{t}F_*^e(r_i^qm_jx^{q(i+j)+k}) \text{ and }
  \end{equation*}
  \begin{equation*}
    r\phi(f)= (\sum_{i=0}^{s}r_ix^i)(\sum_{j=0}^{t}F_*^e(m_jx^{qj+k}))=\sum_{i=0}^{s}\sum_{j=0}^{t}F_*^e(r_i^qm_jx^{q(i+j)+k}).
  \end{equation*}
  This shows that $\phi:F_*^e(M)[x]\rightarrow M_k$ is an isomorphism of $R[x]$-modules.

  (d) follows from (b) and (c).
\end{proof}

\begin{proposition}\label{Le2.34}
 Let $R$ be a ring. If  $f=\sum_{n=0}^{\infty}f_n\in R[\![x_1,...,x_t]\!]$ where $f_n$ is a homogeneous polynomial in $R[x_1,...,x_t]$ of degree $n$ for all $n \geq 0$, then $f^q=\sum_{n=0}^{\infty}f_n^q$.
\end{proposition}
\begin{proof}
 let $\mathfrak{m}$ be the maximal ideal of $R[\![x_1,...,x_t]\!]$.  If $g=\sum_{n=0}^{\infty}g_n\in R[\![x_1,...,x_t]\!]$ where $g_n$ is a homogeneous polynomial in $R[x_1,...,x_t]$ of degree $n$, notice that   $f=g$ if and only if $f+\mathfrak{m}^n=g+\mathfrak{m}^n$ for all $n\geq 1$. For every $n\geq 1$, we have
 \begin{eqnarray*}
   f^q+\mathfrak{m}^n &=& (f+ \mathfrak{m}^n)^q \\
     &=& (\sum_{j=0}^{n-1}f_j+  \mathfrak{m}^n)^q\\
     &=& (\sum_{j=0}^{n-1}f_j)^q+  \mathfrak{m}^n\\
     &=&  \sum_{j=0}^{n-1}f^q_j+  \mathfrak{m}^n \\
     &=&  h_n+\mathfrak{m}^n
 \end{eqnarray*}
 where
 \begin{equation*}
    h_n= \left\{
           \begin{array}{ll}
             \sum_{j=0}^{r_n-1}f_j^q, & \hbox{if } n=r_nq \hbox{ for some  }r_n \in \mathbb{Z}_{+} \\
             \sum_{j=0}^{r_n}f_j^q, & \hbox{if } n=r_nq+s_n \hbox{ for some  }r_n \in \mathbb{Z}_{+} \hbox{ and } 1 \leq s_n \leq q-1.
           \end{array}
         \right.
  \end{equation*}
   On the other hand, if $g=\sum_{n=0}^{\infty}f_n^q$, then for all $n\geq 1$ we get
 \begin{equation*}
  g+\mathfrak{m}^n =  g_n+\mathfrak{m}^n
 \end{equation*}
where
 \begin{equation*}
    g_n= \left\{
           \begin{array}{ll}
             \sum_{j=0}^{r_n-1}f_j^q, & \hbox{if } n=r_nq \hbox{ for some  }r_n \in \mathbb{Z}_{+} \\
             \sum_{j=0}^{r_n}f_j^q, & \hbox{if } n=r_nq+s_n \hbox{ for some  }r_n \in \mathbb{Z}_{+} \hbox{ and } 1 \leq s_n \leq q-1.
           \end{array}
         \right.
  \end{equation*}

 This shows that $f^q=\sum_{n=0}^{\infty}f_n^q$.
 \end{proof}
\begin{proposition}\label{L2.35}
  Let $R$ be a ring. If $S=R[\![x_1,...,x_t]\!]$, then $F_*^e(S)$ is isomorphic to $\prod_{n=0}^{\infty}F_*^e(R_n)$ as $S$-modules where $R_n$ is the group of all homogeneous polynomials in  $R[x_1,...,x_t]$ of degree $n$ for all $n\in \mathbb{Z}_{+}$ with $R_0=R$. Furthermore, if $f=\sum_{j=0}^{\infty}f_n\in R[\![x_1,...,x_n]\!]$ where $f_n$ is a homogeneous polynomial in $R[x_1,...,x_n]$ of degree $n$ for all $n \geq 0$, we can write $F_*^e(f)= \sum_{n=0}^{\infty}F_*^e(f_n)$.
\end{proposition}
\begin{proof}
 Notice that $R[\![x_1,...,x_t]\!]=\prod_{n=0}^{\infty}R_n$ where $R_n$ is the group of all homogeneous polynomials in  $R[x_1,...,x_t]$ of degree $n$ for all $n\in \mathbb{Z}_{+}$ with $R_0=R$. If we define $\phi (F_*^e(\sum_{n=0}^{\infty}f_n))=\sum_{n=0}^{\infty}F_*^e(f_n) $,  for every $f= \sum_{n=0}^{\infty}f_n \in R[\![x_1,...,x_n]\!] $, then $\phi: F_*^e(R[\![x_1,...,x_t]\!])\rightarrow \prod_{n=0}^{\infty}F_*^e(R_n)$ is a group isomorphism. Let  $f=\sum_{n=0}^{\infty}f_n$ and $g=\sum_{n=0}^{\infty}g_n$ be elements in $R[\![x_1,...,x_n]\!]$. Recall from Proposition \ref{Le2.34} that $f^q=\sum_{n=0}^{\infty}f_n^q$ and hence
       \begin{equation*}
         fF_*^e(g)=F_*^e(f^qg)=F_*^e( \sum_{n=0}^{\infty}h_n)
       \end{equation*}
       where $h_n = \sum_{j=0}^{r_n}f_{r_n-j}^qg_{jq+s_n}$ whenever $n=r_nq+s_n$ for $r_n,s_n \in \mathbb{Z}_{+}$ with $0 \leq s_n \leq q-1$. Notice that $ \prod_{n=0}^{\infty}F_*^e(R_n)$ can be considered as $S$-module as follows:

       If $f=\sum_{n=0}^{\infty}f_n \in S$ and $ \sum_{n=0}^{\infty}F_*^e(g_n)\in \prod_{n=0}^{\infty}F_*^e(R_n)$, then
       \begin{equation*}
         f\sum_{n=0}^{\infty}F_*^e(g_n)= \sum_{n=0}^{\infty}w_n
       \end{equation*}
       where $w_n= \sum_{j=0}^{r_n}f_{r_n-j}F_*^e(g_{jq+s_n})$ whenever $n=r_nq+s_n$ for $r_n,s_n \in \mathbb{Z}_{+}$ with $0 \leq s_n \leq q-1$. Therefore
       \begin{equation*}
         \phi(fF_*^e(g))=\phi(F_*^e(f^qg)) = f\sum_{n=0}^{\infty}F_*^e(g_n)=f\phi(F_*^e(g)).
       \end{equation*}
       This proves that $\phi: F_*^e(S)\rightarrow \prod_{n=0}^{\infty}F_*^e(R_n)$ is an isomorphism as $S$-modules and hence we can write $F_*^e(\sum_{j=0}^{\infty}f_n)=\sum_{n=0}^{\infty}F_*^e(f_n)$.
\end{proof}

 \begin{proposition}\label{L2.34}
  If $\Lambda_e$ is a subset of the ring $R$, then
  \begin{enumerate}
    \item [(a)] $\Lambda_e$ is a free basis  of $R$ as a free $R^q$-module   if and only if $\{ F_*^e(\lambda) \, | \, \lambda \in \Lambda_e\}$ is a free basis of $F_*^e(R)$ as a free $R$-module.
    \item [(b)] If $\Lambda_e$ is a free basis  of $R$ as a free $R^q$-module and $x$ is a variable on $R$, then $\{ \lambda x^j \, | \, \lambda \in \Lambda_e \text{ and } 0 \leq j \leq q-1 \}$ is a free basis of $R[x]$ as a free $R^q[x^q]$-module.
   \item [(c)] If $\Lambda_e$ is a free basis  of $R$ as a free $R^q$-module and $x_1,...,x_n$ are variables on $R$, then
   \begin{equation*}
     \{ \lambda x_1^{k_1}x_2^{k_2}...x_{n}^{k_{n}} \, | \, \lambda \in \Lambda_e ,  0\leq k_j\leq q-1 ,  1\leq j\leq n  \}
   \end{equation*}
    is a free  basis of $R[x_1,...,x_n]$ as a free $R^q[x_1^q,...,x_n^q]$-module.
    \item [(d)] If $\Lambda_e$ is a free basis  of $R$ as a free $R^q$-module and $x_1,...,x_n$ are variables on $R$, then
   \begin{equation*}
     \{ F_*^e(\lambda x_1^{k_1}x_2^{k_2}...x_{n}^{k_{n}}) \, | \, \lambda \in \Lambda_e ,  0\leq k_j\leq q-1 ,  1\leq j\leq n  \}
   \end{equation*}
    is a free  basis of $F_*^e(R[x_1,...,x_n])$ as a free $R[x_1,...,x_n]$-module.
    \item [(e)] If $\Lambda_e$ is a finite free basis  of $R$ as a free $R^q$-module of finite rank and $x$ is a variable on $R$, then $\{ \lambda x^j \, | \, \lambda \in \Lambda_e \text{ and } 0 \leq j \leq q-1 \}$ is a free basis of $R[\![x]\!]$ as a free $R^q[\![x^q]\!]$-module.
    \item [(f)] If $\Lambda_e$ is a finite free basis  of $R$ as a free $R^q$-module of finite rank and $x_1,...,x_n$ are variables on $R$, then
     \begin{equation*}
     \{ \lambda x_1^{k_1}x_2^{k_2}...x_{n}^{k_{n}} \, | \, \lambda \in \Lambda_e ,  0\leq k_j\leq q-1 ,  1\leq j\leq n  \}
   \end{equation*}
     is a finite free  basis of $R[\![x_1,...,x_n]\!]$ as a free $R^q[\![x_1^q,...,x_n^q]\!]$-module.
    \item [(g)] If $\Lambda_e$ is a finite free basis  of $R$ as a free $R^q$-module of finite rank and  $x_1,...,x_n$ are variables on $R$, then
     \begin{equation*}
     \{ F_*^e(\lambda x_1^{k_1}x_2^{k_2}...x_{n}^{k_{n}}) \, | \, \lambda \in \Lambda_e ,  0\leq k_j\leq q-1 ,  1\leq j\leq n  \}
   \end{equation*}
    is a finite  free  basis of $F_*^e(R[\![x_1,...,x_n]\!])$ as a free $R[\![x_1,...,x_n]\!]$-module.

  \end{enumerate}

 \end{proposition}

 \begin{proof}
  (a) For any finite subset $\Lambda \subseteq \Lambda_e$ notice that $ r= \sum_{\lambda\in \Lambda} r_{\lambda}^q\lambda$ if and only if  $ F_*^e(r)= \sum_{\lambda\in \Lambda} r_{\lambda}F_*^e(\lambda)$ where $r_{\lambda} \in R$ for all $\lambda \in \lambda_e$. This proves the result.

  (b) If $r \in R$,  there exists a finite set $\Lambda \subseteq \Lambda_e$ such that $r= \sum_{\lambda \in \Lambda}r_{\lambda}^q\lambda $  where $r_{\lambda}\in R$ for all $\lambda \in \Lambda$. Since every    $n \in \mathbb{Z}_{+}$ can be written as $n=uq+t$  where $u,t \in \mathbb{Z}_{+}$ and  $0 \leq t \leq q-1$, it follows that $rx^n= \sum_{\lambda \in \Lambda}r_{\lambda}^q(x^{u})^q\lambda x^t$. This shows that $\{ \lambda x^j \, | \, \lambda \in \Lambda_e \text{ and } 0 \leq j \leq q-1 \}$ is a generating set of $R[x]$ as $R^q[x^q]$-module. Our task now is to show that $\{ \lambda x^j \, | \, \lambda \in \Lambda_e \text{ and } 0 \leq j \leq q-1 \}$ is linearly independent set (See Definition \ref{Def2.5}). It is enough to show that every finite set $\Omega$ on the following form is linearly independent where
  \begin{equation*}
    \Omega= \{\lambda_{(i,j)}x^{j} \, | \,\lambda_{(i,j)}\in \Lambda_e , 0 \leq j \leq q-1, 1 \leq i \leq n_j \} \text{ where } n_j \in \mathbb{N} \text{  for all  } j.
  \end{equation*}
  For every $f \in R[x]$ and $n \in \mathbb{Z}_{+}$ , let $[f]_n$ denote the coefficient of $x^n$ in $f$. Now let $f=\sum_{j=0}^{q-1}\sum_{i=1}^{n_j}f_{(i,j)}^q\lambda_{(i,j)}x^{j}$ where $f_{(i,j)}\in R[x]$ for all $0 \leq j \leq q-1$ and $1 \leq i \leq n_i$ and we aim to show that $f=0$ implies that  $f_{(i,j)}=0$ for all $i$ and $j$. This can be achieved by proving that $[ f_{(i,j)}]_s=0$ for every $s \in \mathbb{Z}_{+}$. Let $s \in \mathbb{Z}_{+}$  and $0 \leq t \leq q-1$. If    $\alpha,\beta \in \mathbb{Z}_{+}$ with $0 \leq \beta \leq q-1$ notice  that $sq+t= \alpha q +\beta$ if and only if $s=\alpha$ and $t=\beta$ and consequently we get
  \begin{equation}\label{Equat3}
   [f]_{sq+t}=\sum_{i=0}^{n_t}([f_{(i,t)}]_s)^q\lambda_{(i,t)}.
  \end{equation}
Now if $f=0$, we get  $[f]_{sq+t}=0$. Since $\lambda_{(i,t)} \in \Lambda_e$ for all $0 \leq i \leq n_t$,   it follows from  \ref{Equat3} that  $[f_{(i,t)}]_s=0$ for all $0 \leq i \leq n_t$.  This shows that  $[f_{(i,j)}]_s=0$ for all $0 \leq j \leq q-1$,   $0 \leq i \leq n_j$, and $s \in \mathbb{Z}_{+}$ and consequently $\{ \lambda x^j \, | \, \lambda \in \Lambda_e \text{ and } 0 \leq j \leq q-1 \}$ is a basis of $R[x]$ as $R^q[x^q]$-module.

(c) Use the  result (b) above  and the  induction on $n$.

(d) Use the results above ((c) and (a)).

(e) Let $\Lambda_e= \{ \lambda_1, \ldots, \lambda_m \}$ and let $f=\Sigma_{n=0}^{\infty}r_nx^n$. For every $n \in \mathbb{Z}_{+}$, there exist $r_n,t_n \in \mathbb{Z}_{+}$  with $0 \leq t_n \leq q-1$ such that $n=qr_n+t_n$. This enables us to write

\begin{eqnarray*}
  f &=& \sum_{n=0}^{\infty}r_nx^n=\sum_{k=0}^{\infty}r_{qk}x^{qk}+ \sum_{k=0}^{\infty}r_{qk+1}x^{qk+1}+ \ldots +\sum_{k=0}^{\infty}r_{qk+q-1}x^{qk+q-1} \\
    &=& \sum_{j=0}^{q-1}\sum_{k=0}^{\infty}r_{qk+j}x^{qk+j}.
\end{eqnarray*}
For every $k,j \in \mathbb{Z}_{+}$ with $0 \leq j \leq q-1$, we can write  $r_{qk+j}= \sum_{i=1}^{m}u_{(i,j,k)}^q\lambda_i$ where $u_{(i,j,k)}\in R$ for all $i,j$ and $k$. Therefore, for each   $0 \leq j \leq q-1$ we get
\begin{eqnarray*}
  \sum_{k=0}^{\infty}r_{qk+j}x^{qk+j} &=& \sum_{k=0}^{\infty}[\sum_{i=1}^{m}u_{(i,j,k)}^q\lambda_i]x^{qk+j} \\
    &=& \sum_{i=1}^{m}\sum_{k=0}^{\infty}u_{(i,j,k)}^q\lambda_ix^{qk+j} \\
    &=& \sum_{i=1}^{m}[\sum_{k=0}^{\infty}(u_{(i,j,k)}x^{k})^q]\lambda_ix^{j} \\
    &=& \sum_{i=1}^{m}[\sum_{k=0}^{\infty}u_{(i,j,k)}x^{k}]^q\lambda_ix^{j} \text{  (Proposition \ref{Le2.34}) }.
\end{eqnarray*}
As a result,
\begin{eqnarray*}
  f &=& \sum_{j=0}^{q-1}\sum_{k=0}^{\infty}r_{qk+j}x^{qk+j} \\
    &=& \sum_{j=0}^{q-1}\sum_{i=1}^{m}[\sum_{k=0}^{\infty}u_{(i,j,k)}x^{k}]^q\lambda_ix^{j} \\
    &=& \sum_{j=0}^{q-1}\sum_{i=1}^{m}[f_{(i,j)}]^q\lambda_ix^{j}
\end{eqnarray*}
where $f_{(i,j)}= \sum_{k=0}^{\infty}u_{(i,j,k)}x^{k}$ for all $i$ and $j$. This shows that $\{ \lambda x^j \, | \, \lambda \in \Lambda_e \text{ and } 0 \leq j \leq q-1 \}$ is a generating set  of $R[\![x]\!]$ as a $R^q[\![x^q]\!]$-module. For every $f \in R[\![x]\!]$ and $n \in \mathbb{Z}_{+}$ , let $[f]_n$ denote the coefficient of $x^n$ in $f$.  Let $f=\sum_{j=0}^{q-1}\sum_{i=1}^{m}f_{(i,j)}^q\lambda_{i}x^{j}$ where $f_{(i,j)}\in R[\![x]\!]$ for all $1 \leq i \leq m$ and  $0 \leq j \leq q-1$. We aim to show that  if $f=0$ , we get that   $f_{(i,j)}=0$ for all $i$ and $j$. This can be achieved by proving that $[ f_{(i,j)}]_s=0$ for every $s \in \mathbb{Z}_{+}$. Let $s \in \mathbb{Z}_{+}$  and $0 \leq t \leq q-1$. If    $\alpha,\beta \in \mathbb{Z}_{+}$ with $0 \leq \beta \leq q-1$ notice  that $sq+t= \alpha q +\beta$ if and only if $s=\alpha$ and $t=\beta$ and consequently we get
  \begin{equation}\label{Equat4}
   [f]_{sq+t}=\sum_{i=0}^{m}([f_{(i,t)}]_s)^q\lambda_{i}.
  \end{equation}
Now if $f=0$, we get  $[f]_{sq+t}=0$. Since $\lambda_{i} \in \Lambda_e$ for all $0 \leq i \leq m$,   it follows from  \ref{Equat4} that  $[f_{(i,t)}]_s=0$ for all $0 \leq i \leq m$.  This shows that  $[f_{(i,j)}]_s=0$ for all $0 \leq j \leq q-1$,   $0 \leq i \leq m$, and $s \in \mathbb{Z}_{+}$ and consequently $\{ \lambda x^j \, | \, \lambda \in \Lambda_e \text{ and } 0 \leq j \leq q-1 \}$ is a basis of $R[\![x]\!]$ as $R^q[\![x^q]\!]$-module.

(f) Use the  result (e) above  and the  induction on $n$.

(g) Use the results above ((f) and (a)).
 \end{proof}

\begin{corollary}\label{L2.5}
Let $K$ be a  field of positive prime characteristic $p$ and $q=p^e$ for some $e \in \mathbb{N} $. Let $\Lambda_e$ be the basis of $K$ as $K^q$ vector space.

\begin{enumerate}
  \item [(a)] If $S:= K[x_1,...,x_{n}]$, then $F_*^{e}(S)$  is a free  $S$-module with the basis
\begin{center}
$\Delta_n^e:= \{ F_*^{e}(\lambda x_1^{k_1}x_2^{k_2}...x_{n}^{k_{n}}) \, | \, \lambda \in \Lambda_e ,  0\leq k_j\leq q-1 ,  1\leq j\leq n  \}$.
\end{center}
Furthermore, if $K(x_1,...,x_{n})$ is the fraction field of $K[x_1,...,x_{n}]$  and
\begin{equation*}
  \Omega_n^e:= \{ F_*^{e}(\frac{\lambda x_1^{k_1}x_2^{k_2}...x_{n}^{k_{n}}}{1}) \, | \, \lambda \in \Lambda_e ,  0\leq k_j\leq q-1 ,  1\leq j\leq n  \},
\end{equation*}

 then $\Omega_n^e$ is a basis of  $F_*^e(K(x_1,...,x_{n}))$  as $K(x_1,...,x_{n})$-vector space.
\item [(b)]If $S:= K[x_1,...,x_{n},....]$, then $F_*^{e}(S)$  is a free  $S$-module with the basis $\Delta^e = \cup_{n \geq 1}\Delta_n^e$ where $\Delta_n^e$ as  above. Furthermore, if $K(x_1,...,x_{n},...)$ is the fraction field of $K[x_1,...,x_{n},...]$  and   $\Omega^e= \cup_{n \geq 1}\Omega_n^e$ where $\Omega_n^e$ as above, then we get $F_*^e(K(x_1,...,x_{n},...))$  is a  $K(x_1,...,x_{n},...)$-vector space  with the infinite  basis $\Omega^e$.

  \item [(c)]If $S:= K[\![x_1,...,x_{n}]\!]$ and $K$ is finite $K^q$-vector space, then $F_*^{e}(S)$  is a finitely generated free  $S$-module with the basis
\begin{center}
$\Delta_n^e:= \{ F_*^{e}(\lambda x_1^{k_1}x_2^{k_2}...x_{n}^{k_{n}}) \, | \, \lambda \in \Lambda_e , 0\leq k_j\leq q-1 ,  1\leq j\leq n   \}$.
\end{center}
\end{enumerate}
\end{corollary}

\begin{proof}

(a) It is obvious from Proposition \ref{L2.34} (d) that $\Delta_n^e$ is a basis for $F_*^{e}(S)$  is a free  $S$-module. Notice that $\{ \frac{\lambda x_1^{k_1}x_2^{k_2}...x_{n}^{k_{n}}}{1} \, | \, \lambda \in \Lambda_e ,  0\leq k_j\leq q-1 ,  1\leq j\leq n  \}$ is a basis for $K(x_1,...,x_{n})$ as $K^q(x^q_1,...,x^q_{n})$-vector space. It follows from Proposition \ref{L2.34} (a) that $\Omega_n^e$ is a basis of  $F_*^e(K(x_1,...,x_{n}))$  as $K(x_1,...,x_{n})$-vector space.

(b) Notice that $ K[x_1,...,x_{n},....]= \cup_{n \geq 1}K[x_1,...,x_{n}] $ and consequently we get  $ K(x_1,...,x_{n},...)= \cup_{n \geq 1}K(x_1,...,x_{n}) $.  It follows  obviously that   $F_*^e(K[x_1,...,x_{n},...])= \cup_{n \geq 1}F_*^e(K[x_1,...,x_{n}])$ and   $F_*^e(K(x_1,...,x_{n},...))= \cup_{n \geq 1}F_*^e(K(x_1,...,x_{n}))$. Therefore,   we obtain from  (a) that $\Delta^e = \cup_{n \geq 1}\Delta_n^e$  is a basis for $F_*^e(S)$ as $S$-module and hence  $\Omega^e$  is an infinite basis of $F_*^e(K(x_1,...,x_{n},...))$  as $K(x_1,...,x_{n},...)$-vector space.

(c)It is obvious  from Proposition \ref{L2.34} (g).
\end{proof}

The following example  explains that we can not remove  the finiteness condition  in Corollary \ref{L2.5} (c).

\begin{example}
Let $K$ be  a  field of positive prime characteristic $p$ and $S=K[\![x]\!]$. Suppose that  $\Lambda_e$ is an  infinite  basis of $K$ as $K^q$ vector space and let $\Delta_e=\{\lambda x^j \, |\, \lambda \in \Lambda_e, 0 \leq j \leq q-1 \}$. We aim to show that  $\{F_*^e(\lambda x^j) \, |\, \lambda \in \Lambda_e, 0 \leq j \leq q-1 \}$ is not a basis for $F_*^e(S)$ as $S$-module. It is enough by Proposition \ref{L2.34} (a) to show that $\Delta_e$ is not a free basis for $S$ as $S^q$-module. Assume the contrary that $\Delta_e$ is a free basis for $S$ as $S^q$-module. For every $f \in S$ and $n \in \mathbb{Z}_{+}$ , let $[f]_n$ denote the coefficient of $x^n$ in $f$. Let  $\{\lambda_n\}_{n \geq 0}$ be an infinite subset of $\Lambda_e$ such that $\lambda_i \neq \lambda_j$ whenever $i\neq j$ and let  $f= \sum_{n=0}^{\infty}\lambda_nx^{n}\in S$. Therefore, there exist nonnegative integers $n_0,...,n_{q-1}$   such that $f=\sum_{j=0}^{q-1}\sum_{i=1}^{n_j}f_{(i,j)}^q\lambda_{(i,j)}x^{i}$ where $ \lambda_{(i,j)} \in \Lambda_e $ and $f_{(i,j)} \in S $ for all $0 \leq j \leq q-1$ and $1 \leq i \leq n_j$. For every $s, \alpha \in \mathbb{Z}_{+}$  and $0 \leq t,\beta \leq q-1$,  we notice  that $sq+t= \alpha q +\beta$ if and only if $s=\alpha$ and $t=\beta$ and consequently
\begin{equation*}
   \lambda_{sq+t}=[f]_{sq+t}=\sum_{i=0}^{n_t}([f_{(i,t)}]_s)^q\lambda_{(i,t)}.
  \end{equation*}
The above equation implies that $\lambda_{sq+t} \in \{\lambda_{(i,j)} \, | \, 0 \leq j \leq q-1, 1 \leq i \leq n_j \}\cup \{0\}$ and consequently $\{\lambda_n\}_{n \geq 0} \subseteq \{\lambda_{(i,j)} \, | \, 0 \leq j \leq q-1, 1 \leq i\leq n_j \}\cup \{0\}$ which is a contradiction as $\{\lambda_n\}_{n \geq 0}$ is infinite set of  distinct elements.
\end{example}

 \begin{definition}\label{D2.6}
\emph{A Noetherian ring $R$ is said to be F-finite if $F_*^1(R)$ is finitely generated $R$ module  (or equivalently that $F_*^e(R)$ is finitely generated $R$-module for all $e\in \mathbb{N}$)}
\end{definition}
\begin{remark}
If  $R$ is $F$-finite ring, then it follows that:
\begin{enumerate}
  \item [(a)]$R/I$ is $F$-finite ring for any ideal $I$ of $R$.
  \item [(b)]$R[x]$ is $F$-finite ring.
  \item [(c)] $R[\![x]\!]$ is $F$-finite ring.
\end{enumerate}
\end{remark}

\begin{proof}
Let $\Delta$ be a set of $R$ satisfying that $\{F_*^1(\delta) | \delta \in \Delta \}$ is a generating set of $F_*^1(R)$ as $R$-module.\\
 (a) Notice  that $\{F_*^1(\delta +I ) | \delta \in \Delta \}$ is a generating set of $F_*^1(R/I)$ as $R/I$-module.\\
 (b) Let $f=\sum_{j=0}^nr_jx^j \in R[x]$ and hence $F_*^1(f)=\sum_{j=0}^nF_*^1(r_j)F_*^1(x^j)$. For each $0\leq j \leq n$, $j=u_jp+t_j$ for some $u_j, t_j \in \mathbb{Z}$ with $0 \leq t_j < p$ and we can write $F_*^1(r_j)= \sum_{\delta \in \Delta }r_{(j,\delta)} F_*^1(\delta)$ where $r_{(j,\delta)} \in R$ for all $ \delta \in \Delta$. Therefore,
  \begin{equation*}
    F_*^1(f)=\sum_{j=0}^nF_*^1(r_j)F_*^1(x^j)= \sum_{j=0}^n\sum_{\delta \in \Delta }r_{(j,\delta)}
    x^{u_j} F_*^1(\delta x^{t_j}).
  \end{equation*} This shows that  $\{F_*^1(\delta x^t ) | \delta \in \Delta ,   0 \leq t \leq p-1 \}$ is a generating set of $F_*^1(R[x])$   as $R[x]$-module.\\
  (c) Let $I$ be the ideal generated by $x$ in $R[x]$ and let $A=R[x]$. Since $F_*^1(A)$ is finitely generated $A$-module, it follows from Theorem \ref{thm 2.21} that
  $F_*^1(A)\otimes_A\widehat{A}_I$ is isomorphic to $ \widehat{F_*^1(A)}_I$ as $\widehat{A}_I$-module and consequently $ \widehat{F_*^1(A)}_I$ is finitely generated $\widehat{A}_I$-module. Now apply Proposition \ref{L2.4}.
\end{proof}

 \begin{remark}\label{C2.9}
 Let $M$ be a finitely generated $R$-module. If $R$ is F-finite, then $F_*^e(M)$ is a finitely generated $R$-module for all $e \in \mathbb{N}$.
 \end{remark}

 \begin{proof}
 Let $B$ be a generating set of $M$ as $R$-module. If $A$ is a generating set of $R$ as an $R^{p^e}$-module, notice  that  $\{F_*^e(ab) | a \in A \text{  and  } b\in B \}$ is a generating set of $F_*^e(M)$ as a finitely generated $R$-module.
 \end{proof}

 If $R$ is  any  ring (not necessarily of prime characteristic $p$), a non-zero $R$-module $M$ is said to be  decomposable provided that there exist non-zero R-modules $M_1, M_2$ such that $M = M_1 \oplus M_2$; otherwise $M$ is indecomposable.

\begin{discussion}\label{disc2.33}

\begin{enumerate}
  \item [(a)] Recall that if  $M$ is  a non-zero Noetherian module, then $M$ is a direct sum (not necessarily unique) of finitely many indecomposable
modules \cite[Proposition 2.1, Example 2.3]{NW}. However, if $M$ is a finitely generated module over a complete Noetherian local ring $R$ (where $R$ is not necessarily of prime characteristic $p$), by the Krull-Remak-Schmidt theorem \cite[Corollary 1.10]{CMR} or \cite[Theorem 2.13.]{NW}, $M$  can be written uniquely up to isomorphism as a direct sum of finitely many indecomposable
$R$-modules. In other words,  if $M$ is a finitely generated module over a complete Noetherian local ring $R$, and If $ M \cong M_1\oplus ... \oplus M_s \approx N_1\oplus ... \oplus N_t$, where the $M_i$ and $N_j$ are finitely generated indecomposable $R$-modules, then
$s=t$ and, after renumbering,  $ M_i \cong N_i$ for each $i$.
  \item [(b)]As a result, if $(R, \mathfrak{m})$ is a  Noetherian local ring not necessarily of prime characteristic $p$ and $M$ is a finitely generated $R$-module, then $M$ can be decomposed as $ M \cong R^a \oplus N$  where $a$ is a nonnegative integer and $N$ is an  $R$-module that has no free direct summand. The number $a$ is unique and independent of the decomposition as when we take the $\mathfrak{m}$-adic completion,  by the Krull-Remak-Schmidt theorem we stated above  and Proposition \ref{Pro2.22}  $a$ is uniquely determined.
\end{enumerate}

\end{discussion}
The following notion   was introduced in \cite[Section 0]{Y}.

\begin{definition}\label{D2.11}
\emph{Let $(R, \mathfrak{m})$ be a Noetherian local ring. If $M$ is a finitely generated $R$-module, the maximal rank of
free direct summand of $M$ is denoted by $\sharp(M,R)$.}
\end{definition}

If $(R, \mathfrak{m})$ is  an F-finite  Noetherian local ring and $M$ is a finitely generated $R$-module, it follows from  Remark \ref{C2.9} that $F_*^e(M)$ is a finitely generated $R$-module for all $e \in \mathbb{N}$ and consequently $F_*^e(M)$ is Noetherian $R$-module for all $e \in \mathbb{N}$. From the above discussion, there exists a unique non-negative integer $a_e$ and an $R$-module $M_e$ that has no free direct summand  such that $F_*^e(M) \cong R^{a_e}\oplus M_e$.

  This proves the following Remark.
\begin{remark}\label{C2.10}
Let $(R, \mathfrak{m})$ be an F-finite  Noetherian local ring. If $M$ is a finitely generated $R$-module, for every $e \in \mathbb{N}$ there exists a unique nonnegative integer $a_e$ such that  $F_*^e(M)\cong R^{a_e}\oplus M_e$ where $M_e$ has no non-zero free direct summand with the convention that $R^0=\{0\}$. The number $a_e$ is the maximal rank of
free direct summand of the  $R$-module $F_*^e(M)$ and we write $\sharp(F_*^e(M),R)=a_e$.
\end{remark}
By \cite[Section 1]{RF} we can define  $F$-pure ring as follows.
\begin{definition}
\emph{If $(R, \mathfrak{m})$ is an F-finite  Noetherian local ring, then $R$ is $F$-pure  if $\sharp(F_*^1(R),R) > 0$ (equivalently,  $\sharp(F_*^e(R),R) > 0$ for all $e\in \mathbb{N}$).}
\end{definition}
Recall that if $I$ and $J$ are two ideals of a ring $R$, then $(I:J)$ is the following ideal
\begin{equation*}
  (I:J)=\{r \in R | rJ \subseteq I \}.
\end{equation*}
R.Fedder in his paper \cite{RF} established the following criterion for the $F$-purity of a quotient of  regular local ring of characteristic $p$.
\begin{proposition}\cite[Proposition 1.7]{RF}\label{P2.13}
Let $(S,\mathfrak{m})$ be a regular local ring (see Definition \ref{Def2.64}) of prime characteristic $p$ and let
$R = S/I$ where $I$ is an ideal of $S$. Then $R$ is $F$-pure if and only if $(I^{[p]}:I)\nsubseteq \mathfrak{m}^{[p]}.$
\end{proposition}

\subsection{Maximal Cohen Maculay Modules}
\label{subsection:Maximal Cohen Maculay Modules}

Throughout this section, $R$ is a Noetherian ring, and  $M$ is a finitely generated $R$-module.
Recall that an element $r \in R$ is called a zerodivisor on $M$ if there exists a nonzero  element $m \in M$ such that $rm=0$. An element  $r\in R$ is said to be  nonzerodivisor  on M if $r$ is not a zerodivisor on
$M$, i.e. for every $m \in M\setminus\{0\}$ it follows that $rm \neq 0$ .
\begin{definition}\cite[Section 19.1]{BCA}
\emph{An element $r$ of $R$ is said to be $M$-regular
if $rM \neq M$ and $r$  is a nonzerodivisor on $M$. The sequence $r_1,...,r_n$  is an $M$-regular sequence or simply an $M$-sequence if $M\neq (r_1,...,r_n)M$, $r_1$ is a nonzerodivisor on $M$, and $r_i$  a nonzerodivisor on $M/(r_1,...,r_{i-1})M $
for every $ 2 \leq i\leq n$. An $M$-sequence    $r_1,...,r_n$   in an ideal $I \subseteq R$,  is called  a maximal $M$-sequence in $I$ if $r_1,...,r_n, r$ is not a sequence on $M$ for any $r \in I$.}
\end{definition}
we can  observe the following remark.
\begin{remark}\label{R 2.40}
If $r_1,...,r_n$  is a sequence on $M$, then $r_1^{k_1},...,r_n^{k_n}$ is a sequence on $M$ for every positive integers $k_1,...,k_n$. In particular case, if $R$ has a prime characteristic $p$ and  $r_1,...,r_n$  is a sequence on $M$, then $r_1,...,r_n$  is a sequence on $F_*^e(M)$ for every $e\in \mathbb{N}$.
\end{remark}
It is well known \cite[ Corollary 19.1.4]{BCA} that any two maximal $M$-regular sequences in an ideal $I \subseteq R$  have the same length. This enables us to provide this definition.
\begin{definition}\cite[Section 19.1]{BCA}
\emph{Let $I$ be an ideal of $R$ with $IM\neq M$. The $I$-depth of $M$, denoted  $\depth_IM$, is the length of any maximal $M$-regular sequence in $I$.
If $(R,\mathfrak{m})$ is a Noetherian local ring and $M$ is a nonzero finitely generated
$R$-module (so that $\mathfrak{m}M \neq M$  by Nakayama Lemma \cite{BCA}) then the $\mathfrak{m}$-depth of $M$ is called
simply the depth of $M$, and in this case it is also denoted  $\depth M$.}
\end{definition}
\begin{definition}\cite[Definition 2.1.1]{BH}
\emph{Let $(R, \mathfrak{m})$ be a Noetherian local ring and let $M$ be a finitely generated $R$-module.  $M$  is  a Cohen Macaulay module if $M = 0$ or $ \depth M =
\dim M $. If  $ \depth M = \dim R $, then $M$ is called  a Maximal Cohen Macaulay module (henceforth abbreviated  MCM). If $R$ is an arbitrary Noetherian ring and $M$ is a finitely generated $R$-module, then $M$ is  Cohen Macaulay  (respectively MCM) if $M_{\mathfrak{m}}$ is a Cohen Macaulay $R_{\mathfrak{m}}$-module  (respectively a MCM $R_{\mathfrak{m}}$-module) for all maximal ideals $\mathfrak{m}$ of $R$.}
\end{definition}

\begin{lemma}\label{L2.55}
  Let $M$ be an $R$-module where $R$ is a ring (not necessarily Noetherian). Suppose that  $I$ is an ideal of $R$ such that $IM=0$. Let $P$ be a prime ideal of $R$ containing $I$. If   $\bar{P}= P/I$, then $M_P$ is isomorphic to $M_{\bar{P}}$ as $(R/I)_{\bar{P}}$-modules.
\end{lemma}
\begin{proof}
First notice that  $M$ is an $R/I$-module via the scalar multiplication $(r+I)m=rm$ for all $r+I \in R/I$ and $m\in M$  and consequently $I_PM_P=0$ makes $M_P$  an $R_P/I_P$-module with the scalar multiplication $(\frac{r}{s}+I_P)\frac{m}{t}=\frac{rm}{st}$ for all $\frac{r}{s} \in R_P$ and $\frac{m}{t} \in M_P$. As a result, it follows from Proposition \ref{LLL2}(c)   that  $M_P$ is an $(R/I)_{\bar{P}}$-modules via the scalar multiplication $\frac{r+I}{s+I}\frac{m}{t}=\frac{rm}{st}$ for all $\frac{r+I}{s+I} \in(R/I)_{\bar{P}}$ and $\frac{m}{t} \in M_P$. Furthermore, $M_{\bar{P}}$ is an $(R/I)_{\bar{P}}$-module  via the scalar multiplication $\frac{r+I}{s+I}\frac{m}{t+I}=\frac{rm}{st+I}$ for all $\frac{r+I}{s+I} \in(R/I)_{\bar{P}}$ and $\frac{m}{t+I} \in M_{\bar{P}}$. For every $\frac{m}{t} \in M_P$, define $ \phi(\frac{m}{t})= \frac{m}{t+I}$. One can check that $\phi$ defines an isomorphism   $ \phi: M_P \rightarrow M_{\bar{P}}$ of $(R/I)_{\bar{P}}$-modules.
\end{proof}
\begin{proposition}\label{P2.56}
 Let  $I$ be an ideal of $R$ such that $IM=0$. If $M$ is a Cohen Macaulay (respectively MCM) $R$-module, then $M$ is a Cohen Macaulay (respectively MCM) $R/I$-module.
\end{proposition}

\begin{proof}
  First recall that $M$ is an $R/I$-module via the scalar multiplication given by $(r+I)m=rm$ for all $r \in R$ and $m\in M$.  As a result, if $r_1,...,r_n \in R$ , then $r_1,...,r_n$  is  $M$-sequence on $R$ if and only if $r_1+I,...,r_n+I$ is  $M$-sequence on $R/I$. Therefore, if $R$ is  local, then    $M$ is a Cohen Macaulay (respectively MCM) $R/I$-module whenever $M$ is a Cohen Macaulay (respectively MCM) $R$-module. Now suppose that $R$ is non local and $M$ is a Cohen Macaulay (respectively MCM) $R$-module. This means that  $M_P$  is a Cohen Macaulay (respectively MCM) $R_P$-module for every maximal ideal $P$ of $R$. As a result, if $P$ is a maximal ideal of $R$ containing $I$ and $\bar{P}=P/I$, then $M_P$ is  a Cohen Macaulay (respectively MCM) $R_P/R_P$-module and consequently $M_P$  is a Cohen Macaulay (respectively MCM) $(R/I)_{\bar{P}}$-module. Since $M_P$ is isomorphic to $M_{\bar{P}}$ as $(R/I)_{\bar{P}}$-modules (Lemma \ref{L2.55}), it follows that $M_{\bar{P}}$ is a  Cohen Macaulay (respectively MCM) $(R/I)_{\bar{P}}$-module. This shows that $M$ is a Cohen Macaulay (respectively MCM) $R/I$-module.
\end{proof}

\begin{proposition}\cite[Theorem 2.1.3]{BH}\label{P2.43}
  Let $R$ be a Noetherian ring and $M$ a finitely generated $R$-module. Suppose that  $r_1,...,r_t$  is  an $M$-sequence and let $I=(r_1,...,r_t)$. If $M$ is a Cohen Macaulay $R$-module, then $M/IM$ is a Cohen Macaulay $R/I$-module.
\end{proposition}
The following Proposition describes the behaviour of the depth along exact sequences.
\begin{proposition}\label{P2.29}\cite[Proposition 1.2.9]{BH}
Let $ 0 \rightarrow U \rightarrow M \rightarrow N \rightarrow 0$ be an exact sequence of finite $R$-modules. If $ I \subseteq R$ is an ideal, then: \\

\begin{enumerate}
  \item [(a)]$ \depth_I(M) \geq \min \{ \depth_I(U) , \depth_I(N) \}$ .
  \item [(b)]$ \depth_I(U) \geq \min \{ \depth_I(M) , \depth_I(N)+1 \}$.
  \item [(c)]$ \depth_I(N) \geq \min \{ \depth_I(U)-1 , \depth_I(M) \}$.
\end{enumerate}

\end{proposition}
This proposition leads to the following corollary.

\begin{corollary}
Let $(R, \mathfrak{m})$ be a Noetherian local ring. If $U$ and $N$  are finitely generated $R$-modules, then the $R$-module $U\oplus N$ is MCM if and only if $U$ and $N$ are both MCM $R$-modules.
\end{corollary}

\begin{proof}
Let $M$ be the $R$-module $U\oplus N$. Then we have the following short exact sequence
\begin{equation*}
0 \rightarrow U \rightarrow M \rightarrow N \rightarrow 0.
\end{equation*}
If $U$ and $N$ are both MCM $R$-modules, it follows that $\depth N= \depth U = \dim R$. Since $\depth M\leq \dim M$ \cite[Theorem 19.2.1]{BCA}, it follows from Proposition \ref{P2.29} (a) that
\begin{equation*}
  \dim R = \min \{ \depth U , \depth N \} \leq \depth M \leq \dim M \leq \dim R.
\end{equation*}
Therefore, $\depth M =\dim R$ and consequently $M$ is MCM $R$-module. \\
Now assume that $M = U\oplus N$ is MCM $R$-module.
First we will show that $\depth U=\depth N$. Assume that $\depth U<\depth N$.  Since $$ 0 \rightarrow U \rightarrow M \rightarrow N \rightarrow 0$$ is a short exact sequence, it follows from Proposition \ref{P2.29} (b) that
$$\min \{ \depth M , \depth N+1 \} \leq \depth U.$$  If $\depth N+1 =\min \{ \depth M , \depth N+1 \}$, then $\depth N+1 \leq \depth U < \depth N$ which is absurd. This makes
$$\dim R= \depth M = \min \{ \depth M , \depth N+1 \} \leq \depth U < \depth N $$  which is absurd too (as $\depth N \leq \dim N \leq \dim R$). Therefore, the assumption that $\depth U<\depth N$ is impossible and we conclude that $\depth N \leq \depth U$. Using the fact that $ 0 \rightarrow N \rightarrow M \rightarrow U \rightarrow 0$ is a short exact sequence and similar argument as above we conclude that $\depth U \leq \depth N$ and consequently  $\depth U = \depth N$. Now the fact that $ 0 \rightarrow U \rightarrow M \rightarrow N \rightarrow 0$ is a short exact sequence and  Proposition \ref{P2.29}(b) imply that
$$\min \{ \depth M , \depth N+1 \} \leq \depth U$$ and consequently $$\dim R=\depth M =\min \{ \depth M , \depth N+1 \} \leq \depth U=\depth N\leq \dim R.$$ This shows that $\dim R=\depth M =\depth U=\depth N$ as desired.
\end{proof}

An easy induction yields the following corollary.
\begin{corollary}\label{C2.30}
Let $(R, \mathfrak{m})$ be a Noetherian local ring. If $M_1,...,M_n$ are finitely generated $R$-modules, then the $R$-module $ \bigoplus_{i=1}^nM_i$ is MCM if and only if $M_i$ is MCM for every $1 \leq i \leq n$.
\end{corollary}

\begin{definition}\cite[Background]{RW}
\emph{The ring $(R,\mathfrak{m})$ is said to have finite Cohen-Macaulay type (or finite CM type) if there are, up to isomorphism, only finitely many indecomposable MCM R-modules.}
\end{definition}

In 1957, M.Auslander and D.Buchsbaum introduced a formula that relates  the projective dimension of an $R$-module $M$  with  $\depth M$ and $\depth R$  as follows

\begin{proposition}\cite[Section 19.2]{BCA}\label{P2.23}
Let $(R,\mathfrak{m})$ be a Noetherian local ring and $M$ a nonzero finitely generated $R$-module. If $\pd_RM < \infty$, then
\begin{equation*}
  \pd M+\depth M=\depth R.
\end{equation*}
\end{proposition}

\subsection{Multiplicity and simple singularities}
\label{subsection:The multiplicity and simple singularity}
 $R$ is a ring and $M$ is an $R$-module throughout this subsection (which provides the required material for the main result in section \ref{section:Class of rings that have FFRT but not finite CM type}).
\begin{definition}\cite[Section 2]{Mat}
 \emph{ A chain $M=M_0  \supset M_1 \supset  ... \supset M_n = 0 $ of submodules of $M$ is called a composition series if each $M_j/M_{j+1}$ is simple,i.e $M_j/M_{j+1} \cong R/ \mathfrak{m}_j$ for some maximal ideal $\mathfrak{m}_j$ in $R$. In this case, $n$ is called the length of the composition series.}
\end{definition}
Any two composition series of an $R$-module $M$, by Jordan--H\"{o}lder theorem \cite[Theorem 6.1.4]{BCA}, have the same length. This yields the following definition.

\begin{definition}\cite[Section 2]{Mat}
 \emph{ An $R$-module $M$ is called of finite length if it has a composition series. The length of this composition series, denoted $\ell_R(M)$, is called the length  of $M$.}
\end{definition}

\begin{definition}\cite[Section 20.1]{BCA} \label{Def2.64}
 \emph{ Let $(R,\mathfrak{m})$ be a Noetherian local ring. If $\dim(R)=n$, we say that $R$ is a regular local ring or (RLR) if $\mathfrak{m}$ can be generated by $n$ elements.}
\end{definition}

\begin{definition}\label{D7.5}  \cite[Definition A.19 ]{CMR} \\
\emph{Let $(R,\mathfrak{m},k)$ be a local ring of dimension $d$, let $I$ be an
$\mathfrak{m}$-primary ideal of $R$, and let $M$ be a finitely generated $R$-module.
 The multiplicity of $I$ on $M$ is defined by
$$e_R(I,M) = \lim_{n\rightarrow \infty}\frac{d!}{n^d}\ell_R(M/I^nM) $$
where  $\ell_R(-)$ denotes length as an $R$-module. In particular we set $e_R(M)=e_R(\mathfrak{m},M)$ and call it the multiplicity of $M$. Finally, we denote $e(R)=e_R(R)$
and call it the multiplicity of the ring $R$.}
\end{definition}

\begin{proposition}\label{F1} \cite[Corollary  A.24]{CMR}\\
   Let $(S,\mathfrak{n})$ be a regular local ring and $f \in S$ a non-zero nonunit.
Then the multiplicity of the hypersurface ring $R = S/( f )$ is the largest
integer $t$ such that $f \in  \mathfrak{n}^t$.
\end{proposition}

\begin{definition} \cite[Definition 9.1]{CMR}
\emph{Let $(S,\mathfrak{n})$ be a regular local ring, and let $R = S/( g )$, where
$0 \neq g \in \mathfrak{n}^2$. We call R a simple  singularity  provided there are only finitely
 many ideals $L$ of $S$ such that $g \in  L^2$}.
\end{definition}

\begin{proposition}\label{F2}  \cite[ Lemma  9.3]{CMR} \\
 Let $(S,\mathfrak{n},k)$ be a regular local ring, $0 \neq f \in \mathfrak{n}^2$, and $R =S/( f )$
with $d = \dim(R)>1$.
If $R$ is a simple singularity and $k$ is an infinite field, then $e(R) \leq 3$.
\end{proposition}

\begin{proposition}\label{F3} \cite[Theorem 9.2]{CMR} \\
Let $(S,\mathfrak{n})$ be a regular local ring, $0 \neq f \in \mathfrak{n}^2$, and $R =S/( f )$.
If $R$ has finite CM type, then $R$ is a simple singularity.
\end{proposition}

\section{Technical Lemmas}
\label{section:Technical Lemmas}
Throughout   this section, we adopt the following notation
\begin{notation}
\emph{Let $\mathfrak{P}$ denote a ring with identity that is not necessarily commutative.
Let $m$ and $ n$ be  positive integers. If $\lambda \in \mathfrak{P}$, $ 1 \leq i \leq m $ and $ 1 \leq j \leq n  $, then $L_{i,j}^{m \times n}(\lambda)$ (and $L_{i,j}^{ n}(\lambda)$) denotes the $m \times n$ (and $n\times n$) matrix whose $(i,j)$ entry is $\lambda$ and the rest are all zeros. When $i\neq j$, we write $E_{i,j}^n(\lambda):= I_n+L_{i,j}^n(\lambda)$ where $I_n$ is the identity matrix in $M_n(\mathfrak{P})$. If there is no ambiguity, we write $E_{i,j}(\lambda)$ (and $L_{i,j}(\lambda)$)  instead of $E_{i,j}^n(\lambda)$ (and $L_{i,j}^{ n}(\lambda)$).}
\end{notation}

It is easy to observe the following remark

\begin{remark}\label{r2.1}
Let $m,n$ and $k$ be  positive integers such that  $1 \leq k,m \leq n$ with $k\neq m$. If $ \lambda \in \mathfrak{P}$ and $A \in M_n(\mathfrak{P})$,  then :
\begin{enumerate}
  \item [(a)] $E_{k,m}(\lambda)A$ is the matrix obtained from $A$ by adding $\lambda$ times row $m$ to row $k$.
  \item [(b)] $AE_{k,m}(\lambda)$ is the matrix obtained from $A$ by adding $\lambda$ times column  $k$ to column $m$.
\end{enumerate}

\end{remark}

\begin{lemma}\label{L2.9}
Let $m$ be an integer with $m \geq 2$  and $n=2m$. If $A$ is a matrix in $M_n(\mathfrak{P})$ that is given by
\begin{equation*}
  A= \left[
       \begin{array}{ccccccc}
         b &   &  &   &   &  & x  \\
         0 & b &  &   &   &   & \\
         1 & 0 & b &   &   &   &  \\
           & 1 & 0  & b &   &   &   \\
           &   & \ddots & \ddots & \ddots &   &  \\
           &   &   &  1 & 0 & b &   \\
           &   &   &   &  1 & 0 & b \\
       \end{array}
     \right]
\end{equation*}
then there exist two invertible matrices $M,N \in M_n(\mathfrak{P})$ such that:
\begin{enumerate}

\item[(a)]

$M$  has the form
\begin{equation*}
   M= \left[
        \begin{array}{ccccccc}
          1 & 0 & a_{1,3} &   & \ldots  &   & a_{1,n} \\
            & 1 & 0 & a_{2,4} &   &  &  \\
            &   & 1 & 0 &   &   &  \vdots \\
            &   &   & \ddots & \ddots & \ddots &   \\
            &   &   &   & \ddots & \ddots & a_{n-2,n} \\
            &   &   &   &   & 1 & 0 \\
            &   &   &   &   &   & 1 \\
        \end{array}
      \right].
\end{equation*}
%$M$  has the form
%\begin{equation*}
 % M= \left[
 %                                  \begin{array}{cc}
 %                                   I_2  & \tilde{M}  \\
 %                                      & I_{2(m-1)} \\
  %                                 \end{array}
  %                               \right]
%\end{equation*}
% where $I_2$ and $I_{2(m-1)}$ are the identity matrices in $M_2(\mathfrak{P})$ and $M_{2(m-1)}(\mathfrak{P})$ respectively,
%and $\tilde{M}$ is  a $2 \times 2(m-1)$ matrix over $\mathfrak{P}$.\\
 \item[(b)]$N$ has the form
 \begin{equation*}
   N= \left[
        \begin{array}{ccccccc}
          1 & 0 & b_{1,3} &   & \ldots  &   & b_{1,n} \\
            & 1 & 0 & b_{2,4} &   &  &  \\
            &   & 1 & 0 &   &   &  \vdots \\
            &   &   & \ddots & \ddots & \ddots &   \\
            &   &   &   & \ddots & \ddots & b_{n-2,n} \\
            &   &   &   &   & 1 & 0 \\
            &   &   &   &   &   & 1 \\
        \end{array}
      \right].
 \end{equation*}
 \item[(c)]
 \begin{equation*}
   MAN=\left[
       \begin{array}{ccccccc}
         0 &   &  &   &   & (-1)^{m-1}b^m & x  \\
         0 & 0 &  &   &   &   &(-1)^{m-1}b^m \\
         1 & 0 & 0 &   &   &   &  \\
           & 1 & 0  & 0 &   &   &   \\
           &   & \ddots & \ddots & \ddots &   &  \\
           &   &   &  1 & 0 & 0 &   \\
           &   &   &   &  1 & 0 & 0 \\
       \end{array}
     \right].
 \end{equation*}

 \end{enumerate}
\end{lemma}

 \begin{proof}
 We will prove the result by induction on $m\geq 2$. Let
 \begin{equation*}
   A=\left[
       \begin{array}{cccc}
         b & 0 & 0 & x \\
         0 & b & 0 & 0 \\
         1 & 0 & b & 0 \\
         0 & 1 & 0 & b \\
       \end{array}
     \right]
 \end{equation*}
  It follows from remark \ref{r2.1} that
 \begin{equation*}
  E_{2,4}(-b)E_{1,3}(-b)A E_{1,3}(-b) E_{2,4}(-b)=\left[
       \begin{array}{cccc}
         0 & 0 & -b^2 & x \\
         0 & 0 & 0 & -b^2 \\
         1 & 0 & 0 & 0 \\
         0 & 1 & 0 & 0 \\
       \end{array}
     \right].
 \end{equation*}
 Taking  $M=E_{2,4}(-b)E_{1,3}(-b)$ and $N=E_{1,3}(-b) E_{2,4}(-b)$ yields that
 \begin{equation*}
   M=N=\left[
         \begin{array}{cccc}
           1 & 0 & -b & 0 \\
           0 & 1 & 0 & -b \\
           0 & 0 & 1 & 0 \\
           0 & 0 & 0 & 1 \\
         \end{array}
       \right].
 \end{equation*}

 Now let $n=2(m+1)$ and let $A$ be the $n \times n$ matrix that is given by
 \begin{equation*}
  A= \left[
       \begin{array}{ccccccc}
         b &   &  &   &   &  & x  \\
         0 & b &  &   &   &   & \\
         1 & 0 & b &   &   &   &  \\
           & 1 & 0  & b &   &   &   \\
           &   & \ddots & \ddots & \ddots &   &  \\
           &   &   &  1 & 0 & b &   \\
           &   &   &   &  1 & 0 & b \\
       \end{array}
     \right].
 \end{equation*}
 Let $\hat{A}$ be the  $2m \times 2m$ matrix, obtained from $A$ by deleting the last two rows and the last tow columns of $A$, that is given by
 \begin{equation*}
 \hat{A}= \left[
       \begin{array}{ccccccc}
         b &   &  &   &   &  &    \\
         0 & b &  &   &   &   & \\
         1 & 0 & b &   &   &   &  \\
           & 1 & 0  & b &   &   &   \\
           &   & \ddots & \ddots & \ddots &   &  \\
           &   &   &  1 & 0 & b &   \\
           &   &   &   &  1 & 0 & b \\
       \end{array}
     \right].
 \end{equation*}
We can write
\begin{equation*}
 A= \left[
      \begin{array}{cccc|cc}
          &   &   &   &  0 & x \\
          &  \hat{A} &   &   &  0 &  0 \\
          &   &   &   & \vdots  & \vdots  \\
          &   &   &   &   &   \\\hline
        0 & \ldots &  1 & 0 & b &  0 \\
        0 & \ldots &  0 & 1 & 0 & b \\
      \end{array}
    \right].
\end{equation*}
By the induction hypothesis (where $x=0$), there exist   two matrices $\hat{M},\hat{N} \in M_{2m}(\mathfrak{P})$  such that

\begin{equation*}
 \hat{M}\hat{A}\hat{N}=\left[
       \begin{array}{ccccccc}
         0 &   &  &   &   & (-1)^{m-1}b^m & 0  \\
         0 & 0 &  &   &   &   &(-1)^{m-1}b^m \\
         1 & 0 & 0 &   &   &   &  \\
           & 1 & 0  & 0 &   &   &   \\
           &   & \ddots & \ddots & \ddots &   &  \\
           &   &   &  1 & 0 & 0 &   \\
           &   &   &   &  1 & 0 & 0 \\
       \end{array}
     \right].
\end{equation*}

 let $B$ and $C$ be $n \times n$ matrices, where $n=2(m+1)$,that are  given by

  $ B = \left[
                                                                                     \begin{array}{cc}
                                                                                       \hat{M} &  \\
                                                                                         & I_2 \\
                                                                                     \end{array}
                                                                                   \right]$ and $ C = \left[
                                                                                     \begin{array}{cc}
                                                                                       \hat{N} &  \\
                                                                                         & I_2 \\
                                                                                     \end{array}
                                                                                   \right]$
where $I_2$ is the identity matrix in $M_2(\mathfrak{P})$.

As a result, it follows that

\begin{equation*}
  BAC =\left[
       \begin{array}{ccccccc|cc}
         0 &   &  &   &   & (-1)^{m-1}b^m & 0 & 0 &x  \\
         0 & 0 &  &   &   &   &(-1)^{m-1}b^m & 0 & 0 \\
         1 & 0 & 0 &   &   &   & & & \\
           & 1 & 0  & 0 &   &   & & &  \\
           &   & \ddots & \ddots & \ddots &   & & & \\
           &   &   &  1 & 0 & 0 &  &  &   \\
           &   &   &   &  1 & 0 & 0 &  &  \\ \hline
           &   &   &   &    &  1 &  0 & b &  \\
           &   &   &   &    &   & 1  & 0 & b \\
       \end{array}
     \right].
\end{equation*}
Now  multiply the $(n-1)$-th row by $(-1)^{m}b^m$ and add it to the first row and then  multiply the $(n-3)$-th  column by $-b$ and add it to the $(n-1)$-th column. After that, multiply the $n$-th row  by $(-1)^{m}b^m$ and add it to the second  row and then  multiply the $(n-2)$-th column by $-b$ and add it to the $n$-th column.  It follows that

\begin{equation*}
MAN
=\left[
       \begin{array}{ccccccc}
         0 &   &  &   &   & (-1)^{m}b^{m+1} & x  \\
         0 & 0 &  &   &   &   &(-1)^{m}b^{m+1} \\
         1 & 0 & 0 &   &   &   &  \\
           & 1 & 0  & 0 &   &   &   \\
           &   & \ddots & \ddots & \ddots &   &  \\
           &   &   &  1 & 0 & 0 &   \\
           &   &   &   &  1 & 0 & 0 \\
       \end{array}
     \right]
\end{equation*}
where $M=E_{2,n}((-1)^{m}b^m)  E_{1,n-1}((-1)^{m}b^m)B $ and $N= C E_{n-3,n-1}(-b) E_{n-2,n}(-b)$. It is clear from the construction of the matrices $B$ and $C$ and from  remark \ref{r2.1} that $M$ and $N$ have the right form.
 \end{proof}

 \begin{corollary}\label{C2.14}
Let $n=2m+1$ where $m$ is an integer with $m \geq 2$ and let $A$ be a matrix in $M_n(\mathfrak{P})$ given by
\begin{equation*}
  A= \left[
       \begin{array}{ccccccc}
         b &   &  &   &   & x & 0  \\
         0 & b &  &   &   &   & y \\
         1 & 0 & b &   &   &   &  \\
           & 1 & 0  & b &   &   &   \\
           &   & \ddots & \ddots & \ddots &   &  \\
           &   &   &  1 & 0 & b &   \\
           &   &   &   &  1 & 0 & b \\
       \end{array}
     \right].
\end{equation*}
Then there exist  invertible matrices $M$ and $N$ in $M_n(\mathfrak{P})$ such that
\begin{equation*}
 MAN =\left[
       \begin{array}{ccccccc}
         0 &   &  &   &   &  x &  (-1)^{m}b^{\frac{n+1}{2}} \\
         0 & 0 &  &   &   & (-1)^{m-1}b^{\frac{n-1}{2}}  & y \\
         1 & 0 & 0 &   &   &   &  \\
           & 1 & 0  & 0 &   &   &   \\
           &   & \ddots & \ddots & \ddots &   &  \\
           &   &   &  1 & 0 & 0 &   \\
           &   &   &   &  1 & 0 & 0 \\
       \end{array}
     \right].
\end{equation*}
\end{corollary}

\begin{proof}
Let $ \hat{A}$ be the $2m \times 2m $ matrix obtained from $A$ be deleting the last row and the last column of $A$

\begin{equation*}
\hat{A}=\left[
       \begin{array}{ccccccc}
         b &   &  &   &   &  & x  \\
         0 & b &  &   &   &   & \\
         1 & 0 & b &   &   &   &  \\
           & 1 & 0  & b &   &   &   \\
           &   & \ddots & \ddots & \ddots &   &  \\
           &   &   &  1 & 0 & b &   \\
           &   &   &   &  1 & 0 & b \\
       \end{array}
     \right].
\end{equation*}
It follows that

 \begin{equation*}
   A= \left[
        \begin{array}{ccccc|c}
            &   &   &   &   & 0 \\
            &   &   &   &   & y \\
            &   &  \hat{A} &   &   & 0 \\
            &   &   &   &   & \vdots \\
            &   &   &   &   & 0 \\ \hline
          0 & \ldots & 0 & 1 & 0 & b \\
        \end{array}
      \right].
 \end{equation*}

 By Lemma \ref{L2.9}  there exist $2m \times 2m$ matrices $\hat{M}$ and $ \hat{N}$ such that

 \begin{equation*}
  \hat{M}\hat{A}\hat{N}=\left[
       \begin{array}{ccccccc}
         0 &   &  &   &   & (-1)^{m-1}b^m & x  \\
         0 & 0 &  &   &   &   &(-1)^{m-1}b^m \\
         1 & 0 & 0 &   &   &   &  \\
           & 1 & 0  & 0 &   &   &   \\
           &   & \ddots & \ddots & \ddots &   &  \\
           &   &   &  1 & 0 & 0 &   \\
           &   &   &   &  1 & 0 & 0 \\
       \end{array}
     \right].
 \end{equation*}

 Let $B$ and $C$ be the matrices in $M_n(\mathfrak{P})$   given by

 \begin{equation*}
   B= \left[
        \begin{array}{ccc|c}
            &   &   & 0 \\
           & \hat{M} &   & \vdots \\
            &   &   & 0 \\ \hline
          0 & \ldots & 0 & 1 \\
        \end{array}
      \right]    \text{   and   } C= \left[
        \begin{array}{ccc|c}
            &   &   & 0 \\
           & \hat{N} &   & \vdots \\
            &   &   & 0 \\ \hline
          0 & \ldots & 0 & 1 \\
        \end{array}
      \right].
 \end{equation*}

 As a result, it follows that
 \begin{equation*}
   BAC= \left[
        \begin{array}{ccccc|c}
           0 &   &   & (-1)^{m-1}b^m  &  x & 0 \\
          0  & 0  &   &   & (-1)^{m-1}b^m  & y \\
           1 & 0  & 0  &   &   & 0 \\
            & \ddots  & \ddots  & \ddots  &   & \vdots \\
            &   &   1& 0  & 0  & 0 \\ \hline
          0 & \ldots & 0 & 1 & 0 & b \\
        \end{array}
      \right].
 \end{equation*}

 Now multiply the last row of $BAC$ by $(-1)^{m}b^m$ and add it to the first row after that multiply the $(n-2)$-th column  of $BAC$ by $-b$ and add it to the last column. This produces the required result.
 Indeed, if  $M = E_{1,n}((-1)^{m}b^m)B$ and  $N= C E_{n-2,n}(-b)$, we get that $M$ and $N$ are invertible matrices satisfying that

 \begin{equation*}
 MAN =\left[
       \begin{array}{ccccccc}
         0 &   &  &   &   &  x &  (-1)^{m}b^{m+1} \\
         0 & 0 &  &   &   & (-1)^{m-1}b^{m}  & y \\
         1 & 0 & 0 &   &   &   &  \\
           & 1 & 0  & 0 &   &   &   \\
           &   & \ddots & \ddots & \ddots &   &  \\
           &   &   &  1 & 0 & 0 &   \\
           &   &   &   &  1 & 0 & 0 \\
       \end{array}
     \right].
\end{equation*}
\end{proof}
\begin{lemma}\label{L.16}
Let $n$ be an integer   with $n \geq 2 $. If $ A \in M_n(\mathfrak{P})$ is given by
 $$ A = \begin{bmatrix} b & & & & &  \\
                        1 & b & & & & \\
                          & 1 & b & & & \\
                          & & \ddots & \ddots & & \\
                         & & &1& b   \end{bmatrix} , $$
                         there exist invertible upper triangular matrices $B, C \in M_n(\mathfrak{P})$ such that the $(i,i)$ entries of $B$ and $C$ are the identity element of $\mathfrak{P}$ for all $i=1, ..., n $ and
                        \[ BAC = \left[\begin{array}{ccccc} 0 & & & &  (-1)^{n+1}b^n \\
                        1 & 0 & & &  \\
                          & 1 &0 & &  \\
                         &  & \ddots &\ddots & \\
                         & & &1& 0 \end{array} \right]. \]
\end{lemma}

\begin{proof}

 We will prove this lemma by induction on $n \geq 2$. If $ A= \begin{bmatrix} b & 0  \\ 1 & b \end{bmatrix} $, then $ E_{1,2}(-b) A E_{1,2}(-b)= \begin{bmatrix} 0 & (-1)^3b^2 \\ 1 & 0 \end{bmatrix}$ as required. Let $ A$ be a matrix in $ M_{n+1}(\mathfrak{P}) $ that is  given by

 \begin{equation*}
   A = \begin{bmatrix} b & & & & &  \\
                        1 & b & & & & \\
                          & 1 & b & & & \\
                          & & \ddots & \ddots & & \\
                         & & &1& b   \end{bmatrix}.
 \end{equation*}

  Let $\hat{A} $ be  the $n \times n$  matrix over $\mathfrak{P}$ obtained from $A$ by deleting the last row and last column of $A$
 \begin{equation*}
\hat{A} = \begin{bmatrix} b & & & & &  \\
                        1 & b & & & & \\
                          & 1 & b & & & \\
                          & & \ddots & \ddots & & \\
                         & & &1& b   \end{bmatrix}.
 \end{equation*}

  It follows that  $ A$  can be written as

 \begin{equation*}
A=
\left[
\begin{array}{c|c}
   & 0 \\ \hat{A} & \vdots \\ &  0 \\ \hline
  0 \ldots 1 & b\\
\end{array}
\right].
 \end{equation*}

By the induction hypothesis, there exist invertible  upper triangular matrices $\hat{B}, \hat{C} \in M_n(\mathfrak{P})$ such that the $(i,i)$ entry of $\mathfrak{P}$ and $C$ is the identity element of $A$ for all $i=1, ..., n $ and
                         $$\hat{B} \hat{A}\hat{C} = \begin{bmatrix} 0 & & & &  (-1)^{n+1}b^n \\
                        1 & 0 & & &  \\
                          & 1 &0 & &  \\
                          & & \ddots & \ddots & & \\
                         & & &1& 0  \end{bmatrix}.  $$

Let
\begin{equation*}
M=\left[
\begin{array}{cccc|c}
 & & & & 0 \\ & \hat{B} & & &\vdots \\ &&&&0 \\\hline
  0& \ldots &0& 0 & 1\\
\end{array}
\right]
\end{equation*}
and
\begin{equation*}
 N= \left[
\begin{array}{cccc|c}
 & & & & 0 \\ &\hat{C} & & &\vdots \\ &&&&0 \\\hline
  0& \ldots &0& 0 & 1\\
\end{array}
\right].
\end{equation*}

As a result, it follows that

\begin{equation*}
  MAN= \left[\begin{array}{ccccc|c} 0 & & & &  (-1)^{n+1}b^n &0 \\
                        1 & 0 & & & &  \\
                          & 1 &0 & & & \\
                         &  & \ddots &\ddots & &\\
                         & & &1& 0 & 0 \\ \hline
                      0 & & &0 & 1 & b \end{array} \right].
\end{equation*}

Now if  $B=E_{1,n+1}((-1)^{n+2}b^n)M$ and $ C=NE_{ n,n+1}(-b) $, then $B$ and $C$ are invertible upper triangular matrices in $M_{n+1}(\mathfrak{P})$ such that the $(i,i)$ entry of $B$ and $C$ is the identity element of $\mathfrak{P}$ for all $i=1, ..., n+1 $ and

\begin{equation*}
BAC=\begin{bmatrix} 0 & & & &  (-1)^{n+2}b^{n+1} \\
                        1 & 0 & & &  \\
                          & 1 &0 & &  \\
                          & & \ddots & \ddots & & \\
                         & & &1& 0  \end{bmatrix}.
\end{equation*}
\end{proof}
\begin{lemma}\label{1}
Let  $n$ be an element in $ \mathbb{N} $  with $n \geq 2 $. If $ B \in M_n(\mathfrak{P})$ is given by
 $$ B = \begin{bmatrix} b & & & &  y \\
                        1 & b & & & & \\
                          & 1 & b & & & \\
                          & & \ddots & \ddots & & \\
                         & & &1& b   \end{bmatrix} , $$

                        then the matrix $B$ is equivalent to the matrix
\begin{equation*}
 \begin{bmatrix} 1 & & & & &   \\
                          & 1 & & & & \\
                          &   &1 & & & \\
                          & &   & \ddots & & \\
                         & & & & y+(-1)^{n+1}b^n   \end{bmatrix}.
\end{equation*}
\end{lemma}

\begin{proof}
Let $B$ be a matrix in $M_n(\mathfrak{P})$ with $n \geq 2$ . The result is obvious when   $n=2$. Now assume that $n > 2$ and let $\hat{B}$ be the $(n-1) \times (n-1)$ matrix over $\mathfrak{P}$ obtained from $B$ by deleting the last row and last column of $B$
\begin{equation*}
 \hat{B}= \begin{bmatrix} b & & & &  \\
                        1 & b & & & & \\
                          & 1 & b & & & \\
                          & & \ddots & \ddots & & \\
                         & & &1& b   \end{bmatrix}.
\end{equation*}
Then $B$ has the following form
\begin{equation*}
B =\left[
\begin{array}{cccc|c}
 & & & & y \\  & \hat{B}& & \\ &&&&0 \\\hline
  0& \ldots &0& 1 & b\\
\end{array}
\right].
\end{equation*}

Now use \ref{L.16} and appropriate row and column operations to get the result.
\end{proof}

\begin{corollary}\label{L.20}
Let $n $  be a positive integer such that $ n \geq 3 $ ,  $ 1 \leq k \leq n-1 $ and let $m= n-k$. Suppose that $u$ and $ v$ are two variables on $\mathfrak{P}$ and let $A_1^{(k)} \in M_k(\mathfrak{P})$ and $A_2^{(k)} \in M_m(\mathfrak{P})$ be given by $$A_1^{(k)} = \begin{bmatrix} b & & & & &  \\
                        1 & b & & & & \\
                          & 1 & b & & & \\
                          & & \ddots & \ddots & & \\
                         & & &1& b   \end{bmatrix}  \text{ and  }A_2^{(k)} = \begin{bmatrix} b & & & & &  \\
                        1 & b & & & & \\
                          & 1 & b & & & \\
                          & & \ddots & \ddots & & \\
                         & & &1& b   \end{bmatrix} . $$

 If $B_k = \left[ \begin{array}{c|c} A_1^{(k)} & L_{1,m}^{k \times m}(v) \\ \hline  L_{1,k
 }^{m \times k}(u) & A_2^{(k)} \end{array}\right]  =
 \left[ \begin{array}{cccc|cccc} b & & & & & & &  v  \\
                                    1 & b & & & & & &   \\
                                      &  \ddots & \ddots &  & & & &  \\
                                      & & 1& b & & & &  \\ \hline
                                      & & &u &    b & & &  \\
                                      & & & &  1 & b & &  \\
                                      & & & &   &  \ddots & \ddots &  \\
                                      & & & &   & &1 & b
                                        \end{array}\right]  $,    then $ B_k$ is equivalent to the matrix $C_k= I_{n-2} \oplus \left[
                                                                                                                                  \begin{array}{cc}
                                                                                                                                    (-1)^{k+1}b^k & v \\
                                                                                                                                    u & (-1)^{m+1}b^m \\
                                                                                                                                  \end{array}
                                                                                                                                \right]\in M_n(\mathfrak{P})$ where $I_{n-2}$ is the identity matrix in $M_{n-2}(\mathfrak{P})$.

Moreover, if $D \in M_n(\mathfrak{P})$ is given by $ D = \begin{bmatrix} b & & &  &uv  \\
                        1 & b & & & \\
                          & 1 & b &  & \\
                          & & \ddots & \ddots &  \\
                        & & &1& b   \end{bmatrix}$ , then $D$ is equivalent to the matrix $\tilde{D}= I_{n-2}\oplus \left[
                                                                                                                      \begin{array}{cc}
                                                                                                                        (-1)^{n}b^{n-1} & uv \\
                                                                                                                        1 & b \\
                                                                                                                      \end{array}
                                                                                                                    \right]\in M_n(\mathfrak{P})$  where $I_{n-2}$ is the identity matrix in $M_{n-2}(\mathfrak{P})$.
\end{corollary}

\begin{proof}
By Lemma \ref{L.16}, there exist upper triangular matrices $B_1, C_1 \in M_k(\mathfrak{P})$ and $B_2, C_2 \in M_{n-k}(\mathfrak{P})$ with $1$ along their diagonal such that
$$ B_1 A_1^{(k)}C_1 = \left[ \begin{array}{cccc} 0 & & & (-1)^{k+1}b^k \\
                     1 & 0 & &  \\
                      & \ddots & \ddots &  \\
                      & &1 & 0 \end{array} \right] ,
 B_2 A_2^{(k)}C_2 = \left[ \begin{array}{cccc} 0 & & & (-1)^{m+1}b^m \\
                     1 & 0 & &  \\
                      & \ddots & \ddots &  \\
                      & &1 & 0 \end{array} \right] $$
where $m=n-k$.
Define $B,C \in M_n(\mathfrak{P})$ to be
$B = \begin{bmatrix} B_1 & 0 \\ 0 & B_2 \end{bmatrix} $ and $C = \begin{bmatrix} C_1 & 0 \\ 0 & C_2 \end{bmatrix} $.

Therefore,
\begin{align*}
                             B B_kC & = \left[ \begin{array}{cc} B_1 A_1^{(k)}C_1 & L_{1,m}^{k \times m}(v) \\  L_{1,k}^{m \times k}(u) & B_2 A_2^{(k)}C_2 \end{array}\right]  \\
                              & = \left[ \begin{array}{cccc|cccc} 0 & & &(-1)^{k+1}b^k & & & &  v  \\
                                    1 & 0 & & & & & &   \\
                                      &  \ddots & \ddots &  & & & &  \\
                                      & & 1&0 & & & &  \\ \hline
                                      & & &u &    0 & & & (-1)^{m+1}b^m \\
                                      & & & &  1 &0 & &  \\
                                      & & & &   &  \ddots & \ddots &  \\
                                      & & & &   & &1 & 0
                                        \end{array}\right].
                           \end{align*}
                           Switching columns and rows of $ R B_kC $ yields  the desired equivalent matrix.

Now by induction on $n \geq 3$ we prove the result related to $D$.

If $D= \left[
         \begin{array}{ccc}
           b & 0 & uv \\
          1 & b & 0 \\
           0 & 1 & b \\
         \end{array}
       \right]$, we get $E_{1,2}(-b)DE_{1,2}(-b)=\left[
                                                           \begin{array}{ccc}
                                                             0 & -b^2 & uv \\
                                                             1 & 0 & 0 \\
                                                             0 & 1 & b \\
                                                           \end{array}
                                                         \right].$
Switch the rows to get the desired result. Now assume that  $D$ is  $(n+1) \times (n+1)$ matrix. Then $D$ can be written as

\begin{equation*}
D=
\left[
\begin{array}{c|c}
   & uv \\ \hat{D} & \vdots \\ & 0 \\ \hline
  0 \ldots 1 & b\\
\end{array}
\right]
 \end{equation*}

where $\hat{D}$ is the $n\times n$  matrix over $\mathfrak{P}$ that is given by

\begin{equation*}
 \hat{D}= \begin{bmatrix} b & & & &  \\
                        1 & b & & & & \\
                          & 1 & b & & & \\
                          & & \ddots & \ddots & & \\
                         & & &1& b   \end{bmatrix}.
\end{equation*}

By Lemma \ref{L.16},  there exist upper triangular matrices $\hat{B}, \hat{C }\in M_n(\mathfrak{P})$ such that the $(i,i)$ entries of $\hat{B}$ and $\hat{C}$ are the identity element of $\mathfrak{P}$ for all $i=1, ..., n $ and
                        \[ \hat{B}\hat{D}\hat{C} = \left[\begin{array}{ccccc} 0 & & & &  (-1)^{n+1}b^n \\
                        1 & 0 & & &  \\
                          & 1 &0 & &  \\
                         &  & \ddots &\ddots & \\
                         & & &1& 0 \end{array} \right]. \]

Let
\begin{equation*}
M=\left[
\begin{array}{cccc|c}
 & & & & 0 \\ & \hat{B} & & &\vdots \\ &&&&0 \\\hline
  0& \ldots &0& 0 & 1\\
\end{array}
\right]
\end{equation*}
and
\begin{equation*}
 N= \left[
\begin{array}{cccc|c}
 & & & & 0 \\ &\hat{C} & & &\vdots \\ &&&&0 \\\hline
  0& \ldots &0& 0 & 1\\
\end{array}
\right].
\end{equation*}

As a result, it follows that

\begin{equation*}
  MDN= \left[\begin{array}{ccccc|c} 0 & & & &  (-1)^{n+1}b^n &uv \\
                        1 & 0 & & & &  \\
                          & 1 &0 & & & \\
                         &  & \ddots &\ddots & &\\
                         & & &1& 0 & 0 \\ \hline
                      0 & & &0 & 1 & b \end{array} \right].
\end{equation*}

 Switching the columns  of $MDN $ yields  the desired equivalent matrix.
\end{proof}
%%%%%%%%%%%%%%%%%%%%%%%%%%%%%%%%%%%%%%%%%%%%%%%%%%%%%%%%%%%%%%

%%CHAPTER3%%%%%%%%%%%%%%%%%%%%%%%%%%%%%%%%%%%%%%%%%%%%%%%%%
%%%%%%%%%%%%%%%%%%%%%%%%%%%%%%%%%%%%%%%%%%%%%%%%%%%%%%%%%%%
%%%%%%%%%%%%%%%%%%%%%%%%%%%%%%%%%%%%%%%%%%%%%%%%%%%%%%%%%%%%
\chapter{Matrix Factorization}
\label{chapter:Matrix Factorization}
   In this chapter, we discus the concept of a matrix factorization and their basic properties  needed later in the rest of this thesis.

Matrix factorizations were introduced by David Eisenbud in \cite{ED} who  proved that the MCM modules over hypersurfaces have a periodic resolutions.

\section{Definitions and Properties}
\label{section:Definitions and Properties}
\begin{definition}\cite[Definition 1.2.1]{DI}
\emph{Let $f$ be a nonzero element of a  ring $S$. A matrix factorization of $f$ is a pair $(\phi,\psi)$ of homomorphisms
between finitely generated free $S$-modules  $ \phi: G \rightarrow F$  and $\psi: F \rightarrow G $, such
that $\psi\phi = f I_G$ and $\phi\psi = f I_F$}.
\end{definition}

\begin{remark}\label{C3.2}
Let $f$ be a nonzero element of a commutative ring $S$. If $( \phi: G \rightarrow F,\psi: F \rightarrow G )$ is a matrix factorization of $f$, then:
\begin{enumerate}
  \item [(a)] $f\Cok( \phi )= f \Cok( \psi )=0$.
  \item [(b)] If $f$ is a  non-zerodivisor, then $\phi$ and $\psi$ are injective.
  \item [(c)] If $S$ is a domain, then $G$ and $F$ are finitely generated free modules having the same rank.
\end{enumerate}
\end{remark}
\begin{proof}
(a) Since $\psi\phi = f I_G$ and $\phi\psi = f I_F$, it follows that $fG \subseteq \Ima(\psi)$ and $fF\subseteq \Ima( \phi)$ which proves the result. \\
(b) Let $x \in G$  such that $\phi(x)=0$. Thus $ \psi(\phi(x))=0$ and hence $fx=0$. Since $f$ is a  non-zerodivisor, it follows that $x=0$ and hence $\phi$ is injective. By similar argument, we prove that $\psi$ is injective.\\
(c) If  $M= \Cok( \phi )$, then the following short sequence is exact.
\begin{equation}\label{E111}
  0 \longrightarrow G\xrightarrow{\phi}F  \rightarrow  M \longrightarrow 0
\end{equation}
Recall that the rank of the  finitely generated module $M$ over the domain $S$, denoted  $rank_SM$,  is the dimension of the vector space $K\otimes_SM$ over $K$ \cite[Section 11.6]{E} where $K$ is the quotient field of the integral domain $S$.  We know  from Proposition \ref{LLL2} (d) that $K\otimes_SM $ is isomorphic to $W^{-1}M$ as $S$-module where $W=S\setminus\{0\}$. Since $fM=0$, it follows that $K\otimes_SM=0$. Therefore, tensoring the short exact sequence (\ref{E111}) with $K$ over $S$ yields that $K\otimes_SG \simeq K\otimes_SF$ and thus $G$ and $F$ have the same rank as free $S$-modules.
\end{proof}

As a result, we can define the matrix factorization of a nonzero element $f$ in a domain as the following.

\begin{definition}\label{D23}
\emph{Let $S$ be a domain and let $f\in S$ be a nonzero element. A matrix factorization (of size $n$) is a pair $(\phi,\psi)$  of $n \times n$  matrices with
coefficients in $S$ such that $\psi\phi = \phi\psi = fI_n$ where $I_n$ is the identity matrix in $M_n(S)$. By $\Cok_S(\phi,\psi)$ and $\Cok_S(\psi,\phi)$, we mean $\Cok_S(\phi)$ and $\Cok_S(\psi)$ respectively. There are two distinguished trivial matrix factorizations of any element
$f$, namely $( f ,1)$ and $(1, f )$. Note that $\Cok_S(1, f ) = 0$, while $\Cok_S(f ,1)=S/fS$. Two matrix factorizations $(\phi,\psi)$ and $(\alpha,\beta)$  of $f$ are said to be equivalent (and we write $(\phi,\psi)\sim(\alpha,\beta)$) if $\phi,\psi,\alpha,\beta \in M_n(S)$ for some positive integer $n$ and there exist invertible matrices $V,W \in M_n(S)$ such that $V\phi = \alpha W$ and $ W \psi = \beta V$. If $(S,\mathfrak{m})$ is a local domain,    a matrix factorization $(\phi,\psi)$ of an element   $f \in \mathfrak{m}\setminus\{0\}$ is reduced if all entries of $\phi$ and $\psi$ are in $\mathfrak{m}$.}
\end{definition}
%\begin{definition}
%Let $(S,\mathfrak{m})$ be a regular local ring and $f \in \mathfrak{m}-\{0\}$.

 %A matrix factorization of $f$ is a pair $(\phi,\psi)$ of homomorphisms
%between free $S$-modules of the same rank, $ \phi: G \rightarrow F$  and $\psi: F \rightarrow G $, such
%that
%$\psi\phi = f I_G$ and $\phi\psi = f I_F$ .
%Equivalently (after choosing bases), $\phi$ and $\psi$ are  matrices in $M_n(S)$ such that
%$\psi\phi = \phi\psi = fI_n$ . By $\Cok_S(\phi,\psi)$ and $\Cok_S(\psi,\phi)$, we mean $\Cok_S(\phi)$ and $\Cok_S(\psi)$ respectively. There are two distinguished trivial matrix factorizations of any element
%$f$, namely $( f ,1)$ and $(1, f )$. Note that $cok_S(1, f ) = 0$, while $cok_S(f ,1)=S/fS$. A matrix factorization $(\phi,\psi)$ of $f$ is reduced if all entries of $\phi$ and $\psi$ are in $\mathfrak{m}$.  Two matrix factorizations $(\phi,\psi)$ and $(\alpha,\beta)$  of $f$ are said to be equivalent and we write $(\phi,\psi)\sim(\alpha,\beta)$ if $\phi,\psi,\alpha,\beta \in M_n(S)$ for some positive integer $n$ and there exist invertible matrices $V,W \in M_n(S)$ such that $V\phi = \alpha W$ and $ W \psi = \beta V$.

%\end{definition}
We can  notice the following remark:
\begin{remark}\label{R3.8}
Let $(S,\mathfrak{m})$  be a local domain,   $f \in \mathfrak{m} \setminus \{0\}$, $R=S/fS$ and let $ u$, $v$ and $z$  be  variables on $S$.
Suppose that $(\phi,\psi)$ and $(\alpha,\beta)$ are two $n \times n$ matrix factorizations of $f$. Then
\begin{enumerate}
\item[(a)] $\Cok_S(\phi)$ and $\Cok_S(\psi)$ are both modules over the ring $R$ as $f\Cok_S(\phi)=0$ and $f\Cok_S(\psi) =0$.
\item[(b)]  The $S$-linear maps $ \phi: S^n \rightarrow S^n$  and $\psi: S^n \rightarrow S^n $ are both injective.
\item[(c)]  If $(\phi,\psi)\sim(\alpha,\beta)$, then $\Cok_S(\phi,\psi)$ is isomorphic to $\Cok_S(\alpha, \beta)$ over $S$ (and consequently over $R$), likewise,
$\Cok_S(\psi, \phi)$ is isomorphic to $\Cok_S(\beta, \alpha)$ over $S$ (and consequently over $R$).
\item[(d)]  If $(\phi,\psi)$ and $(\alpha,\beta)$ are reduced matrix factorizations of $f$ such that $\Cok_S(\phi,\psi)$ is isomorphic to $\Cok_S(\alpha,\beta)$,
then   $(\phi,\psi)\sim(\alpha,\beta)$.

\item[(e)]  We define $(\phi,\psi)\oplus(\alpha,\beta):=(\phi \oplus \alpha, \psi \oplus \beta)$ and hence $(\phi,\psi)\oplus(\alpha,\beta)$ is a matrix factorization of $f$.
\item[(f)] We define $ (\phi, \psi)^{\maltese}:= ( \begin{bmatrix}  \phi & -vI \\  uI & \psi \end{bmatrix} , \begin{bmatrix}  \psi & vI \\  -uI & \phi \end{bmatrix})$ and hence $(\phi, \psi)^{\maltese}$
  is a matrix factorization for $f+uv$ in $S[\![u,v]\!]$ (and in $S[u,v]$). Furthermore, if $(\phi,\psi)\sim(\alpha,\beta)$, then  $(\phi,\psi)^{\maltese}\sim(\alpha,\beta)^{\maltese}$.
\item[(g)]  $[(\phi, \psi)\oplus(\alpha,\beta)]^{\maltese}$ is equivalent to $ (\phi, \psi)^{\maltese}\oplus(\alpha,\beta)^{\maltese}$.
% \item[(h)]  $ \Cok_{S[[u,v]]}[(\phi, \psi)\oplus(\alpha,\beta)]^{\maltese}$ is isomorphic to $\Cok_{S[[u,v]]}(\phi, \psi)^{\maltese}\oplus \Cok_{S[[u,v]]}(\alpha, \beta)^{\maltese} $.
\item[(h)]  We define $ [\Cok_S(\phi, \psi)]^{\maltese}=\Cok_{S[[u,v]]}(\phi, \psi)^{\maltese}$ and hence if $(\phi_j,\psi_j)$ is a matrix factorization of $f$ for all $1 \leq j \leq n$, then   $$[ \bigoplus\limits_{j=1}^{n}\Cok_S(\phi_j,\psi_j)]^{\maltese} =\bigoplus \limits_{j=1}^{n}\Cok_{S[\![u,v]\!]}(\phi_j,\psi_j)^{\maltese}.$$
\item[(i)]  If $R^{\bigstar}=S[\![u,v]\!]/(f+uv)$, then   $\Cok_{S[\![u,v]\!]}(f, 1)^{\maltese}=R^{\bigstar}= \Cok_{S[\![u,v]\!]}(1, f)^{\maltese}$  and hence we can write $(R)^{\maltese}= R^{\bigstar}$ as $R=\Cok_S(f,1)$.
\item[(j)]   We define $ (\phi, \psi)^{\sharp}:= ( \begin{bmatrix}  \phi & -zI \\  zI & \psi \end{bmatrix} ,
\begin{bmatrix}  \psi & zI \\  -zI & \phi \end{bmatrix})$, and hence $ (\phi, \psi)^{\sharp}$ is a matrix factorization of $f+z^2$ in $S[\![z]\!]$ (and in $S[z]$).
\item[(k)] If $(\phi,\psi)\sim(\alpha,\beta)$, then  $(\phi,\psi)^{\sharp}\sim(\alpha,\beta)^{\sharp}$.
\item[(l)] $ [(\phi, \psi)\oplus(\alpha,\beta)]^{\sharp}$ is equivalent to $ (\phi, \psi)^{\sharp}\oplus(\alpha,\beta)^{\sharp}$.
\item[(m)]  If $R^{\sharp}=S[\![z]\!]/(f+z^2)$, then $$R^{\sharp}=S[\![z]\!]/(f+z^2)= \Cok_{S[\![z]\!]}(f, 1)^{\sharp} = \Cok_{S[\![z]\!]}(1,f)^{\sharp}.$$
\end{enumerate}
\end{remark}

\begin{proof}
We will just prove  the  result (d) and (f) as the rest follows from  the   definitions. \\
 (d) Assume that  $(\phi,\psi)$ and $(\alpha,\beta)$ are reduced matrix factorizations of $f$ such that $\Cok_S(\phi,\psi)$ is isomorphic to $\Cok_S(\alpha,\beta)$. This means that $\Cok_S(\phi)$ is isomorphic to $\Cok_S(\alpha)$ and consequently from Proposition \ref{L3.5} it follows that $V\phi = \alpha W$ for invertible matrices $V$ and $W$. Therefore, we have $V\phi \psi = \alpha W \psi$ and thus $V (fI) = \alpha W \psi$. As a result,  we get $\beta V (fI) = \beta \alpha W \psi = f I W \psi$ and consequently $ f \beta V= f W \psi$. Since $f$ is an element of the integral domain $S$, we conclude that $  \beta V=  W \psi$. This proves that $(\phi,\psi)$ is equivalent to $(\alpha,\beta)$.

(f)  Assume that  $V\phi = \alpha W$ and $ W \psi = \beta V$ for invertible matrices $V$ and $W$. It follows that
\begin{equation*}
   \left[
     \begin{array}{cc}
       V &   \\
         & W \\
     \end{array}
   \right] \left[
             \begin{array}{cc}
               \phi & -vI \\
               uI & \psi \\
             \end{array}
           \right] = \left[
                       \begin{array}{cc}
                         \alpha & -vI \\
                         uI & \beta \\
                       \end{array}
                     \right]\left[
                              \begin{array}{cc}
                                W &   \\
                                 & V \\
                              \end{array}
                            \right]
\end{equation*}

and
\begin{equation*}
   \left[
     \begin{array}{cc}
      W &   \\
         & V \\
     \end{array}
   \right] \left[
             \begin{array}{cc}
               \psi &  vI \\
               -uI & \phi \\
             \end{array}
           \right] = \left[
                       \begin{array}{cc}
                          \beta &  vI \\
                         -uI & \alpha \\
                       \end{array}
                     \right]\left[
                              \begin{array}{cc}
                                V &   \\
                                 & W \\
                              \end{array}
                            \right].
\end{equation*}
This proves that the matrix factorization $ (  \left[
             \begin{array}{cc}
               \phi & -vI \\
               uI & \psi \\
             \end{array}  \right], \left[
             \begin{array}{cc}
               \psi &  vI \\
               -uI & \phi \\
             \end{array}
           \right] )$ is equivalent to the matrix factorization $$ (  \left[
             \begin{array}{cc}
               \alpha & -vI \\
               uI & \beta \\
             \end{array} \right] , \left[
             \begin{array}{cc}
               \beta &  vI \\
               -uI & \alpha \\
             \end{array}
           \right] )$$ and therefore $(\phi,\psi)^{\maltese}\sim(\alpha,\beta)^{\maltese}.$
\end{proof}

\begin{notation}
 \emph{ If  $(\phi,\psi)$ is a matrix factorization, we write $$(\phi,\psi)^n=\underbrace{(\phi,\psi)\oplus...\oplus(\phi,\psi)}_{\text{$n$ times}}.$$}
\end{notation}

\begin{definition}
  \emph{Let $f$ be a nonzero element of a domain. A matrix factorization $(\phi,\psi)$ of $f$ is called trivial if it is equivalent  to one of the following forms:
  \begin{equation*}
    (f,1)^n, (1,f)^n, \text{  or  } (f,1)^r \oplus (1,f)^t
  \end{equation*}
where $n$ is the size of $(\phi,\psi)$  and $0 < r,s < n$ with $ r+s=n$.}
\end{definition}
\section{Matrix factorization and Maximal Cohen Maculay modules}
\label{section:Matrix factorization and Maximal Cohen Maculay modules}

The importance of the concept of the matrix factorization in the subject of MCM modules appears clearly in the following proposition.
\begin{proposition}\label{P25}  \cite[Proposition 8.3]{CMR}

Let $(S,\mathfrak{m})$ be a regular local ring and let $f$ be a non-zero
element of $\mathfrak{m}$ and $R=S/fS$ .
\begin{enumerate}
\item[(a)] For every MCM $R$-module $M$, there is a matrix factorization $(\phi , \psi)$ of $f$
with $\Cok\phi \cong M$. \\
\item[(b)] If $(\phi,\psi)$ is a matrix factorization of $f$, then $\Cok\phi$ and $\Cok\psi$ are MCM
$R$-modules.
\end{enumerate}
\end{proposition}

\begin{proof}
 Let $\dim R= d$. \\
(a) By   \cite[Corollary 11.18]{AM} it follows that $\dim S=d+1$. Notice that $M$ can be viewed as   $S$-module such that every $M$-regular sequence on $R$ is $M$-regular sequence on $S$ and every $M$-regular sequence on $S$ is also $M$-regular sequence on $R$. As a result, $\depth_SM=\depth_RM$. Since $M$ is finitely generated module over the regular local ring $S$, it follows from Auslander-Buchsbaum-Serre Theorem \cite[Theorem2.2.7]{BH} that $\pd_SM < \infty $. Auslander-Buchsbaum formula  \ref{P2.23} implies  that
\begin{equation*}
 \pd_SM=\depth S-\depth_SM=\dim S-\depth_RM=\dim S-\dim R=1.
\end{equation*}

Therefore, there exist projective $S$-modules $F$ and $G$, and consequently they are free $S$-modules \cite[Proposition 18.4.1]{BCA},  such that the following sequence is exact

\begin{equation*}
  0 \longrightarrow G\xrightarrow{\phi}F \xrightarrow{\lambda}  M \longrightarrow 0.
\end{equation*}
 For each $x \in F$, we notice that $\lambda(fx)=0$ and consequently $fF \subseteq \Ker \lambda = \Ima \phi$. As $\phi$ is injective,  for each $x \in F$ there exists a unique element $y \in G$ such that $ \phi( y )=fx$. This enables us to define a map $ \psi: F \rightarrow G$ via $ \psi(x)= y$ if and only if $ \phi( y )=fx$. It is easy to check that $\psi$ is an injective homomorphism over $S$ satisfying that $\psi \phi =  fI_G$ and $\psi\phi=fI_F$ where $I_F$ and $I_G$ denote the identity maps on $F$ and $G$ respectively. Therefore, $M=\Cok(\phi,\psi)$. \\
(b) If $(\phi,\psi)$ is a matrix factorization of $f$, then there exist two free $S$-modules $F$ and $G$ having the same rank and making the following short sequences exact over $S$

\begin{equation*}
  0 \longrightarrow G\xrightarrow{\phi}F  \rightarrow  \Cok \phi \longrightarrow 0
\end{equation*}

and
\begin{equation*}
  0 \longrightarrow G\xrightarrow{\psi}F  \rightarrow  \Cok \psi \longrightarrow 0.
\end{equation*}

This means that $\pd_S ( \Cok \phi )= 1 = \pd_S ( \Cok \psi )$. Since $\depth_S S=\dim S $, it follows from Auslander-Buchsbaum formula \ref{P2.23} that
\begin{equation*}
\depth_R \Cok \phi= \depth_S \Cok \phi = \depth_SS-\pd_S ( \Cok \phi )=\dim S-1=\dim R
\end{equation*}
and
\begin{equation*}
\depth_R \Cok \psi= \depth_S \Cok \psi = \depth_SS-\pd_S ( \Cok \psi )=\dim S-1=\dim R.
\end{equation*}

This proves that $\Cok \phi$  and $\Cok \psi$ are MCM $R$-modules.
\end{proof}

\begin{definition} \emph{Let $R$ be a ring, a non-zero $R$-module  $M$ is called a stable $R$-module if   $M$  does not have
a direct summand isomorphic to $R$.}
\end{definition}

D.  Eisenbud has established a relationship between  reduced matrix factorizations and stable  MCM modules as follows.
\begin{proposition}\label{A}\cite[ Corollary 7.6]{YY} \cite[Theorem 8.7]{CMR}
Let $(S,\mathfrak{m})$ be a regular local ring and let $f$ be a non-zero
element of $\mathfrak{m}$ and $R=S/fS$. Then
the association $ (\phi, \psi)\mapsto \Cok(\phi, \psi)$   yields a bijective correspondence between the
set of equivalence classes of reduced matrix factorizations of f and the set of isomorphism
classes of stable MCM modules over R.
\end{proposition}

\begin{remark}
  Let $(S,\mathfrak{m})$ be a regular local ring. If $f$ is  a non-zero
element of $\mathfrak{m}$, let   $R=S/fS$ and $R^{\bigstar}:=S[\![u,v]\!]/(f+uv)$. If $M$ is a stable MCM $R$-module, then $M=\Cok_S(\phi,\psi)$ where $(\phi,\psi)$ is a reduced matrix factorization of $f$. This enables us to define $M^{\maltese}=\Cok_{S[\![u,v]\!]}(\phi,\psi)^{\maltese}$. Indeed, if $M=\Cok_S(\alpha,\beta)$ for some matrix factorization $(\alpha,\beta)$ of $f$, then by Proposition \ref{A} $(\alpha,\beta)$ is reduced matrix factorization of $f$ and $(\alpha,\beta) \sim (\phi,\psi)$. According to Remark \ref{R3.8} (f), $(\alpha,\beta)^{\maltese} \sim (\phi,\psi)^{\maltese}$. This shows that $\Cok_{S[\![u,v]\!]}(\alpha,\beta)^{\maltese} \cong \Cok_{S[\![u,v]\!]}(\phi,\psi)^{\maltese}$. Therefore, the association $M\rightarrow M^{\maltese}$ is well defined on the class of stable MCM $R$-modules.
\end{remark}

If $R$ and $R^{\bigstar}$ are  as in Remark \ref{R3.8} , the indecomposable non-free MCM modules over $R$ and $R^{\bigstar}$ can be  related in the following situation.

\begin{proposition}\label{P28}  \cite[Theorem 8.30]{CMR}
Let $(S, \mathfrak{m}, K)$ be a complete regular local ring such that $K$ is algebraically closed of characteristic not $2$ and $f\in \mathfrak{m}^2 \smallsetminus \{0\}$. If $R=S/fS$ and $R^{\bigstar}:=S[\![u,v]\!]/(f+uv)$, then the association $M\rightarrow M^{\maltese}$ defines a bijection between the isomorphisms classes of indecomposable non-free MCM modules over $R$ and $R^{\bigstar}$.
\end{proposition}
Any matrix factorization has a decomposition as follows.
\begin{proposition}\label{P.24}(cf.\cite[Result 7.5.2]{YY})
Let $(S,\mathfrak{m})$ be a  local domain and  $f \in \mathfrak{m} \smallsetminus  \{0\}$.
If  $(\phi,\psi)$ is   a nontrivial   matrix factorization  of $f$  of size $n$, then $(\phi,\psi)$ can be written  as
\begin{equation*}\label{E4}
  (\phi,\psi) = (\alpha,\beta)\oplus (f,1)^t\oplus (1,f)^r
\end{equation*}
where $(\alpha,\beta)$ is a reduced matrix factorization of $f$ and  $0\leq t, r < n$. Furthermore, if $(S,\mathfrak{m})$ is a regular local ring, the above decomposition is unique up to equivalence.
\end{proposition}
\begin{proof}
  By the induction on the size of the matrix factorization we will prove the result.
The case when the size is one is obvious.  Suppose that $(\phi,\psi)$  is a matrix factorization of $f$ of size $(n+1) $. If $(\phi,\psi)$  is reduced, we are done. Without lose of generality, assume that one entry of $\phi$ is a unite. Using row and column operations, there exist  invertible matrices $U,V$ in $M_n(S)$ such that

\begin{equation*}
  U \phi V= \left[
              \begin{array}{cc}
                1 &  \\
                 & \tilde{\phi} \\
              \end{array}
            \right]
\end{equation*}
where $\tilde{\phi}$ is $n \times n$ matrix. Set $\hat{\psi}= V^{-1} \psi U^{-1}$ and notice that $ \hat{\psi}\left[
              \begin{array}{cc}
                1 &  \\
                 & \tilde{\phi} \\
              \end{array}
            \right]=fI$ and $ \left[
              \begin{array}{cc}
                1 &  \\
                 & \tilde{\phi} \\
              \end{array}
            \right]\hat{\psi}=fI$. This makes  $\hat{\psi}= \left[ \begin{array}{cc}
                f &  \\
                 & \tilde{\psi} \\
              \end{array}\right]$ where $(\tilde{\phi},\tilde{\psi})$ is a matrix factorization of $f$. Therefore, $(\phi,\psi) \sim (\tilde{\phi},\tilde{\psi}) \oplus (1,f)$. If $(\tilde{\phi},\tilde{\psi})$ is reduced, we get the desired result. Otherwise, apply the induction hypothesis on $(\tilde{\phi},\tilde{\psi})$ to completes the proof.

Now assume that $(S,\mathfrak{m})$ is  a regular local ring, $R=S/fS$, and   let $(\phi,\psi)$ be a nontrivial   matrix factorization  of $f$  of size $n$. Suppose
 that  $(\phi,\psi)$ can be written as  $ (\alpha_j,\beta_j)\oplus (f,1)^{t_j}\oplus (1,f)^{r_j}$ where $(\alpha_j,\beta_j)$ is a reduced matrix factorization of $f$ for $j=1,2$.   This makes $\widehat{(M_1)}_{\mathfrak{m}}\oplus (\widehat{R}_{\mathfrak{m}})^{\oplus t_1}=\widehat{(M_2)}_{\mathfrak{m}}\oplus (\widehat{R}_{\mathfrak{m}})^{\oplus t_2}$ where $M_j=\Cok_j(\alpha_j,\beta_j)$ for $j=1,2$. Since $M_j$ has no free direct summands (Proposition \ref{A}), it follows from Proposition \ref{Pro2.22}(a)  that $\widehat{(M_j)}_{\mathfrak{m}}$ has  no free direct summands where $j=1,2$. Therefore,  by  Krull-Remak-Schmidt theorem (see discussion \ref{disc2.33})  $t_1=t_2$. Since $(\psi,\phi)$ can be written as  $ (\beta_j,\alpha_j) \oplus (1,f)^{r_j}\oplus (f,1)^{t_j}$, it follows from a similar argument that $r_1=r_2$. Furthermore, Krull-Remak-Schmidt theorem (see discussion \ref{disc2.33}) implies that $\widehat{(M_1)}_{\mathfrak{m}} \cong \widehat{(M_2)}_{\mathfrak{m}}$ and consequently
\begin{equation*}
 \Cok_S(\alpha_1,\beta_1) =M_1 \cong M_2= \Cok_S(\alpha_2,\beta_2) \text{ (see Proposition \ref{Pro2.22} (b))}.
\end{equation*}
 Since $(\alpha_1,\beta_1)$  and $(\alpha_2,\beta_2)$ are reduced   with $\Cok_S(\alpha_1,\beta_1)  \cong \Cok_S(\alpha_2,\beta_2)$,  it follows from Remark \ref{R3.8} (d) that $(\alpha_1,\beta_1)$  and $(\alpha_2,\beta_2)$ are equivalent. This finishes the proof.
\end{proof}

One can use Proposition \ref{A} and  Remark \ref{R3.8} to show the following corollary.
\begin{corollary}\label{C3.8}
Let $(S,\mathfrak{m})$ be a regular local ring ,  $f \in\mathfrak{ m} \smallsetminus \{0\}$,  $R=S/fS$, $R^{\bigstar}=S[\![u,v]\!]/(f+uv)$ , and $R^{\sharp}=S[\![z]\!]/(f+z^2)$ where $u,v$ and $z$ are  variables over $S$ .
If  $(\phi,\psi)$ is a matrix factorization of $f$ having the decomposition
\begin{equation*}\label{E44}
  (\phi,\psi) = (\alpha,\beta)\oplus (f,1)^t \oplus (1,f)^r
\end{equation*}
where  $(\alpha,\beta)$ is  reduced and $t, r$ are non-negative integers, then :
\begin{enumerate}
\item[(a)] $\sharp (\Cok_S(\phi,\psi),R)=t$ and $\sharp (\Cok_S(\psi,\phi),R)=r$.
\item[(b)] $\sharp (\Cok_{S[\![u,v]\!]}(\phi,\psi)^{\maltese},R^{\bigstar})=\sharp (\Cok_S(\phi,\psi),R)+\sharp (\Cok_S(\psi,\phi),R)$.
\item[(c)] $\sharp (\Cok_{S[\![z]\!]}(\phi,\psi)^{\sharp},R^{\sharp})=\sharp (\Cok_S(\phi,\psi),R)+\sharp (\Cok_S(\psi,\phi),R)$.
\end{enumerate}
\end{corollary}

\begin{proof}
(a) It is obvious that $ \Cok( f,1)=R$ but  $\Cok(1,f)=\{0\}$ and hence $\Cok(\phi,\psi)=\Cok(\alpha,\beta)\oplus R^t$. Since $(\alpha,\beta)$ is  reduced, it follows  from Proposition \ref{A} that $\Cok(\alpha,\beta)$ is stable. This implies by Discussion \ref{disc2.33}(b) that  $\sharp (\Cok(\phi,\psi),R)=t$. Now, since $(\psi,\phi) = (\beta, \alpha)\oplus (f,1)^r \oplus (1,f)^t$, it follows from what we have proved that $\sharp (\Cok(\psi,\phi),R)=r$.\\
(b) Notice by Remark \ref{R3.8} (i)  that $\Cok(f,1)^{\maltese} = R^{\bigstar} =\Cok(1,f)^{\maltese}$. This makes

\begin{equation*}
  \Cok(\phi,\psi)^{\maltese}= \Cok(\alpha,\beta)^{\maltese} \oplus (R^{\bigstar})^t \oplus (R^{\bigstar})^r.
\end{equation*}

Since $(\alpha,\beta)^{\maltese}$ is a reduced  matrix factorization of $f+uv$, it follows  from Proposition \ref{A} that $\Cok(\alpha,\beta)^{\maltese}$ is stable $R^{\bigstar}$-module. This proves, by Discussion \ref{disc2.33}(b), that
$$\sharp (\Cok(\phi,\psi)^{\maltese},R^{\bigstar})=t+r=\sharp (\Cok(\phi,\psi),R)+\sharp (\Cok(\psi,\phi),R).$$
(c) Similar argument to the above argument.
\end{proof}

 Krull-Remak-Schmidt theorem (see discussion \ref{disc2.33}) and Proposition \ref{Pro2.22} enable us to establish the following Proposition.
\begin{proposition}\label{C3.14}
 Let $(S,\mathfrak{m})$ be a regular local ring, $f \in \mathfrak{m} \smallsetminus \{0\}$, and $R=S/fS$. If $(\phi, \psi)$  is a  reduced matrix factorization  of $f$, then $$(\phi, \psi)\sim [(\phi_1,\psi_1)\oplus (\phi_2,\psi_2)\oplus \dots\oplus(\phi_n,\psi_n)]$$ where $(\phi_i,\psi_i)$ is a reduced matrix factorization of $f$ with $\Cok_S(\phi_i,\psi_i)$ is non-free indecomposable MCM $R$-module for all $ 1 \leq i \leq n$. Furthermore,  the above representation of $(\phi, \psi)$  is unique up to equivalence when $S$ is also complete.
\end{proposition}

\begin{proof}
By Proposition  \ref{A},   $M:=\Cok_S(\phi , \psi )$ is stable MCM $R$-module.   We may assume by Discussion \ref{disc2.33}(a) that $M=M_1\oplus \dots \oplus M_n$ where $M_j$ is a non-free indecomposable MCM $R$-module for each $ 1 \leq j \leq n$ (Corollary \ref{C2.30}). Again by Proposition \ref{A}, we have $M_j=\Cok_S(\phi_j,\psi_j)$ for some reduced matrix factorization $(\phi_j,\psi_j)$ of $f$ for all $ 1 \leq j \leq n$. As a result, $\Cok_S(\phi , \psi )$ is isomorphic to  $\Cok_S[(\phi_1,\psi_1)\oplus (\phi_2,\psi_2)\oplus \dots\oplus(\phi_n,\psi_n)]$ and hence  by Remark \ref{R3.8}(d), $(\phi, \psi)\sim [(\phi_1,\psi_1)\oplus (\phi_2,\psi_2)\oplus \dots\oplus(\phi_n,\psi_n)]$. Now if $(\phi, \psi)\sim [(\alpha_1,\beta_1)\oplus (\alpha_2,\beta_2)\oplus \dots\oplus(\alpha_m,\beta_m)]$ is another representation of $(\phi, \psi)$ and $N_j= \Cok_S(\alpha_j,\beta_j)$ where $(\alpha_j,\beta_j)$ is a matrix factorization of $f$  for all $1 \leq j \leq m$, then  $ \bigoplus_{i=1}^n M_i$ is isomorphic to  $ \bigoplus_{j=1}^m N_j$ as $R$-modules. By Krull-Remak-Schmidt theorem (Discussion \ref{disc2.33}(a)), $n=m$ and, after renumbering, $M_i\cong N_i$ for each $i$. Therefore,  by  Proposition \ref{Pro2.22} (b) and hence  by Remark \ref{R3.8}(d) $(\phi_i,\psi_i) \sim (\alpha_i, \beta_i)$ for each $i$.
\end{proof}
\begin{proposition}\label{P3.15}
Let $(S, \mathfrak{m}, K)$ be a complete regular local ring such that $K$ is algebraically closed of characteristic not $2$, and   $f\in \mathfrak{m}^2 \smallsetminus \{0\}$. let $R=S/fS$ and  $R^{\bigstar}:=S[\![u,v]\!]/(f+uv)$. If $(\phi, \psi)$ and $(\alpha, \beta)$ are reduced matrix factorizations of $f$, then $\Cok_S(\phi, \psi)$ is isomorphic to $\Cok_S(\alpha, \beta)$ over $R$ if and only if $\Cok_{S[\![u,v]\!]}(\phi, \psi)^{\maltese}$ is isomorphic to $\Cok_{S[\![u,v]\!]}(\alpha,\beta)^{\maltese}$  over $R^{\bigstar}$.
\end{proposition}
\begin{proof}
If $\Cok_S(\phi, \psi)$ is isomorphic to $\Cok_S(\alpha, \beta)$, by Remark \ref{R3.8} (d), $(\phi, \psi)\sim(\alpha, \beta)$ and consequently $(\phi, \psi)^{\maltese}\sim(\alpha, \beta)^{\maltese}$. This shows that $\Cok_{S[[u,v]]}(\phi, \psi)^{\maltese}$ is isomorphic to $\Cok_{S[[u,v]]}(\alpha,\beta)^{\maltese}$  over $R^{\bigstar}$.

Assume now that $\Cok_{S[\![u,v]\!]}(\phi, \psi)^{\maltese}$ is isomorphic to $\Cok_{S[\![u,v]\!]}(\alpha,\beta)^{\maltese}$  over $R^{\bigstar}$.
Using Proposition \ref{C3.14} we see that  $(\phi, \psi)\sim [(\phi_1,\psi_1)\oplus (\phi_2,\psi_2)\oplus \dots\oplus(\phi_n,\psi_n)]$
 and $(\alpha, \beta)\sim [(\alpha_1,\beta_1)\oplus (\alpha_2,\beta_2)\oplus \dots\oplus(\alpha_t,\beta_t)]$ where $(\phi_i,\psi_i)$
 and $(\alpha_j,\beta_j)$ are reduced  matrix factorizations of $f$ satisfying that $\Cok_S(\phi_i,\psi_i)$ and $\Cok_S(\alpha_j,\beta_j)$ are
 non-free indecomposable MCM $R$-module for all $ i $ and $ j $.
 Remark \ref{R3.8} (f),(g) and (c) imply that
 $$\Cok_{S[\![u,v]\!]}(\phi, \psi)^{\maltese}= \bigoplus _{j=1}^n \Cok_{S[\![u,v]\!]}(\phi_j, \psi_j)^{\maltese} $$
  and  $$\Cok_{S[\![u,v]\!]}(\alpha,\beta)^{\maltese}= \bigoplus_{i=1}^t \Cok_{S[\![u,v]\!]}(\alpha_i,\beta_i)^{\maltese}.$$
  Notice that Proposition \ref{P28} implies that $\Cok_{S[\![u,v]\!]}(\phi_j, \psi_j)^{\maltese}$ and $\Cok_{S[\![u,v]\!]}(\alpha_i,\beta_i)^{\maltese}$ are indecomposable non-free MCM modules over  $R^{\bigstar}$ for all $i$ and $j$. Now  Krull-Remak-Schmidit Theorem (Discussion \ref{disc2.33}(a)) gives that $n=t$ and after renumbering we get $\Cok_{S[\![u,v]\!]}(\phi_i, \psi_i)^{\maltese}$ is isomorphic to $\Cok_{S[\![u,v]\!]}(\alpha_i,\beta_i)^{\maltese}$ over $R^{\bigstar}$  and consequently
  Proposition \ref{P28} implies that  $\Cok_S(\phi_i, \psi_i)$ is isomorphic to $\Cok_S(\alpha_i,\beta_i)$ over $R$ for all $1 \leq i \leq n$. Therefore, by Remark \ref{R3.8}(d) $(\phi_i, \psi_i)\sim (\alpha_i,\beta_i)$. As a result, it follows that
  $ (\phi , \psi ) \sim \bigoplus_{j=1}^n (\phi_j , \psi_j )\sim \bigoplus_{j=1}^n (\alpha_j ,\beta_j ) \sim (\alpha , \beta) $ and hence $\Cok_S(\phi , \psi )$
  is isomorphic to $\Cok_S(\alpha, \beta )$ as desired.
\end{proof}

%%%%%%%%%%%%%%%%%%%%%%%%%%%%%%%%%%%%%%%%%%%%%%%%%%%%%%%%%%%%%%
%%%%%%%%%%%%%%%%%%%%%%%%%%%%%%%%%%%%%%%%%%%%%%%%%%%%%%%%%%%%%%
%%%%%%%%%%%%%%%%%%%%%%%%%%%%%%%%%%%%%%%%%%%%%%%%%%%%%%%%%%%%%%

\chapter{Modules of finite F-representation type }
\label{Chapter: Modules of finite Frepresentation type}
In this chapter, all rings are   Noetherian of prime characteristic $p$  unless otherwise stated.

\section{Definition and examples}
\label{section:Definition and examples}

The notion of finite F-representation type was introduced by  K. Smith
and M. Van den Bergh in \cite{SV} for $F$-finite rings  over which the  Krull-Remak-Schmidit Theorem is satisfied, i.e,  they defined the notion of finite F-representation type for an  $F$-finite ring $R$ with the property that  every  finitely generated $R$-module can be written uniquely up to isomorphism as a direct sum of finitely many indecomposable
$R$-modules. However, Y.Yao in \cite{Y} generalized this notion to be defined for  finitely generated modules over Noetherian rings of prime characteristic $p$. After that,  S. Takagi and R. Takahashi    in \cite{TT} and T.Shibuta in  \cite{TS} studied  this notion under  the general  assumption made by Y.Yao that the ring is just  Noetherian  of prime characteristic $p$. In this thesis, we also adapt the same definition of  this notion under the same general assumptions for the ring as they appear  in \cite[Definition1.1]{Y} and \cite[Definition 2.1]{TS}.
\begin{definition}\label{D1}

%\begin{enumerate}
\emph{Let $R$ be a   ring.
 % \item [(a)]
  If $M$ , $M_1$, ... , $M_s$ are finitely generated $R$-modules, then  $M$ is said to have   finite F-representation type (henceforth abbreviated FFRT) by the $R$-modules $M_1$, ... , $M_s$ if for every positive integer $e$, the $R$-module $F_*^e(M)$ is isomorphic
to a finite direct sum of the $R$-modules $M_1,\ldots ,M_s$, that is, there exist
nonnegative integers $t_{(e,1)}, \ldots , t_{(e,s)}$ such that
 $$F_*^e(M)= \bigoplus_{j=1}^s M_j^{\oplus t_{(e,j)}}.$$
 In particular,  $R$ is said to have finite F-representation type if there exist finitely
generated $R$-modules $M_1,\ldots,M_s$ by which $R$ has finite F-representation
type.}
%  \item [(b)] Let $R = \bigoplus_{n \geq 0}R_n$ be a Noetherian graded ring of prime characteristic $p$. A  finitely generated graded $R$-module $M$ is said to have
% finite graded F-representation type by finitely generated $\mathbb{Q}$-
%graded $R$-modules $M_1, \ldots ,M_s$ if for every positive integer $e $, the $\mathbb{Q}$-graded $R$-module $F_*^e(M)$
%is isomorphic to a finite direct sum of the $\mathbb{Q}$-graded $R$-modules $M_1, \ldots ,M_s$,
%that is, there exist nonnegative integers $t_{ei}$ and rational numbers $\alpha_{ij}^{(e)}$ for all
%$1 \leq i \leq s$ and $1 \leq j \leq t_{ei}$ such that there exists a $\mathbb{Q}$-graded isomorphism of
%the form
%$$F_*^e(M)= \bigoplus_{i=1}^s\bigoplus_{j=1}^{t_{ei}} M_i(\alpha_{ij}^{(e)}) $$
%In particular case,   $R$ is said to  have finite graded F-representation type if there exist
%finitely generated $\mathbb{Q}$-graded $R$-modules $M_1, \ldots ,M_s$ by which $R$ has finite
%graded F-representation type.
%\end{enumerate}
\end{definition}

  S. Takagi and R. Takahashi    exhibit in \cite{TT} examples of rings with finite F-representation type.

 \begin{example}\label{Ex FFRT}
A ring $R$ has finite F-representation type in the following cases:
%(resp. finite graded F-representation type)

\begin{enumerate}

  \item [(a)] $R=K[x_1,...,x]$ or $R=K[\![x_1,...,x]\!]$ where $K$ is a field of prime characteristic $p$ with $[K:K^p]< \infty$ (Corollary \ref{L2.5}).
  \item [(b)](\cite[Observation 3.1.2]{SV})  $R$ is an    complete $F$-finite regular local
ring of prime characteristic $p > 0$ with  $[K:K^p]< \infty$ where $K$ is the residue field of $R$.
%(resp. a polynomial ring $k[x_1, \ldots ,x_d]$ over a
%field $k$ of characteristic $p > 0$ such that $[k : k^p] < \infty$).
  \item [(c)](\cite[Observation 3.1.3]{SV})  $R$ is a   Cohen-Macaulay $F$-finite local ring of prime characteristic $p$ with finite Cohen-Macaulay type.
  %(resp. a
%Cohen-Macaulay graded ring) of prime characteristic $p$ with finite representation type (resp. finite graded representation type).
  \item [(d)]\cite[Example1.3(ii) ]{TT} $R$ is an  Artinian  $F$-finite local ring of prime characteristic  $p$ with  $[K:K^p]< \infty$ where $K$ is the residue field of $R$.
  %(resp. $R =\bigoplus_
%{n \geq 0} R_n$ is an Artinian graded ring with $k := R_0$ a field of characteristic
%$p > 0$) such that $[k : k^p] < \infty$.
  \item [(e)]\cite[Theorem 1]{TS}  $R$ is a complete local  one-dimensional domain
of prime characteristic such that its residue field is algebraically closed or finite.
\end{enumerate}
\end{example}

T.Shibuta in his paper \cite{TS}  presents examples \cite[Example 3.3 and Example 3.4]{TS} of a complete local
one-dimensional domains which do not have finite F-representation type with a
perfect residue field.

\begin{remark}
  Notice from the definitions \ref{D1} and \ref{D2.6} that a Noetherian ring that has FFRT is F-finite. As a result, if $R$ is a Noetherian ring that is not F-finite, then $R$ does not have FFRT. For example, if $K=\mathbb{Z}/p\mathbb{Z}$ where $p$ is a prime integer and $R=K(x_1,x_2,...,x_n,...)$, then $F_*^e(R)$ is $R$-vector space with infinite basis (Corollary \ref{L2.5}(b)). This proves that $R$ is not F-finite and hence $R$ does not have $FFRT$ (K.Schwede and W.Zhang observed in \cite[Section 2]{SZ} that $R$ is not F-finite).
\end{remark}

\section{Several FFRT extensions of FFRT rings}
\label{section:Several FFRT extensions of FFRT rings}
If $R$ is a  ring that has FFRT,  Y.Yao in \cite{Y} observed that  localizations and  completions of $R$  both have FFRT. In this section, we prove this observation and that each of $R[x]$ and $R[\![x]\!]$ has FFRT.
\begin{proposition}\label{P4.2}
Let $R$ be  ring and $M$ be a finitely generated $R$-module. Assume that $M$ has FFRT on $R$. If $W$ is a multiplicative  closed set and $I$ is an ideal of $R$, then:

\begin{enumerate}
  \item [(a)] $W^{-1}M$ has FFRT on $W^{-1}R$.
  \item [(b)] $\widehat{M}_I$ has FFRT on $\widehat{R}_I$.
\end{enumerate}

\end{proposition}
\begin{proof}
Suppose that $M$ has  FFRT  by finitely generated
$R$-modules $M_1$, ... , $M_s$. If $e \in \mathbb{Z}^{+}$,  there exist
nonnegative integers $n_{(e,1)}, \ldots , n_{(e,s)}$ such that
 $$F_*^e(M)= \bigoplus_{j=1}^s M_j^{\oplus n_{(e,j)}}. $$

It follows from Proposition \ref{L2.3} and Proposition \ref{LLL2} that
\begin{equation*}
 F_*^e(W^{-1}M) = W^{-1}F_*^e(M)=\bigoplus_{j=1}^s (W^{-1}M_j)^{\oplus n_{(e,j)}}.
\end{equation*}
As a result, $W^{-1}M$ has  FFRT on $W^{-1}R$  by the finitely generated
$W^{-1}R$-modules $W^{-1}M_1$, ... , $W^{-1}M_s$. Furthermore, by Proposition \ref{L2.4} and Theorem \ref{thm 2.21}, we get that

\begin{equation*}
  F_*^e(\widehat{M}_I) = \widehat{F_*^e(M)}_I= \bigoplus_{j=1}^s (\widehat{(M_j)}_I)^{\oplus n_{(e,j)}}.
\end{equation*}
Therefore, we conclude that $\widehat{M}_I$ has  FFRT on $\widehat{R}_I$  by the finitely generated
$\widehat{R}_I$-modules $\widehat{M_1}_I$, ... , $\widehat{M_s}_I$.
\end{proof}

\begin{proposition}\label{P4.5}
Let $R$ be a  ring  and let $M$ be a finitely generated $R$-module.
If $M$ has FFRT over $R$, then $M[x]$ has FFRT over $R[x]$.
\end{proposition}

\begin{proof}
 If $e \in \mathbb{Z}^{+}$ and $q=p^e$, recall from  Proposition \ref{P2.41} that $$ F_*^e(M[x])= (F_*^e(M)[x])^{\oplus^{q}}.$$
Since $M$ has FFRT, there exist finitely generated $R$-modules $M_1,...,M_s$ and nonnegative integers $n_{(e,1)},...,n_{(e,s)}$ such that $$F_*^e(M)=\bigoplus_{j=1}^s M_j^{\oplus n_{(e,j)}}.$$
Tensoring with $R[x]$ both sides of the above equality and using the  Remark \ref{R4.3} yield that $$F_*^e(M)[x] =\bigoplus_{j=1}^s (M_j[x])^{\oplus n_{(e,j)}}.$$ As a result, it follows that $$F_*^e(M[x]) =\bigoplus_{j=1}^s (M_j[x])^{\oplus qn_{(e,j)}}.$$
Therefore, we conclude  that $ M[x]$ has FFRT by the finitely generated $R[x]$-modules $M_1[x],...,M_s[x]$.
\end{proof}

\begin{theorem}\label{C4.6}
Let $R$ be a  ring of  prime characteristic $p$. If $R$ has FFRT, then:
\begin{enumerate}
  \item [(a)] $R[x]$ has FFRT over $R[x]$.
  \item [(b)] $R[\![x]\!]$ has FFRT over $R[\![x]\!]$.
  \item[(c)] $R[x_1,\ldots,x_n]$ has FFRT over $R[x_1,\ldots,x_n]$.
  \item[(d)] $R[\![x_1,\ldots,x_n]\!]$ has FFRT over $R[\![x_1,\ldots,x_n]\!]$.
\end{enumerate}
\end{theorem}
\begin{proof}
(a) This is a particular case of Proposition \ref{P4.5}.\\
(b) Recall from Proposition \ref{EX1} that if $I=(x)R[x]$, then $\widehat{R[x]}_I= R[\![x]\!]$. Now apply the above result and Proposition \ref{P4.2}.\\
(c) and (d) follows from (a) and (b).
\end{proof}

\section{Modules of FFRT by FFRT system}
\label{section: Rings of FFRT by FFRT system}

The notion of FFRT system was introduced by Yao in his paper \cite{Y} as follows:

\begin{definition}
\emph{A finite set $\Gamma$ of   finitely generated $R$-modules is said to be a finite $F$-representation type System  (or FFRT system)  if  for every $N \in \Gamma$, $F_*^1(N)$ can be written as a finite direct sum whose direct summands are all taken from $\Gamma$.
We say that $M$ has FFRT by a FFRT system if there exists  a FFRT system $\Gamma  $  such that $M$ has FFRT by $\Gamma  $.}
\end{definition}
Let $R$ be a Noetherian ring not necessarily of prime characteristic.  A class $\textit{C}(R)$  of finitely generated $R$-modules is called reasonable if it satisfies that  every $R$-module
that is isomorphic to a direct summand of a module in $\textit{C}(R)$ is in $\textit{C}(R)$. The notion of reasonable class was introduced by R.Wiegand  in his paper \cite[Section 1]{RW} who proved the following useful theorem.

\begin{proposition}\label{T1}\cite[Theorem 1.4.]{RW}
Let $R$ be a Noetherian semilocal ring not necessarily of prime characteristic,i.e. $R$ has finitely many maximal ideals, and let $S$ be  a faithfully flat $R$-algebra (Definition \ref{D2.7}).  Let $\textit{C}(R)$  and $\textit{C}(S)$  be reasonable classes of modules such that
$S\otimes_{R}M \in \textit{C}(S)$ for all $M \in \textit{C}(R)$. If $\textit{C}(S)$ contains only finitely many
indecomposable modules up to isomorphism, the same holds for $\textit{C}(R)$.
\end{proposition}
\begin{theorem}\label{P1}
Let $R$ be a  local ring. If a finitely generated $R$-module  $M$ has FFRT, then $M$ has FFRT by a FFRT system.
\end{theorem}

\begin{proof}
Let $\mathfrak{m}$ be the maximal ideal of $R$ and let $\hat{R}$  (respectively$\hat{M}$) denote the $\mathfrak{m}$-adic completion of $R$ (respectively $M$).  Recall from Proposition \ref{thm 2.21} that   $\hat{R}$  is a faithfully
flat $R$-algebra. Assume that $\Gamma$ is a finite set of finitely generated $R$-modules such that $M$ has FFRT by $\Gamma$. If  $\hat{\Gamma}= \{ \hat{N} \, | \, N \in \Gamma \}$, then $\hat{M}$ has FFRT by $\hat{\Gamma}$ (Proposition \ref{P4.2}). Now let $ L_{\hat{R}}^0(\hat{\Gamma})= \hat{\Gamma}$ and let   $ L_{\hat{R}}^e (\hat{\Gamma})$ denote the set of all direct summands of $F_*^e(N)$ for all $N \in  L_{\hat{R}}^{e-1} (\hat{\Gamma)}$. Set  $ L_{\hat{R}} (\hat{\Gamma})= \bigcup_{e \in \mathbb{N}}L_{\hat{R}}^e (\hat{\Gamma})$. Similarly, we define  $ L_R (\Gamma)= \bigcup_{e \in \mathbb{N}}L_R^e (\Gamma)$ where  $ L_R^e (\Gamma)$ is  the set of all direct summands of $F_*^e(N)$ for all $N \in  L_R^{e-1} (\Gamma)$ and $ L_R^0(\Gamma)= \Gamma$. Therefore, $L_{\hat{R}} (\hat{\Gamma})$ and $ L_R (\Gamma)$ are reasonable classes of modules. We aim now to show that
$ \hat{R}\otimes_{R}N \in L_{\hat{R}} (\hat{\Gamma})$ for all $N \in L_{\hat{R}} (\hat{\Gamma})$. For this purpose, we show that $ \hat{R}\otimes_{R}N \in L_{\hat{R}}^e(\hat{\Gamma})$ whenever  $N \in  L_{R}^e(\Gamma)$ for all $e \in \mathbb{Z}_{+}$. The case when $e=0$ is obvious. If $N \in  L_{R}^e(\Gamma)$, then $N$ is a direct summand of $F_*^e(H)$  for some $H \in L_{R}^{e-1}(\Gamma)$ and hence $\hat{R}\otimes_{R}N$ is a direct summand of $\hat{R}\otimes_{R}F_*^e(H)$. Notice from Proposition \ref{thm 2.21} and Proposition \ref{L2.4} that $\hat{R}\otimes_{R}F_*^e(H)= \widehat{F_*^e(H)}=F_*^e(\widehat{H})=F_*^e(\hat{R}\otimes_{R}H)$. The induction hypothesis implies that $\hat{R}\otimes_{R}H \in L_{\hat{R}}^{e-1}(\hat{\Gamma})$ and consequently $ \hat{R}\otimes_{R}N \in L_{\hat{R}}^e(\hat{\Gamma})$. Since $\widehat{M}$ has FFRT and $\hat{R}$ satisfies Krull-Schmidt theorem on the class of all finitely generated $R$-modules (Discussion \ref{disc2.33}), it follows  that $\widehat{M}$ has FFRT by FFRT system and hence  $L_{\hat{R}} (\hat{\Gamma})$ contains only finitely many indecomposable modules up to isomorphism and consequently by Proposition \ref{T1} the same holds for $ L_R (\Gamma)$. Now let $\{ M_1, ... , M_s \}$ be the  set of representatives for the isomorphism classes of those indecomposable modules in $ L_R (\Gamma)$. Therefore, every  $R$-module  $N \in \Gamma$ can be written as a finite direct sum whose direct summands are all taken from $\{ M_1, ... , M_s \}$ and hence $M$ has FFRT by $\{ M_1, ... , M_s \}$. Furthermore,  $F_*^1(M_j)$ can be written as a finite direct sum whose direct summands are all taken from $\{M_1, ... , M_s\}$ which makes $\{ M_1, ... , M_s \}$ a  FFRT system.  This proves that $M$ has FFRT by the FFRT system $\{ M_1, ... , M_s \}$.
\end{proof}

\section{FFRT locus of a module is an open set}
\label{section: FFRT locus of a module is an open set}
Let $R$ be a  ring and $M$ be a finitely generated $R$-module.

The FFRT  locus of $M$ is the set
$$FFRT(M):=\{Q \in \Spec(R) \;|\; M_Q \text{ has } FFRT \text{ over } R_Q\}.$$ In this section, we will prove that $FFRT(M)$ is an open set in the Zariski topology on $\Spec(R)$.

The following Lemma is essential to prove the main result in this section.
\begin{lemma}\label{L4}
Let $R$ be a ring and  $M$   a finitely generated $R$-module. If  $Q$ is  a prime ideal such  that $M_Q$ has Finite F-representation type over $R_Q$, then  $M_u$ has Finite F-representation type over $R_u$ for some $u \in R \setminus Q$.

\end{lemma}

\begin{proof}
Let $Q $ be a prime ideal  for which $M_Q$ has Finite F-representation type. It follows from Theorem \ref{P1}  that $M_Q$ has FFRT by a FFRT $\{{M_1}_Q,...,{M_t}_Q \}$ where $M_1,...,M_t$ are finitely generated $R$-modules.   As a result, for  every positive integer $e$ and every $j \in \{ 1,...,t \}$ there exist nonnegative integers $\alpha(1,j)$ and $\beta(i,j)$ for all $ 1 \leq i,j \leq t$ such that
\begin{equation*}
 F_*^1(M)_Q=F_*^1(M_Q)= \bigoplus_{j=1}^t[(M_j)_Q]^{\oplus\alpha(1,j)} = [\bigoplus_{j=1}^tM_j^{\oplus\alpha(1,j)}]_Q
\end{equation*}
and
\begin{equation*}
F_*^1(M_i)_Q=F_*^1(({M_i})_Q)= \bigoplus_{j=1}^t[(M_j)_Q]^{\oplus\beta(i,j)} = [\bigoplus_{j=1}^tM_j^{\beta(i,j)}]_Q  \text{ for all } 1 \leq i \leq t.
\end{equation*}
By lemma \ref{LL2}, there exist $s, s_1,...,s_t \in R \setminus Q$ such that

\begin{equation*}
 F_*^1(M)_s=F_*^1(M_s)= \bigoplus_{j=1}^t[(M_j)_s]^{\oplus\alpha(1,j)} = [\bigoplus_{j=1}^tM_j^{\oplus\alpha(1,j)}]_s
\end{equation*}
and
\begin{equation*}
F_*^1(M_i)_{s_i}=F_*^1(({M_i})_{s_i})= \bigoplus_{j=1}^t[(M_j)_{s_i}]^{\oplus\beta(i,j)} = [\bigoplus_{j=1}^tM_j^{\oplus\beta(i,j)}]_{s_i}  \text{ for all } 1 \leq i \leq t.
\end{equation*}
Let $u=ss_1...s_t$. We will prove by the induction on $e\geq 1$ that $F_*^e(M_u)$ can be written as a direct sum with direct summand taken from $\{(M_1)_u,\ldots,(M_t)_u \}$.

It follows from Proposition \ref{LLL2} and the above equations that

\begin{equation} \label{Equ1}
 F_*^1(M)_u=F_*^1(M_u)= \bigoplus_{j=1}^t[(M_j)_u]^{\oplus\alpha(1,j)} = [\bigoplus_{j=1}^tM_j^{\oplus\alpha(1,j)}]_u
\end{equation}
and
\begin{equation*}
F_*^1(M_i)_u=F_*^1(({M_i})_u)= \bigoplus_{j=1}^t[(M_j)_u]^{\oplus\beta(i,j)} = [\bigoplus_{j=1}^tM_j^{\oplus\beta(i,j)}]_u  \text{ for all } 1 \leq i \leq t.
\end{equation*}

Now assume that $  F_*^e(M_u)= \bigoplus_{i=1}^t[(M_i)_u]^{\oplus\alpha(e,i)}$  where $\alpha(e,i)$ is nonnegative  for all $ 1 \leq i \leq t$.  By Equation \ref{Equ1}, it follows that
\begin{eqnarray*}
   F_*^{e+1}(M_u) &=& F_*^1[\bigoplus_{i=1}^t[(M_i)_u]^{\oplus\alpha(e,i)}] \\
    &=& \bigoplus_{i=1}^tF_*^1[({M_i})_u]^{\oplus\alpha(e,i)} \\
    &=& \bigoplus_{i=1}^t[\bigoplus_{j=1}^t[(M_j)_u]^{\oplus\beta(i,j)}]^{\oplus\alpha(e,i)}
\end{eqnarray*}
This induction on $e$ proves that $M_u$ has Finite F-representation type over $R_u$.
\end{proof}

For every  $u \in R$ where $R$ is a ring, recall that $V(u)$ denote the set of all prime ideals $P$ containing $u$ and let $D(u)= \Spec(R)\backslash V(u)$. Recall that the collection $\{D(u) \, | \, u \in R \}$ forms a basis of open
sets for the Zariski topology on $\Spec(R)$ (cf \cite[Section 4]{Mat}).

\begin{theorem}\label{P4.11}
 If $R$ is a ring and  $M$ is  a finitely generated $R$-module,  then the FFRT locus of $M$  is an open set in the Zariski topology on $\Spec(R)$.
\end{theorem}

\begin{proof}
   If $FFRT(M)$ is empty, $FFRT(M)$ is open. Assume now that $FFRT(M)$ is not empty. If $Q \in FFRT(M) $, there exists by lemma \ref{L4} an element $u \in R \setminus Q$ such that $M_u$ has FFRT over $R_u$ and hence by Proposition \ref{LLL2}(a) and Proposition \ref{P4.2}  $M_P$ has FFRT over $R_P$ for all $P \in \textit{D}(u)$. This proves that  $ \textit{D}(u) \subseteq  FFRT(M)$ and hence $FFRT(M)$ is open set in the Zariski topology on $\Spec(R)$.
\end{proof}
%%%%%%%%%%%%%%%%%%%%%%%%%%%%%%%%%%%%%%%%%%%%%%%%%%%%%%%%%%%%%%

\chapter{On the FFRT over hypersurfaces}
\label{chapter:The finite F-representation type over hypersurfaces}

 Let $S=K[x_1,...,x_n]$ or $S=K[\![x_1,...,x_n]\!]$ where $K$ is a field of prime characteristic $p$ with $[K:K^p]< \infty$ and  $R=S/fS$ for some $f\in S$, in this chapter we introduce  a presentation of $F_*^e(R)$  as a cokernel of a matrix factorization of $f$ that is denoted  $M_S(f,e)$.  The properties of this presentation and its applications to the concept of finite F-representation type will be considered in this chapter.

\section{The presentation of  $F_*^e(S/fS)$ as a cokernel of a Matrix Factorization of $f$}
\label{section:The presentation of  Fe as a cokernel of a Matrix Factorization of f}
Throughout the rest of this thesis, unless otherwise mentioned, we will adopt the following notation:
\begin{notation}\label{N4.1}
\emph{$K$ will denote a  field  of prime characteristic $p$ with $[K:K^p] < \infty $, and we set $q=p^e$ for some $e \geq 1$. $S$ will denote the ring  $K[x_1,\dots,x_n]$ or $K[\![x_1,\dots,x_n]\!]$. Let $\Lambda_e$ be  a basis of $K$ as $K^{p^e}$-vector space.   We set $$ \Delta_e := \{\lambda x_1^{a_1}\dots x_n^{a_n}\,|\, 0 \leq a_i \leq p^e-1 \text{ for all } 1\leq i \leq n \text{ and } \lambda \in \Lambda_e \} $$
and set  $r_e := | \Delta_e| = [K:K^p]^eq^n .$}
\end{notation}
\begin{discussion}\label{D4.2}
Recall from Corollary  \ref{L2.5} that $ \{ F_*^e(j)\, |\, j \in \Delta_e \} $ is a basis of $F_*^e(S)$ as free $S$-module. Let  $f \in S$. If $S \xrightarrow{f} S $ is the $S$-linear map given by $s \longmapsto fs$, let $F_*^e(S) \xrightarrow{F_*^e(f)} F_*^e(S) $ be the $S$-linear map that is given by $F_*^e(s) \longmapsto F_*^e(fs)$ for all $s\in S$. We write $M_S(f,e)$ (or $M(f,e)$ if $S$ is known) to denote the matrix $\Mat(F_*^e(f))$ which is the $r_e \times r_e$  matrix representing the $S$-linear map $F_*^e(S) \xrightarrow{F_*^e(f)} F_*^e(S) $ with respect to the basis $ \{ F_*^e(j)\, |\, j \in \Delta_e \} $ (see \ref{Rem 2.9}). Indeed, if $j \in \Delta_e $, there exists a unique  set  $\{f_{(i,j)} \in S \,|\, i \in \Delta_e \} $ such that $F_*^e(jf)= \bigoplus_{i\in\Delta_e}f_{(i,j)}F^e(i)$ and consequently $M_S(f,e)=[f_{(i,j)}]_{(i,j) \in \Delta^2_e}$.   The   matrix $M_S(f,e)$   is called the matrix of relations of $f$ over $S$ with respect to $e$.
\end{discussion}

\begin{example}
Let $K$ be a perfect field of prime characteristic 3 , $S=K[x,y]$ or $S=k[\![x,y]\!]$ and  let $ f =x^2 + xy$.  We aim to construct  $M_S(f,1)$. Since the  set $\{ F_*^1(1) , F_*^1(x) , F_*^1(x^2) , F_*^1(y) , F_*^1(yx) , F_*^1(yx^2),F_*^1(y^2) , F_*^1(y^2x) , F_*^1(y^2x^2) \}$ is the basis of $F_*^1(S)$ as  $S$-module, we get that \\
$ F_*^1(f) = F_*^1(x^2 + xy) = F_*^1(x^2) + F_*^1(yx) $  \\
$ F_*^1(xf)= F_*^1(x^3 + x^2y) = xF_*^1(1) + F_*^1(x^2y)$ \\
$ F_*^1(x^2f)= F_*^1(x^4 + x^3y) = xF_*^1(x) + xF_*^1(y)$ \\
$ F_*^1(yf) = F_*^1(yx^2 + xy^2) = F_*^1(yx^2) + F_*^1(y^2x) $  \\
$ F_*^1(yxf) = F_*^1(yx^3 + x^2y^2) = xF_*^1(y) + F_*^1(y^2x^2)  $  \\
$ F_*^1(yx^2f) = F_*^1(yx^4 + x^3y^2) = xF_*^1(yx) + xF_*^1(y^2) $  \\
$ F_*^1(y^2f) = F_*^1(y^2x^2 + xy^3) = F_*^1(y^2x^2) + yF_*^1(x) $  \\
$ F_*^1(y^2xf) = F_*^1(y^2x^3 + x^2y^3) = xF_*^1(y^2) + yF_*^1(x^2) $  \\
$ F_*^1(y^2x^2f) = F_*^1(y^2x^4 + x^3y^3) = xF_*^1(y^2x) + xyF_*^1(1) $  \\
Therefore,

$$ M_S(f,1)= \begin{bmatrix}  0 & x & 0 & 0 & 0 & 0 & 0 & 0 & yx \\ 0 & 0 & x & 0 & 0 & 0 & y & 0 & 0 \\ 1 & 0 & 0 & 0 & 0 & 0 & 0& y & 0 \\ 0 & 0 & x & 0 & x & 0 & 0 & 0 & 0 \\ 1 & 0 & 0 & 0 & 0 & x & 0 & 0 & 0 \\ 0 & 1 & 0 & 1 & 0 & 0 & 0 & 0 & 0 \\ 0 & 0 & 0 & 0 & 0 & x & 0 & x & 0 \\ 0 & 0 & 0 & 1 & 0 & 0 & 0 & 0 & x \\ 0 & 0 & 0 & 0 & 1 & 0 & 1 & 0 & 0 \end{bmatrix} $$
\end{example}
\begin{remark}\label{Ex4.3}
If $m \in \mathbb{N}$, then $F_*^e(f^{mq}j)=f^mF_*^e(j)$ for all $j \in \Delta_e$. This makes $M_S(f^{mq},e)=f^mI$ where $I$ is the identity matrix of size $r_e \times r_e$.
\end{remark}

\begin{proposition}\label{P4.3}

If $f,g  \in S$, then
\begin{enumerate}
\item[(a)] $ M_S(f+g,e)=M_S(f,e)+M_S(g,e)$,
\item[(b)] $ M_S(fg,e)=M_S(f,e)M_S(g,e)$ and consequently $ M_S(f,e)M_S(g,e)=M_S(g,e)M_S(f,e)$, and
\item[(c)] $  M_S(f^m,e)= [M_S(f,e)]^m$ for all $m \geq 1$
\end{enumerate}
\end{proposition}
\begin{proof}
The proof follows immediately from Discussion \ref{D4.2} and Remark \ref{Rem 2.9}.
\end{proof}

According to  Remark \ref{Ex4.3} and Proposition \ref{P4.3}, we get that $$M_S(f^k,e)M_S(f^{q-k},e)=M_S(f^{q-k},e)M_S(f^k,e)= M_S(f^q,e)= fI_{r_e}$$ for all $0 \leq k \leq q-1$. This shows the following result.

\begin{proposition}\label{P5.5}
For every $f \in S$ and $0 \leq k \leq q-1$, the pair  $ ( M_S(f^k,e),M_S(f^{q-k},e))$ is a matrix factorization of $f$.

\end{proposition}
\begin{discussion}

Let  $x_{n+1}$ be a new variable  and let $ L=S[x_{n+1}]$  if $S=K[x_1,\ldots,x_n]$ or ($ L=S[\![x_{n+1}]\!]$  if $S=K[\![x_1,\ldots,x_n]\!]$). We aim to describe $M_L(g,e)$ for some $g \in L$ by describing the columns of $M_L(g,e)$. First we will construct a basis  of the free $L$-module  $F_*^e(L)$   using  the basis $ \{ F_*^e(j) | j \in\Delta_e \} $  of the free $S$-module $F_*^e(S)$.  For each $0\leq v \leq q-1 $, let $ \mathfrak{B}_v= \{F_*^{e}(j x_{n+1}^{v})\, |\, j \in\Delta_e  \}$  and set  $ \mathfrak{B} = \mathfrak{B}_0 \cup \mathfrak{B}_1 \cup \mathfrak{B}_2\cup \dots \cup \mathfrak{B}_{q-1} $.
Therefore $\mathfrak{B}$  is a basis for $F_*^e(L)$ as free $L$-module and if   $g \in L$,  we write
 \begin{eqnarray*}F_*^e(g)& = & \bigoplus\limits_{i\in\Delta_e}g_{i}^{(0)}F^{e}(i) \oplus\bigoplus\limits_{i\in\Delta_e}g_{i}^{(1)}F^{e}(ix_{n+1}^1)\oplus\dots \oplus \bigoplus\limits_{i\in\Delta_e}g_{i}^{(q-1)}F^{e}(ix_{n+1}^{q-1})
\end{eqnarray*}
where $ g_{i}^{(s)} \in L $ for all $0 \leq s \leq q-1$ and $i\in\Delta_e $.
For each $ 0 \leq s \leq q-1 $ let   $ [F_*^e(g)]_{\mathfrak{B}_s} $ denote the column whose entries are the coordinates $ \{g_{i}^{(s)} \in L \,|\, i\in\Delta_e\} $ of $F_*^e(g)$ with respect to $\mathfrak{B}_s$. Let $[F_*^e(g)]_\mathfrak{B}$ be the $r_eq \times 1$  column that is composed of the columns $[F_*^e(g)]_{\mathfrak{B}_0},\dots,[F_*^e(g)]_{\mathfrak{B}_{q-1}}$ respectively.
%\begin{equation*}
 % [F_*^e(g)]_L= \begin{bmatrix} [F_*^e(g)]_{L_0} \\ \vdots \\ [F_*^e(g)]_{L_{q-1}} \end{bmatrix}
%\end{equation*}
Therefore  $M_L(g,e)$ is the  $r_eq \times r_eq$ matrix over $L$ whose columns are all the columns  $ [F_*^e(jx_{n+1}^s g)]_\mathfrak{B}$ where $ 0 \leq s \leq q-1$ and $ j \in \Delta_e$. This means that
 $ M_L(g,e) = \left[\begin{array}{ccc}
                 C_{0} & \ldots & C_{q-1} \\
               \end{array}
             \right]$  where $C_{m}$ is the  $r_eq \times r_e$ matrix over $L$ whose columns are the columns $ [F_*^e(jx_{n+1}^m g)]_\mathfrak{B}$  for all $j \in \Delta_e$. If we define $C_{(k,m)}$ to be the  $r_e \times r_e$ matrix over $L$ whose columns are $ [F_*^e(jx_{n+1}^m g)]_{\mathfrak{B}_k}$ for all $j\in\Delta_e$, then $C_{m}$ consists of $ C_{(0,m)},\dots,C_{(q-1,m)} $ respectively
%\begin{equation*}
%   C_{m}= \left[
 %             \begin{array}{c}
 %               C_{(0,m)} \\
 %               \vdots\\
  %             C_{(q-1,m)} \\
  %            \end{array}
 %           \right]
%\end{equation*}
and hence the matrix
 $M_L(g,e)$ is given by :
\begin{equation}
  M_L(g,e) = \left[
               \begin{array}{ccc}
                 C_{0} & \ldots & C_{q-1} \\
               \end{array}
             \right]=\left[
               \begin{array}{ccc}
                 C_{(0,0)} & \ldots & C_{(0,q-1)} \\
                 \vdots&   & \vdots \\
                 C_{(q-1,0)} & \ldots & C_{(q-1,q-1)} \\
               \end{array}
             \right]
\end{equation}

\end{discussion}
Using the above discussion we can prove the following lemma
\begin{lemma}\label{L4.7}
Let $f \in S$ with $A = M_S(f,e)$ and let $ L=S[x_{n+1}]$  if $S=K[x_1,\ldots,x_n]$  or ($ L=S[\![x_{n+1}]\!]$  if $S=K[\![x_1,\ldots x_n]\!]$). If $0\leq d \leq q-1 $, then
\begin{equation}
  M_L(fx_{n+1}^d,e) = \left[
               \begin{array}{ccc}
                 C_{(0,0)} & \ldots & C_{(0,q-1)} \\
                 \vdots&   & \vdots \\
                 C_{(q-1,0)} & \ldots & C_{(q-1,q-1)} \\
               \end{array}
             \right]
\end{equation}
where
\begin{equation*}
    C_{(k,m)} = \begin{cases}
              A             &\text{if   } (m,k) \in \{ (d,0),(d+1,1),\ldots ,(q-1,q-1-d)\} \\
             x_{n+1}A                &\text{if   } (m,k) \in \{ (0,q-d),(1,q-1-d),\ldots ,(d,q-1)\}\\
              0               & \text{otherwise}
           \end{cases}
\end{equation*}

\end{lemma}

\begin{proof}
If $A = M_S(f,e)=[f_{(i,j)}]$, for each $j \in \Delta_e$ we can write  $F_*^e(jf)= \bigoplus_{i\in \Delta_e}f_{(i,j)}F^e(i)$. If $g=fx_{n+1}^d$, for every $1 \leq m \leq q-1 $ and   $j \in \Delta_e$, it follows that $F_*^e(jx_{n+1}^mg) = \bigoplus_{i\in \Delta_e}f_{(i,j)}F^e(ix_{n+1}^{d+m})$. Therefore,

\begin{equation*}
    F_*^e(jx_{n+1}^mg) = \begin{cases}
             \bigoplus_{i\in \Delta_e}f_{(i,j)}F^e(ix_{n+1}^{d+m})   &\text{if   } d+m \leq q-1 \\
           \bigoplus_{i\in \Delta_e}x_{n+1}f_{(i,j)}F^e(ix_{n+1}^{d+m -q })             &\text{if   } d+m > q-1
           \end{cases}
\end{equation*}
 Accordingly, if $m \leq q-1-d $, then

 %\begin{equation*}
   %C_{m}= \left[
   %           \begin{array}{c}
   %             C_{(0,m)} \\
    %            \vdots\\
    %           C_{(q-1,m)} \\
     %         \end{array}
    %        \right]
%\end{equation*}

%where

 \begin{equation*}
    C_{(k,m)} = \begin{cases}
             A             &\text{ if } k=d+m \\
              0            &\text{ if } k \neq d+m
           \end{cases}
\end{equation*}
  However,  if $m > q-1-d $, it follows that
  %\begin{equation*}
  % C_{m}= \left[
  %            \begin{array}{c}
   %             C_{(0,m)} \\
   %             \vdots\\
    %           C_{(q-1,m)} \\
   %           \end{array}
    %        \right]
%\end{equation*}

%where

 \begin{equation*}
    C_{(k,m)} = \begin{cases}
              x_{n+1}A              &\text{if } k=d+m -q \\
              0                     &\text{if } k \neq d+m -q
           \end{cases}
\end{equation*}

  This shows the required result.
\end{proof}

\begin{proposition}\label{EP}
Let $ L=S[x_{n+1}]$ (if $S=K[x_1,\ldots,x_n]$)or $ L=S[\![x_{n+1}]\!]$ (if $S=K[\![x_1,\ldots,x_n]\!]$). Suppose that $g \in L$ is given by $$ g = g_0 + g_1x_{n+1} + g_2x_{n+1}^2 + \dots +g_dx_{n+1}^d$$ where $d < q$ and $g_k \in S$ for all $ 0 \leq k \leq d $ . If $A_k=M_S(g_k,e)$ for each $0 \leq k \leq d $ then  \[ M_L(g,e) = \begin{bmatrix}
A_0& &  & & x_{n+1}A_d&x_{n+1}A_{d-1} &\ldots &x_{n+1}A_1  \\
A_1&A_0&  & & &x_{n+1}A_d &   &\vdots \\
 & & & & & &\ddots & \\
 & & & & & & &x_{n+1}A_d \\
\vdots & \vdots & \ddots & \ddots& & & & \\
A_d&\vdots&  & & & & & \\
  & A_d & & & & & &\\
  &  &  & & & & & \\
  &  &  &A_d &A_{d-1} &A_{d-2}& \ldots &A_0

\end{bmatrix}.\]
\end{proposition}
\begin{proof}

Recall from Lemma \ref{L4.7} that

$$ M_A(g_0,e)= \left[
                 \begin{array}{cccc}
                  A_0 &   &   &   \\
                    & A_0 &   &   \\
                    &   & \ddots &   \\
                    &   &  & A_0 \\
                 \end{array}
               \right]$$

$$ M_A(g_1x_{n+1},e)=\left[
                       \begin{array}{cccccc}
                           & &   &   &   & x_{n+1}A_1 \\
                         A_1 &   &  &   &   &   \\
                           & A_1 &  &   &   &  \\
                           &  & \ddots &   &   &   \\
                           &   &   &   & A_1 &  \\
                           &   &   &   &   &   \\
                       \end{array}
                     \right]$$

 $$M_A(g_2x_{n+1}^2,e)=\left[
     \begin{array}{ccccc}
         &   &   & x_{n+1}A_2 &   \\
         &   &  &   & x_{n+1}A_2 \\
       A_2 &   &  &   &   \\
         &\ddots &  &   &  \\
         &   &A_2 &   &   \\
     \end{array}
   \right]
  $$
and finally we get
$$ M_A(g_dx_{n+1}^d,e)=\begin{bmatrix}
 & &  & & x_{n+1}A_d& &  &  \\
 & &  & & &x_{n+1}A_d &   & \\
 & & & & &  &\ddots & \\
 & & & & & & &x_{n+1}A_d \\
  &   &   &  & & & & \\
A_d& &  & & & & & \\
  & A_d & & & & & &\\
  &  & \ddots & & & & & \\
  &  &  &A_d &  & &   & \\
\end{bmatrix}.   $$

 Proposition  \ref{P4.3}(a) implies that $$ M_A(g,e) = M_A(g_0,e) + M_A(g_1x_{n+1},e) + M_A(g_{2}x_{n+1}^2,e) + ... +M_A(g_dx_{n+1}^d,e).$$
 This proves the result.
\end{proof}

\begin{example}
Let $K$ be a perfect field of prime characteristic 3 and let $S=K[x]$ or $S=K[\![x]\!]$ . Assume $ L=S[y]$ (if $S=K[x]$) or $L=S[\![y]\!]$ (if $S=K[\![x]\!]$ ).  Let $ f =x^2 + xy$, $f_0=x^2$, and $f_1=x$.
%and we aim to construct  $M_A(f,1)$ using the above proposition. Let   $f_0 , f_1 \in R$  where $f_0=x$ and $f_1=x^2$. Since the  set $\{ F^1(1) , F^1(x) , F^1(x^2)  \}$ is a basis of $F^1(R)$ as right $R$-module, we get that \\
%$ F_*^1(f_0) = F_*^1(x) = 1F_*^1(x)  $  \\
%$ F_*^1(xf_0)= F_*^1(x^2)  = 1F_*^1(x^2) $ \\
%$ F_*^1(x^2f_0)= F_*^1(x^3) = xF_*^1(1) $ \\
%$ F_*^1(f_1) = F_*^1(x^2) = F_*^1(x^2)  $  \\
%$ F_*^1(xf_1) = F_*^1(x^3)  = xF_*^1(1)  $  \\
%$ F_*^1(x^2f_1) = F_*^1(x^4)  = xF_*^1(x) $  \\

%Therefore,
%$ M_R(f_0,1)= \begin{bmatrix} 0 & 0 & x \\ 1 & 0 & 0 \\ 0 & 1 & 0 \end{bmatrix} $ and $ M_R(f_1,1)= \begin{bmatrix} 0 & x & 0 \\ 0 & 0 & x \\ 1 & 0 & 0 \end{bmatrix} $.
%Let $M_0=M_R(f_0,1)$ and $M_1=M_R(f_1,1)$. Since $f= f_0 + f_1y \in A $,
By Proposition \ref{EP}, it follows that

$$ M_L(f,1)=\begin{bmatrix} M_S(f_0,1) & & yM_S(f_1,1) \\ M_S(f_1,1) & M_S(f_0,1) &  \\ & M_S(f_1,1) & M_S(f_0,1)  \end{bmatrix} = \left[\begin{array}{ccc|ccc|ccc}  0 & x & 0 & 0 & 0 & 0 & 0 & 0 & yx \\ 0 & 0 & x & 0 & 0 & 0 & y & 0 & 0 \\ 1 & 0 & 0 & 0 & 0 & 0 & 0& y & 0 \\ \hline 0 & 0 & x & 0 & x & 0 & 0 & 0 & 0 \\ 1 & 0 & 0 & 0 & 0 & x & 0 & 0 & 0 \\ 0 & 1 & 0 & 1 & 0 & 0 & 0 & 0 & 0 \\ \hline 0 & 0 & 0 & 0 & 0 & x & 0 & x & 0 \\ 0 & 0 & 0 & 1 & 0 & 0 & 0 & 0 & x \\ 0 & 0 & 0 & 0 & 1 & 0 & 1 & 0 & 0 \end{array} \right]. $$
\end{example}

%We can  observe easily the following Lemma
%\begin{lemma}\label{L4.15}
%Let $I \subseteq S$ be an ideal. If $R=S/I$, then $F_*^e(R)$ is a finitely generated $S$-module by the set $ \{ F_*^e(j + I) \,| \,j \in \Delta_e \} $
%\end{lemma}

%\begin{proof}
%Recall that the scalar multiplication of the $S$-module $F_*^e(R)$  is given by $ aF_*^e(f+I) = F_*^e(a^qf+I)$ for every $a\in S$ and $f \in S$. Since $F_*^e(S)$ is a free $S$-module, for $f \in S$ there exists uniquely $ \{ r_j \in S \, | \, j \in \Delta_e \}$ such that $F_*^e(f)= \sum_{j\in \Delta_e}r_jF_*^e(j)$ and hence $f= \sum_{j\in \Delta_e}r_j^qj $ . Therefore
%\begin{equation*}
 % F_*^e(f+I )= F_*^e(\sum_{j\in \Delta_e}r_j^qj+I)= \sum_{j\in \Delta_e}F_*^e(r_j^qj+I)=\sum_{j\in \Delta_e}r_jF_*^e(j+I)
%\end{equation*}

%\end{proof}
\begin{theorem}\label{P5.11}
Let $f \in S$ be a non-zero non-unit element. If $R=S/fS$, then
\begin{enumerate}
\item[(a)]$F_*^e(R)$ is a MCM $R$-module.
\item[(b)] $F_*^e(R)$ is isomorphic to $\Cok_S(M_S(f,e))$ as $S$-modules (and as $R$-modules).
%2) $F_*^e(R)$ is isomorphic to $\Cok_S(M_S(f,e))$ as $R$ -modules.\\
%\item[(b)]  If $S=K[[x_1,\dots,x_n]]$ and $f$ is a nonzero element of the maximal ideal of $S$, then  $F_*^e(R)$  is MCM $R$-module.
\end{enumerate}
\end{theorem}

\begin{proof}

(a) First, if  $S=K[\![x_1,\dots,x_n]\!]$, then  Proposition \ref{P2.43} implies that $R$ is  Cohen Maculay and consequently $F_*^e(R)$ is a MCM $R$-module (see Remark \ref{R 2.40}).
Now, if  $S=K[x_1,\dots,x_n]$,  then  Proposition \ref{P2.43} implies that $R$ is Cohen Maculay,i.e. $R_{\mathfrak{n}}$ is  Cohen Maculay for every maximal ideal $\mathfrak{n}$ of $R$. It follows from Remark \ref{R 2.40} that $F_*^e(R_{\mathfrak{n}})$ is  a MCM $R_{\mathfrak{n}}$-module  and hence by Proposition \ref{L2.3} $F_*^e(R)_{\mathfrak{n}}$ is  a MCM $R_{\mathfrak{n}}$-module. This shows that $F_*^e(R)$ is a MCM $R$-module.

(b) Write $I=f S$.  Since $ \{ F_*^e(j) \, | \, j \in \Delta_e \} $ is a basis of $F_*^e(S)$ as free $S$-module,
the module $F_*^e(R)$ is generated as $S$-module by the set $ \{ F_*^e(j + I) \,| \,j \in \Delta_e \} $. For every $g \in S$, define
$ \phi( F_*^e(g)) = F_*^e(g + I)$. It is clear that $\phi: F_*^e(S)  \longrightarrow F_*^e(R)$ is a surjective homomorphism of $S$-modules whose kernel is the $S$-module
$F_*^e(I)$ that is generated by the set $ \{ F_*^e(jf) \, | \, j \in \Delta_e \} $.
Now, define the $S$-linear map  $ \psi : F_*^e(S) \rightarrow F_*^e(S)$  by $ \psi( F_*^e(h)) = F_*^e(hf)$ for all $h\in S$.
We have an exact sequence $F_*^e(S)\xrightarrow{\psi}F_*^e(S)\xrightarrow{\phi}F_*^e(R) \xrightarrow{} 0$. Notice for each $j \in \Delta_e$  that $\psi(F_*^e(j))= F_*^e(jf)= \bigoplus_{i\in\Delta_e}f_{(i,j)}F^e(i)$
and hence $M_S(f,e)$
represents the map $\psi$ on the given free-bases (Remark \ref{Rem 2.9}).
By  Proposition \ref{P5.5} and  Remark \ref{C3.2}(a), it follows  that  $F_*^e(R)$ is isomorphic to $\Cok_S(M_S(f,e))$ as $R$ -modules.
\end{proof}

\begin{corollary}\label{C4.12}
Let $f\in S$ be a non-zero non-unit element. If $1 \leq k \leq q-1$ and $R=S/fS$, then
\begin{enumerate}
\item[(a)]  $F_*^e(S/f^kS)$ is a MCM $S/f^kS$ -modules isomorphic to  $\Cok_S(M_S(f^k,e))$ as $S$-modules (and as $S/f^kS$-modules), and
\item[(b)] $F_*^e(S/f^kS)$ is a MCM $R$-module isomorphic to $\Cok_S(M_S(f^k,e))$ as $S$-modules (and as $R$-modules).
%\item[(b)]  $F_*^e(S/f^kS)$ is isomorphic to $Cok_S(M_S(f^k,e))$ as $R$ -modules,
%\item[(c)]  If $S=K[\![x_1,\dots,x_n]\!]$ and $f$ is a nonzero element of the maximal ideal of $S$, then $F_*^e(S/f^kS)$  is MCM $R$-module,

%\item[(e)]   If $S=K[\![x_1,\dots,x_n]\!]$ and $f$ is a nonzero element of the maximal ideal of $S$, then $F_*^e(S/f^kS)$  is MCM $S/f^kS$-module.
\end{enumerate}
\end{corollary}
\begin{proof}
%(a) and (c) follow from Proposition \ref{P4.11}.

(a) can be proved by applying Proposition \ref{P5.11} to $f^k$ instead of $f$.

(b) Since  the pair $(M_S(f^k,e),M_S(f^{q-k},e))$ is a matrix factorization of $f$, it follows that $f \Cok_S(M_S(f^k,e))=0$. Notice that  $fF_*^e(S/f^kS)=0$
This makes $F_*^e(S/f^kS)$ and  $\Cok_S(M_S(f^k,e))$ $R$-modules and consequently  $F_*^e(S/f^kS)$ is isomorphic to $\Cok_S(M_S(f^k,e))$ as $R$ -modules. Since $F_*^e(S/f^kS)$ is a MCM $S/f^kS$ -modules and $(f+f^kS)F_*^e(S/f^kS)=0$, it follows by Proposition \ref{P2.56} that $F_*^e(S/f^kS)$ is  MCM module over  the ring $\frac{S/f^kS}{(f+f^kS)(S/f^kS)}= S/fS$.
\end{proof}

\begin{proposition}\label{L4.13}
Let $K$ be a  field of prime characteristic $p>2 $ with $[K:K^p] <\infty$  and let $T=S[z]$ if $S=K[x_1,\dots,x_n]$ (or $T=S[\![z]\!]$ if $S=K[\![x_1,\dots,x_n]\!]$). If $A=M_S(f,e)$ for some $f \in S$,   then
\begin{equation*}
 F_*^e(T/(f+z^2))= \Cok_{T}\left[
                                              \begin{array}{cc}
                                                A^{\frac{q-1}{2}} & -zI \\
                                                zI &  A^{\frac{q+1}{2}} \\
                                              \end{array}
                                            \right].
\end{equation*}
\end{proposition}

\begin{proof}
Let $I$ be the identity matrix in $M_{r_e}(S)$. It follows from Proposition \ref{EP} that $ M_{T}(f+z^2, e)$ is a $q \times q$ matrix over the ring $M_{r_e}(S)$ that is given by

\begin{equation*}
  M_{T}(f+z^2, e) = \left[
       \begin{array}{ccccccc}
         A &   &  &   &   & zI & 0  \\
         0 & A &  &   &   &   & zI \\
         I & 0 & A &   &   &   &  \\
           & I & 0  & A &   &   &   \\
           &   & \ddots & \ddots & \ddots &   &  \\
           &   &   &  I & 0 & A &   \\
           &   &   &   &  I & 0 & A \\
       \end{array}
     \right].
\end{equation*}
It follows from  Corollary \ref{C2.14} and Theorem \ref{P5.11} that

\begin{equation*}
  F_*^e(T/(f+z^2)) = \Cok_{T}\left[
                                              \begin{array}{cc}
                                                zI &  (-1)^mA^{\frac{q+1}{2}} \\
                                               (-1)^{m-1} A^{\frac{q-1}{2}} & zI \\
                                              \end{array}
                                            \right]  \text{ where } m=\frac{q-1}{2}.
\end{equation*}

If $m$ is odd integer, we get

\begin{eqnarray*}
    F_*^e(T/(f+z^2))  &=&\Cok_{T}\left[
                                              \begin{array}{cc}
                                                A^{\frac{q-1}{2}} & zI  \\
                                                 zI & -A^{\frac{q+1}{2}} \\
                                              \end{array}
                                            \right]\\
  &=& \Cok_{T}(
  \left[
                                              \begin{array}{cc}
                                                A^{\frac{q-1}{2}} & zI  \\
                                                 zI & -A^{\frac{q+1}{2}} \\
                                              \end{array}
                                            \right]\left[
                 \begin{array}{cc}
                   I & 0 \\
                   0 & -I \\
                 \end{array}
               \right]) \\
               &=& \Cok_{T}\left[
                                              \begin{array}{cc}
                                                A^{\frac{q-1}{2}} & -zI \\
                                                zI &  A^{\frac{q+1}{2}} \\
                                              \end{array}
                                            \right].
\end{eqnarray*}
However, if $m$ is even, it follows that

\begin{eqnarray*}
    F_*^e(T/(f+z^2))  &=&\Cok_{T}\left[
                                              \begin{array}{cc}
                                                -A^{\frac{q-1}{2}} & zI  \\
                                                 zI &  A^{\frac{q+1}{2}} \\
                                              \end{array}
                                            \right]\\
  &=& \Cok_{T}(\left[
                 \begin{array}{cc}
                   -I & 0 \\
                   0 & I \\
                 \end{array}
               \right]\left[
                                              \begin{array}{cc}
                                                -A^{\frac{q-1}{2}} & zI  \\
                                                 zI &  A^{\frac{q+1}{2}} \\
                                              \end{array}
                                            \right])
   \\
               &=& \Cok_{T}\left[
                                              \begin{array}{cc}
                                                A^{\frac{q-1}{2}} & -zI \\
                                                zI &  A^{\frac{q+1}{2}} \\
                                              \end{array}
                                            \right].
\end{eqnarray*}
\end{proof}
                                            %we conclude that the matrix $ M_{T}(f+z^2, e)$  is equivalent to  the following matrix

%\begin{equation*}
%\left[
%       \begin{array}{ccccccc}
 %        0 &   &  &   &   &  zI &  (-1)^{\frac{q-1}{2}}A^{\frac{q+1}{2}} \\
 %        0 & 0 &  &   &   & (-1)^{\frac{q+1}{2}}A^{\frac{q-1}{2}}  & zI \\
 %        I & 0 & 0 &   &   &   &  \\
 %          & I & 0  & 0 &   &   &   \\
  %         &   & \ddots & \ddots & \ddots &   &  \\
  %         &   &   &  I & 0 & 0 &   \\
  %         &   &   &   &  I & 0 & 0 \\
  %     \end{array}
 %    \right]
%\end{equation*}

%Therefore,

%\begin{equation*}
 % F_*^e(T/(f+z^2)) = \Cok_{T}\left[
  %                                            \begin{array}{cc}
  %                                              A^{\frac{q-1}{2}} & -zI \\
   %                                             zI &  A^{\frac{q+1}{2}} \\
   %                                           \end{array}
    %                                        \right]
%\end{equation*}
%\begin{eqnarray*}
 % F_*^e(T/(f+z^2)) &=& \Cok_{T}M_{T}(f+z^2, e) \\
  % &=& \Cok_{T}\left[
                     % \begin{array}{cc}
                     %   zI & (-1)^{\frac{q-1}{2}}A^{\frac{q+1}{2}} \\
                    %    (-1)^{\frac{q+1}{2}}A^{\frac{q-1}{2}} & zI\\
                     % \end{array}
                   % \right]
   %  \\
  % &=& \Cok_{T}\left[
                                           %   \begin{array}{cc}
                                            %    A^{\frac{q-1}{2}} & -zI \\
                                            %    zI &  A^{\frac{q+1}{2}} \\
                                            %  \end{array}
                                          %  \right]
%\end{eqnarray*}

\begin{proposition}\label{P21}
Let $u$ and $v$ be new variables on $S$ and let $L=S[u,v]$ if $S=K[x_1,\dots,x_n]$ (or $L=S[\![u,v]\!]$  if $S=K[\![x_1,\dots,x_n]\!]$). Let  $R^{\bigstar}=L/(f+uv)$ where $f \in S$. If $A$ is the matrix $M_S(f,e)$ and $I$ is the identity matrix in the ring $M_{r_e}(S)$,  then $$ F_*^e(L/(f+uv))= (R^{\bigstar})^{r_e} \bigoplus \bigoplus_{k=1}^{q-1}\Cok_LB_k $$ where  $ B_k = \begin{bmatrix}  A^k & -vI \\ uI & A^{q-k}  \end{bmatrix}$ for all $ 1\leq k \leq q-1$.
\end{proposition}

\begin{proof}
Recall  that $\mathfrak{D}=\{ F_*^e(ju^sv^t) \,|\, j\in \Delta_e , 0 \leq s,t \leq q-1 \}$ is a free basis of $F_*^e(L)$ as $L$-module.
We introduce a $\mathbb{Z}/q\mathbb{Z}$-grading on both $L$ and $F_*^e(L)$ as follows:
$L$ is concentrated in degree 0, while $\deg(F_*^e(x_i))=0$ for each $1 \leq i \leq n$, $\deg(F_*^e(u))=1$ and $\deg(F_*^e(v))=-1$.
We can now write
$F_*^e(L)= \bigoplus_{k=0}^{q-1}M_k$ where $M_k$ is the free $L$-submodule of $F_*^e(L)$
of homogeneous elements  of degree $k$, i.e. $M_k$ is  the $L$-submodule of $F_*^e(L)$ that is generated by the subset
$$\mathfrak{D}_k = \{F_*^e(ju^sv^t)\,|\, \deg(F_*^e(ju^sv^t))=k \}$$ of the basis $\mathfrak{D}$.
Note that
$\mathfrak{D}_0= \{F_*^e(ju^sv^s)\,|\, j\in\Delta_e , 0 \leq s \leq q-1 \} $,  and that for all $1 \leq k \leq q-1$
\begin{eqnarray*}
  \mathfrak{D}_k&=& \{F_*^e(ju^{k+r}v^{r})\,|\, j\in\Delta_e , 0 \leq r \leq q-k-1 \} \cup\\
    & &  \{F_*^e(ju^{r}v^{q-k+r})\,|\, j\in\Delta_e ,  0 \leq r \leq k-1 \}.
\end{eqnarray*}
Let $J$ be the ideal $(f+uv)L$. Since $\deg(F_*^e(f+uv))=0$, it follows that $F_*^e(J)= \bigoplus_{k=0}^{q-1}M_kF_*^e(f+uv)$ and consequently
\begin{equation}\label{EEEE1}
  F_*^e(L/(f+uv))= F_*^e(L)/F_*^e(J)=\bigoplus_{k=0}^{q-1}M_k/M_kF_*^e(f+uv).
\end{equation}

We now show that
$M_k/M_kF_*^e(f+uv) \cong \Cok_LC_k  $  where $C_0= \begin{bmatrix} (-1)^{q}A^{q-1} & uvI \\ I & A \end{bmatrix} $ and $ C_k = \begin{bmatrix} (-1)^{q-k+1}A^{q-k} & vI \\ uI & (-1)^{k+1}A^{k} \end{bmatrix}$ for all $1\leq k \leq q-1$.
Recall that if $M_s(f,e)=[f_{(i,j)}]$, then $F_*^e(jf)=\bigoplus_{i\in\Delta_e}f_{(i,j)}F_*^e(i)$ for all $j \in \Delta_e$.
Therefore,
 \begin{eqnarray}\label{Eq18}
  \nonumber  F_*^e(ju^sv^t(f+uv)) &=& F_*^e(jf)F_*^e(u^sv^t) + F_*^e(ju^{s+1}v^{t+1}) \\
     &=& \bigoplus_{i\in\Delta_e}f_{(i,j)}F_*^e(iu^sv^t) \oplus F_*^e(ju^{s+1}v^{t+1}).
 \end{eqnarray}

Since $\deg(F_*^e(iu^sv^t))=\deg(F_*^e(ju^{s+1}v^{t+1}))$ for all $i,j  \in \Delta_e$ and all $0 \leq s,t \leq q-1$, it follows that $ F_*^e(ju^sv^t(f+uv)) \in M_k$ for all $F_*^e(ju^sv^t) \in \mathfrak{D}_k$. This enables us to define the homomorphism  $\psi_k :M_k\rightarrow M_k$ that is  given by $\psi_k(F_*^e(ju^sv^t))=F_*^e(ju^sv^t(f+uv))$ for all $F_*^e(ju^sv^t) \in \mathfrak{D}_k$ and consequently we have the following short exact sequence   $$M_k\xrightarrow{\psi_k}M_k\xrightarrow{\phi_k}M_k/M_kF_*^e(f+uv) \xrightarrow{} 0$$
where $\phi_k :M_k\rightarrow M_k/M_kF_*^e(f+uv)$ is the canonical surjection.
Notice that if $0 \leq  s < q-1$, equation (\ref{Eq18}) implies that
\begin{equation}\label{Eq19}
 F_*^e(ju^sv^s(f+uv))=  \bigoplus_{i\in\Delta_e}f_{(i,j)}F_*^e(iu^sv^s) \oplus F_*^e(ju^{s+1}v^{s+1})
\end{equation} and

 \begin{equation}\label{Eq20}
 F_*^e( ju^{q-1}v^{q-1}(f+uv))=  \bigoplus_{i\in\Delta_e}f_{(i,j)}F_*^e(iu^{q-1}v^{q-1}) \oplus uv F_*^e(j) ,
\end{equation}
therefore $\psi_0$ is represented by the matrix $\begin{bmatrix} A & & &uvI \\ I & A & & \\ & \ddots & \ddots & \\ & & I & A\end{bmatrix}$ which is a $q\times q$ matrix over the ring $M_{r_e}(L).$
Now Corollary \ref{L.20} implies that
\begin{equation}
 M_0/M_0F_*^e(f+uv)\cong \Cok_L \begin{bmatrix} (-1)^{q}A^{q-1} & uvI \\ I & A \end{bmatrix}.
\end{equation}

However,
\begin{equation*}
 \begin{bmatrix} (-1)^{q}A^{q-1} & uvI \\ I & A \end{bmatrix}\sim \begin{bmatrix} 0 &A(A^{q-1})+ uvI \\ I & A \end{bmatrix}\sim \begin{bmatrix} A(A^{q-1})+ uvI & 0\\ 0 & I \end{bmatrix}.
\end{equation*}
Since $AA^{q-1}=fI$ as $(A,A^{q-1})$ is a matrix factorization of $f$ (Proposition \ref{P5.5}), it follows that

\begin{equation}\label{EEEE2}
 M_0/M_0F_*^e(f+uv)\cong \Cok_L \begin{bmatrix} (f+uv)I & 0 \\ 0 & I \end{bmatrix}=(R^{\bigstar})^{r_e}.
\end{equation}

Now let $1 \leq k \leq q-1$. If $0 \leq r < q-k-1$, then  it follows from equation (\ref{Eq18}) that
\begin{equation}\label{Eq21}
 F_*^e( ju^{k+r}v^r(f+uv))=  \bigoplus_{i\in\Delta_e}f_{(i,j)}F_*^e(iu^{k+r}v^{r}) \oplus F_*^e(ju^{k+r+1}v^{r+1})
\end{equation}
and
\begin{equation}\label{Eq22}
 F_*^e( ju^{q-1}v^{q-k-1}(f+uv))= \bigoplus_{i\in\Delta_e}f_{(i,j)}F_*^e(iu^{q-1}v^{q-k-1}) \oplus u F_*^e(jv^{q-k}).
\end{equation}
However, if $0 \leq r < k-1$, it follows from (\ref{Eq18}) that
\begin{equation}\label{Eq222}
 F_*^e( ju^{r}v^{q-k-r}(f+uv))= \bigoplus_{i\in\Delta_e}f_{(i,j)}F_*^e(iu^{r}v^{q-k-r}) \oplus  F_*^e(ju^{r+1}v^{q-k-r+1})
\end{equation}
and
\begin{equation}\label{Eq23}
 F_*^e( ju^{k-1}v^{q-1}(f+uv))=\bigoplus_{i\in\Delta_e}f_{(i,j)}F_*^e(iu^{k-1}v^{q-1}) \oplus vF_*^e(ju^{k}).
\end{equation}

As a result, $\psi_k$  is represented by the matrix  $ \left[ \begin{array}{cccc|cccc} A & & & & & & &  vI  \\
                                    I & A & & & & & &   \\
                                      &  \ddots & \ddots &  & & & &  \\
                                      & & I& A & & & &  \\ \hline
                                      & & &uI &    A & & &  \\
                                      & & & &  I & A & &  \\
                                      & & & &   &  \ddots & \ddots &  \\
                                      & & & &   & &I & A
                                        \end{array}\right] $ which is a $q\times q$ matrix over the ring $M_{r_e}(L)$  where   $uI$ is in the

                                         $(q-k+1,q-k)$ spot of this matrix.

Therefore, Corollary \ref{L.20} implies that
\begin{eqnarray}
 \nonumber M_k/M_kF_*^e(f+uv) &\cong & \Cok_L\begin{bmatrix} (-1)^{q-k+1}A^{q-k} & vI \\ uI & (-1)^{k+1} A^{k} \end{bmatrix}.
\end{eqnarray}
If $k$ is odd integer, we notice that

\begin{eqnarray}\label{EEq}
 \nonumber  M_k/M_kF_*^e(f+uv)  &\cong & \Cok_L( \left[
                                                                                              \begin{array}{cc}
                                                                                                -I & 0 \\
                                                                                                0 & I \\
                                                                                              \end{array}
                                                                                            \right] \begin{bmatrix} (-1)^{q-k+1}A^{q-k} & vI \\ uI & (-1)^{k+1} A^{k} \end{bmatrix}) \\
   &\cong & \Cok_L \begin{bmatrix} A^{q-k} & -vI \\ uI &  A^{k} \end{bmatrix}.
\end{eqnarray}

Similar argument when  $k$ is even shows that
\begin{equation}\label{Eq24}
  M_k/M_kF_*^e(f+uv)  \cong \Cok_L \begin{bmatrix} A^{q-k} & -vI \\ uI &  A^{k} \end{bmatrix}.
\end{equation}

Now the proposition follows from (\ref{EEEE1}), (\ref{EEEE2}), (\ref{EEq}) and (\ref{Eq24}).
\end{proof}

If $(\phi,\psi)$ is a matrix factorization of a nonzero nonunit element $f$ in a domain $S$ and $u,v,z$ are variables on $S$, recall from Remark \ref{R3.8} that
\begin{equation*}
 (\phi, \psi)^{\maltese}:= ( \begin{bmatrix}  \phi & -vI \\  uI & \psi \end{bmatrix} , \begin{bmatrix}  \psi & vI \\  -uI & \phi \end{bmatrix}), \text{ and }(\phi, \psi)^{\sharp}:= ( \begin{bmatrix}  \phi & -zI \\  zI & \psi \end{bmatrix} ,
\begin{bmatrix}  \psi & zI \\  -zI & \phi \end{bmatrix})
\end{equation*}
Furthermore, $(\phi, \psi)^{\maltese}$   is a matrix factorization of $f+uv$ in $S[\![u,v]\!]$ (and in $S[u,v]$)
and $ (\phi, \psi)^{\sharp}$ is a matrix factorization of $f+z^2$ in $S[\![z]\!]$ (and in $S[\![z]\!]$). Using this notation and the notation in Definition \ref{D2.11} we can establish the following proposition.
\begin{proposition}\label{C4.19}
Let  $S=K[\![x_1,\dots,x_n]\!]$  and $f \in \mathfrak{m} \setminus \{0\}$ where   $\mathfrak{m}$ is the maximal ideal of $S$. Let $R=S/fS$, $R^{\bigstar}= S[\![u,v]\!]/(f+uv)$, and $ R^{\sharp}= S[\![z]\!]/(f+z^2)$. If $A=M_S(f,e)$ and $1\leq k \leq q-1$, then
\begin{enumerate}
  \item [(a)] $\sharp(\Cok_{S[\![z]\!]}(A^k, A^{q-k})^{\sharp},R^{\sharp}) = \sharp(\Cok_S(A^k),R)+\sharp(\Cok_S(A^{q-k}),R)$.
  \item [(b)] If $K$ be a field of prime characteristic $p > 2$, it follows that\\
  $\sharp(F_*^e(R^{\sharp}),R^{\sharp}) = \sharp(\Cok_S(A^{\frac{q-1}{2}}),R)+\sharp(\Cok_S(A^{\frac{q+1}{2}}),R)$ and hence \\
  $\sharp(F_*^e(R^{\sharp}),R^{\sharp})=\sharp(F_*^e(S/f^{\frac{q-1}{2}}S),R)+\sharp(F_*^e(S/f^{\frac{q+1}{2}}S),R)$.
  \item [(c)] $\sharp(\Cok_{S[\![u,v]\!]}(A^k, A^{q-k})^{\maltese},R^{\maltese})= \sharp(\Cok_S(A^k),R)+\sharp(\Cok_S(A^{q-k}),R)$.
  \item [(d)] $\sharp ( F_*^e(R^{\bigstar}),R^{\bigstar})= r_e + 2 \sum_{k=1}^{q-1} \sharp ( \Cok_S(A^k),R)$.
\end{enumerate}
\end{proposition}

\begin{proof}
Notice that $ (A^{k},A^{q-k})$ is a matrix factorization of $f$ (Proposition \ref{P5.5}).

(a) If $ (A^{k},A^{q-k}) \sim  (f,1)^{r_e} $ (or $ (A^{k},A^{q-k})\sim  (1,f)^{r_e} $) then $(F_*^e( S/f^{q-k}S)=\Cok_S(A^{q-k})= \{0\}$ (or $(F_*^e( S/f^{k}S )=\Cok_S(A^{k})= \{0\}$) which is impossible. As a result, if $ (A^{k},A^{q-k})$ is a trivial matrix factorization of $f$, then the only possible
case is that $   (A^{k},A^{q-k}) \sim  (f,1)^u \oplus (1,f)^v$ where $0 < u,v < r_e$ with $u+v = r_e$ ,  $ \sharp(\Cok_S(A^{k}),R)= u$, and $ \sharp(\Cok_S(A^{q-k}),R)= v$. Remark \ref{R3.8} (j),(k),(e) and (f) implies that
\begin{eqnarray*}
  \Cok_{S[\![z]\!]}(A^k, A^{q-k})^{\sharp} &=& [\Cok_{S[\![z]\!]}(f, 1)^{\sharp}]^u\oplus[\Cok_{S[\![z]\!]}(1,f)^{\sharp}]^v  \\
   &=& [R^{\sharp}]^{u+v}.
\end{eqnarray*}
Therefore, if $ (A^{k},A^{q-k})$ is a trivial matrix factorization of $f$, it follows that
$$\sharp(\Cok_{S[\![z]\!]}(A^k, A^{q-k})^{\sharp},R^{\sharp}) = \sharp(\Cok_S(A^k),R)+\sharp(\Cok_S(A^{q-k}),R).$$

On the other hand,  if $ (A^{k},A^{q-k})$ is not trivial matrix factorization of $f$, it follows from Corollary \ref{C3.8} that
 \begin{equation*}
 \sharp(\Cok_{S[\![z]\!]}(A^k, A^{q-k})^{\sharp},R^{\sharp}) = \sharp(\Cok_S(A^k),R)+\sharp(\Cok_S(A^{q-k}),R).
 \end{equation*}

 (b) follows from the result (a) above, Proposition \ref{L4.13} and the fact that $\Cok_SA^k=\Cok_SM_S(f^k,e)=F_*^e(S/f^kS)$ for all $1\leq k \leq q-1$ (Corollary\ref{C4.12} and Proposition \ref{P4.3}(c)).

 (c) can be proved similarly to (a).

%We shall need the following corollary later
%\begin{corollary}\label{C4.19}
%Let  $S=K[\![x_1,\dots,x_n]\!]$ and let $\mathfrak{m}$ be the maximal ideal of $S$.  Let  $R=S/fS$ where $f \in \mathfrak{m} \setminus \{0\}$. Let  $R^{\bigstar}= S[\![u,v]\!]/(f+uv)$. If $A=M_S(f,e)$, then
%\begin{equation*}
 % \sharp ( F_*^e(R^{\bigstar}),R^{\bigstar})= r_e + 2 \sum_{k=1}^{q-1} \sharp ( \Cok_S(A^k),R)
 %\end{equation*}

%\end{corollary}

%\begin{proof}

(d) By Proposition \ref{P21}, it follows that
\begin{equation*}
  F_*^e(R^{\bigstar})=  (R^{\bigstar})^{r_e} \bigoplus \bigoplus_{j=1}^{q-1} \Cok_{S[\![u,v]\!]} \begin{bmatrix}  A^k & -vI \\ uI & A^{q-k}  \end{bmatrix}.
\end{equation*}

However, For each $1 \leq k \leq q-1 $ , we recall from Remark \ref{R3.8} (e) that

\begin{equation*}
  \Cok_{S[\![u,v]\!]}\begin{bmatrix}  A^k & -vI \\ uI & A^{q-k}  \end{bmatrix}=\Cok_{S[\![u,v]\!]}(A^k,A^{q-k})^{\maltese}.
\end{equation*}

Therefore, by Krull-Remak-Schmidt theorem (Discussion \ref{disc2.33}), the result (c) above  and the convention  that $\Cok_{S}(A^k,A^{q-k})=\Cok_{S}(A^k)$ it follows that

\begin{eqnarray*}
   \sharp ( F_*^e(R^{\bigstar}),R^{\bigstar}) &=& r_e +  \sum_{k=1}^{q-1}\sharp (\Cok_{S[[u,v]]}(A^k,A^{q-k})^{\maltese}, R^{\bigstar}) \\
    &=&  r_e +  \sum_{k=1}^{q-1} \left[\sharp(\Cok_{S}(A^k,A^{q-k}),R)+ \sharp(\Cok_{S}(A^{q-k},A^k), R)\right] \\
    &=&  r_e + 2 \sum_{k=1}^{q-1} \sharp (\Cok_{S}(A^k,A^{q-k}),R)\\
    &=&  r_e + 2 \sum_{k=1}^{q-1} \sharp (\Cok_{S}(A^k),R).
\end{eqnarray*}
\end{proof}

\section{The ring $S[[y]]/(y^{p^d} +f)$ has finite F-representation type}
\label{section:The ring pd has finite Frepresentation type}
We  keep the same notation as in  Notation \ref{N4.1}. The following proposition can be obtained  as a special case from \cite[Theorem 3.10]{TS}. However, we provide   a different proof  that is based basically on the Theorem \ref{P5.11}.

\begin{theorem}\label{P5.1}
 If $d \in \mathbb{N}$ and $S:= K[\![x_1,...,x_n]\!]$,
   then $S[\![y]\!]/(y^{p^d}+f)$ has FFRT for any $f \in S$ and any prime integer $p>0$.
\end{theorem}

\begin{proof}
 Let   $A$ be the matrix  $M_S(f,e)$ in $M_{r_e}(S)$  where $ e > d$ and let $I$ be the identity matrix in $M_{r_e}(S)$.
Let $M$ and $N$ be the diagonal matrices of size $p^d \times p^d$ with entries in  $M_{r_e}(S)$  with $A$ and $I$   along the diagonals of $M$ and $N$ as follows:  $$M= \begin{bmatrix} A & & & & &  \\
                          & A & & & & \\
                          &   & A & & & \\
                          & &   & \ddots & & \\
                         & & & & A  \end{bmatrix}  $$ and $$N= \begin{bmatrix} I & & & & &  \\
                          & I & & & & \\
                          &   & I & & & \\
                          & &   & \ddots & & \\
                         & & & & I  \end{bmatrix}.  $$

                        Using Proposition \ref{EP} we write $M_{S[\![y]\!]}(y^{p^d}+f,e)$ as the following $p^{e-d} \times p^{e-d}$ matrix over the ring $M_{p^d}(M_{r_e}(S[\![y]\!]))$:
                          $$ M_{S[\![y]\!]}(y^{p^d}+f,e) = \begin{bmatrix} M & &  & & yN \\
                       N & M & & & & \\
                          & N & M & & & \\
                          & & \ddots & \ddots & & \\
                         & & &N& M   \end{bmatrix} . $$

Using Lemma  \ref{1}, and Theorem  \ref{P5.11} we see that
\begin{equation}\label{Eq5.14}
  F_*^e(S[\![y]\!]/(y^{p^d}+f)) \cong \Cok_{S[\![y]\!]}( yN+(-1)^{1+p^{e-d}}M^{p^{e-d}})\cong (\Cok_{S[\![y]\!]}(yI+A^{p^{e-d}}))^{\oplus p^d}.
\end{equation}
We aim to prove the  existence of finitely generated $S[\![y]\!]/(y^{p^d}+f) $-modules $M_1,...,M_s$  such that $\Cok_{S[\![y]\!]}(yI+A^{p^{e-d}})$ can be written as a direct sum with direct summands taken from $\{M_1,...,M_s\}$ for every $e$. Notice from Remark \ref{R2.29}(a) and Theorem  \ref{P5.11} that
\begin{eqnarray*}
 F_*^d[ F_*^{e-d}( S[\![y]\!]/(y^{p^d}+f)^{p^{e-d}}))] &\cong & F_*^e( S[\![y]\!]/(y^{p^e}+f^{p^{e-d}}))  \\
  &\cong &\Cok_{S[\![y]\!]}(M_{S[\![y]\!]}(y^{p^e}+f^{p^{e-d}},e)).
 \end{eqnarray*}
However, Proposition \ref{P4.3} and Proposition \ref{EP} imply that
$$ M_{S[\![y]\!]}(y^{p^e}+f^{p^{e-d}},e) = (yI+A^{p^{e-d}})^{\oplus p^e}.$$
As a result, we get
$$  F_*^d[ F_*^{e-d}( S[\![y]\!]/(y^{p^d}+f)^{p^{e-d}}))]  \cong  [ \Cok_{S[\![y]\!]}(yI+A^{p^{e-d}})]^{\oplus p^e}.$$

If $r_{e-d}=[K:K^p]^{e-d}p^{(e-d)(n+1)}$, let $\tilde{I}$ be the identity matrix in $M_{r_{e-d}}(S[\![y]\!])$.
By Theorem  \ref{P5.11} and Remark \ref{Ex4.3}, it follows that

\begin{eqnarray*}
F_*^{e-d}( S[\![y]\!]/(y^{p^d}+f)^{p^{e-d}}) &\cong & \Cok_{S[\![y]\!]}(M_{S[\![y]\!]}((y^{p^d}+f)^{p^{e-d}},e-d)  \\
 &\cong &\Cok_{S[\![y]\!]}((y^{p^d}+f)\tilde{I}) \\
   &\cong & [ S[\![y]\!]/(y^{p^d}+f)]^ {\oplus ^{r_{e-d}}}
 \end{eqnarray*}

                           and hence   $$  F_*^e( S[\![y]\!]/(y^{p^e}+f^{p^{e-d}}))=  [F_*^d(S[\![y]\!]/(y^{p^d}+f)]^ {\oplus ^{r_{e-d}}}.$$
                        This makes
                        \begin{equation}\label{Eq5.15}
                         [ \Cok_{S[\![y]\!]}(yI+A^{p^{e-d}})]^{\oplus p^e}\cong[F_*^d(S[[y]]/(y^{p^d}+f))]^ {\oplus ^{r_{e-d}}}
                        \end{equation}
                        as $S[\![y]\!]/(y^{p^d}+f))$-modules. Since  $ F_*^d(S[\![y]\!]/(y^{p^d}+f))$ can be written as a direct sum with direct summands taken from a finite set of  indecomposable finitely generated $S[\![y]\!]/(y^{p^d}+f))$-modules, say $M_1 , ... , M_s$, it follows from Krull-Remak-Schmidt theorem (Discussion \ref{disc2.33}) that $\Cok_{S[\![y]\!]}(yI+A^{p^{e-d}})$  is also a  direct sum with direct summands taken from $M_1 , ... , M_s$.
\end{proof}
                         %is finitely generated $S[[y]]$-module, it follows that  $ F_*^d(S[[y]]/(y^{p^d}+f))$ is a  direct sum with direct summands taken from a finite set of  indecomposable finitely generated $S[[y]]$-modules, say $M_1 , ... , M_s$ . By Krull-Remak-Schmidt theorem, the $S[[y]]$-module  Therefore, $ F_*^e(S[[y]]/(y^{p^d}+f))$ is a  direct sum with direct summands taken from $M_1 , ... , M_s$ as $S[[y]]$-module. Since $(y^{p^d}+f)F_*^d(S[[y]]/(y^{p^d}+f))=0$, we get that $M_1 , ... , M_s$ are all finitely generated $S[[y]]/(y^{p^d}+f)$-modules.    This shows that $S[[y]]/(y^{p^d}+f)$ has FFRT.

% Now note that
%$$(\Cok_{S[[y]]}(yI+A^{p^{e-d}}))=\Cok M(y^{p^e}+f^{p^{e-d}})\cong F_*^e ( S[\![ y ]\!] / (y^{p^e}+f^{p^{e-d}})) \cong $$
%$$F_*^d F_*^{e-d}  S[\![ y ]\!] / (y^{p^d}+f)^{p^{e-d}} \cong F_*^d \frac{F_*^{e-d}   S[\![ y ]\!] }{ (y^{p^d}+f) F_*^{e-d} S[\![ y ]\!] } $$
%Let  $I_{p^{e-d}}$ be the identity matrix in $M_{p^{e-d}}(S[[y]])$. Using by Proposition \ref{P4.11} and Remark \ref{Ex4.3} we write

%$$F_*^{e-d}( S[[y]]/(y^{p^d}+f)^{p^{e-d}}) \cong \Cok_{S[[y]]}(M_{S[[y]]}((y^{p^d}+f)^{p^{e-d}},e-d) \cong  \Cok_{S[[y]]}((y^{p^d}+f)I_{p^{e-d}})\cong $$
%$$ [ S[[y]]/(y^{p^d}+f)]^ {\oplus ^{p^{(e-d)(n+1)}}} $$
%Applying the functor $F_*^d(-)$ we see that
%$$F_*^e S[[y]]/(y^{p^d}+f)\cong \left(F_*^d S[[y]]/(y^{p^d}+f) \right)^{\oplus p^{(e-d)(n+1)}}$$

\section{When does the ring $S[[u,v]]/(f+uv)$ have finite F-representation type?}
\label{section:When does the ring fuv have finite F-representation type?}

We  keep the same notation as in Notation \ref{N4.1}. The purpose of this section is to provide a characterization of when the ring $S[\![u,v]\!]/(f+uv)$ has finite F-representation type. This characterization enables us to exhibit a class of rings in section \ref{section:Class of rings that have FFRT but not finite CM type}  that have FFRT but not finite CM type.

\begin{theorem}\label{P30}
Let $K$ be an algebraically closed field of prime characteristic $p > 2$ and $q=p^e$. Let  $S:= K[\![x_1,\dots,x_{n}]\!]$ and let $ \mathfrak{m}$ be the maximal ideal of $S$ and $f \in \mathfrak{m}^2 \setminus \{0\}$. Let $R=S/(f)$ and $R^{\bigstar}=S[\![u,v]\!]/(f+uv)$. Then $R^{\bigstar}=S[\![u,v]\!]/(f+uv)$ has FFRT over $R^{\bigstar}$ if and only if there exist indecomposable  $R$-modules $N_1,\dots,N_t$ such that $F_*^e(S/(f^k))$  is  a direct sums with direct summands taken from $R, N_1,\dots,N_t $ for every $e\in \mathbb{N}$ and  $1 \leq k < p^e $ .
\end{theorem}
\begin{proof}
First, suppose that the ring $R^{\bigstar}=S[\![u,v]\!]/(f+uv)$ has FFRT over $R^{\bigstar}$ by $\{R^{\bigstar} , M_1, \dots , M_t , \}$ where $M_j$ is an indecomposable non-free MCM $R^{\bigstar}$-module for all $j$. Therefore, there exist a nonnegative integers $n_{(e)},n_{(e,1)},\ldots,n_{(e,t)}$ such that

\begin{equation*}
 F_*^e(S[\![u,v]\!]/(f+uv))=( R^{\bigstar})^{n_{(e)}}\bigoplus \bigoplus_{j=0}^{t}M_i^{n_{(e,j)}}.
\end{equation*}
 %Since $F_*^e(R^{\bigstar})$ is MCM $R^{\bigstar}$-module (by Proposition \ref{P4.11}), each  $M_j$ is a non-free indecomposable MCM $R^{\bigstar}$-module for every $j\in \{1,\dots,t \}$.
  By Proposition \ref{P28} and \ref{A},  it follows  that $M_j= \Cok_{S[\![u,v]\!]}(\alpha_j,\beta_j)^{\maltese}$ where $(\alpha_j,\beta_j)$ is a reduced matrix factorization of $f$ such that $\Cok_S(\alpha_j,\beta_j)$ is non-free indecomposable MCM $R$-module for all $j$. Let $A=M_S(f,e)$ and consequently by Proposition \ref{P4.3}  $A^k= (M_S(f,e))^k=M_S(f^k,e)$.  Recall from the  Proposition \ref{P21} and the notation in Remark \ref{R3.8} (e) that
  \begin{equation}\label{Eq12}
    F_*^e(S[\![u,v]\!]/(f+uv))= ( R^{\bigstar})^{r_e} \oplus \bigoplus \limits_{k=1}^{q-1}\Cok_{S[\![u,v]\!]}(A^k,A^{q-k})^{\maltese }.
  \end{equation}

  This makes
 \begin{equation}\label{E13}
  ( R^{\bigstar})^{n_{(e)}}\bigoplus \bigoplus_{j=0}^{t}M_i^{n_{(e,j)}} \cong( R^{\bigstar})^{r_e} \oplus \bigoplus \limits_{k=1}^{q-1}\Cok_{S[\![u,v]\!]}(A^k,A^{q-k})^{\maltese }.
 \end{equation}
 %Recall that $\Cok_{S[[u,v]]}(\begin{bmatrix}  A^k  & -vI \\  uI & A^{q-k} \end{bmatrix} )= \Cok_{S[[u,v]]}(A^k,A^{q-k})^{\maltese }$ for each $k$. Since $(A^k,A^{q-k})$ is a matrix factorization of $f$,
  If $(A^k,A^{q-k})$ is nontrivial matrix factorization of $f$, by Proposition \ref{P.24}, there exist  a reduced matrix factorization $(\phi_k , \psi_k)$ of $f$  and non-negative integers $t_k$ and $r_k$ such that $(A^k,A^{q-k})\sim (\phi_k , \psi_k) \oplus (f,1)^{t_k} \oplus (1,f)^{r_k}$. This gives  by Remark \ref{R3.8} (g), (h), and (i)  that  $\Cok_{S[[u,v]]}(A^k,A^{q-k})^{\maltese }= \Cok_{S[[u,v]]}(\phi_k , \psi_k)^{\maltese} \oplus [R^{\bigstar}]^{t_k+r_k} $.
   %Since $R^{\bigstar}$ has FFRT over $R^{\bigstar}$ and $R^{\bigstar}$  is a complete local ring ,by the equality \ref{E13} and the
  By the equation \ref{E13} and Krull-Remak-Schmidt theorem (Discussion \ref{disc2.33}),  there exist  non-negative integers $n_{(e,k,1)},\dots., n_{(e,k,t)}$ such that

\begin{eqnarray*}
 \Cok_{S[\![u,v]\!]}(\phi_k , \psi_k)^{\maltese } &\cong& \bigoplus\limits_{j=1}^{t}M_j^{n_{(e,k,j)}} \\
    &\cong& \bigoplus\limits_{j=1}^{t}[\Cok_{S[\![u,v]\!]}(\alpha_j,\beta_j)^{\maltese}]^{ \oplus n_{(e,k,j)}}\\
   &\cong& \Cok_{S[\![u,v]\!]}[\bigoplus_{j=1}^{t}(\alpha_j,\beta_j)^{ \oplus n_{(e,k,j)}}]^{\maltese}.
\end{eqnarray*}

   Now, from Proposition \ref{P3.15} and Remark \ref{R3.8} (d) it follows that $$ \Cok_{S}(\phi_k , \psi_k) \cong \Cok_{S}[\bigoplus_{j=1}^{t}(\alpha_j,\beta_j)^{ \oplus n_{(e,k,j)}}] \cong \bigoplus_{j=1}^{t} N_j^{ \oplus n_{(e,k,j)}}$$ where $N_j$ denotes the non-free indecomposable MCM $R$-module $ \Cok_S(\alpha_j,\beta_j)$  for all $j \in \{1,\dots,t\}$.  Therefore,

\begin{eqnarray*}
F_*^e(S/f^kS)  &\cong& \Cok_SM_S(f^k,e) \text{  (Theorem \ref{P5.11}) } \\
&=& \Cok_S(A^k) \text{  (Proposition \ref{P4.3} (c)) } \\
   &=& \Cok_S(A^k,A^{q-k}) \text{ (Definition \ref{D23})}  \\
    &=& \Cok_S[(\phi_k , \psi_k) \oplus (f,1)^{r_k} \oplus (1,f)^{t_k}] \\
    &=& R^{r_k}\oplus\bigoplus_{j=1}^{t}N_j^{ \oplus n_{(e,k,j)}}.
\end{eqnarray*}

However,  if $ (A,A^{q-1})\sim  (f,1)^{r_e} $ (or $ (A,A^{q-1})\sim  (1,f)^{r_e} $) then $(F_*^e( \frac{S}{f^{q-1}S})=\Cok_SA^{q-1}= \{0\}$ (or $(F_*^e( \frac{S}{fS})=\Cok_SA= \{0\}$) which is impossible. As a result, if $(A^k,A^{q-k})$ is a trivial matrix factorization of $f$, then the only possible case is that  $(A^k,A^{q-k})\sim (f,1)^{b}\oplus (1,f)^{c}$ where $0 < b,c < r_e$  with $b+c=r_e$. In this case, $ F_*^e(S/f^kS)= R^b$

     This shows that $F_*^e(S/(f^k))$, for every $e\in \mathbb{N}$ and  $1 \leq k<p^e $, is  a direct sums with direct summands taken from $\{ R , N_1,\dots,N_t\}$.

%With the convention that $ (\alpha , \beta )^{\oplus n}= (\alpha , \beta )\oplus \dots \oplus (\alpha , \beta ) $ when $(\alpha , \beta )$ appears $n$ times and by Remark \ref{R3.8} we obtain

%\begin{equation}\label{E34}
%  \Cok_{S[[u,v]]}(\phi_k, \psi_k)^{\maltese }\cong  \Cok_{S[[u,v]]}[\bigoplus_{j=1}^{t}(\alpha_j,\beta_j)^{ \oplus n_{(e,k,j)}}]^{\maltese}
%\end{equation}
% Since  $(\phi_k, \psi_k)$ and $ \bigoplus_{j=1}^{t}(\alpha_j,\beta_j)^{ \oplus n_{(e,k,j)}} $ are  reduced matrix factorizations of $f$ by Preposition \ref{P3.15} , Corollary \ref{C4.12}, and  the above equality, we conclude that $\Cok_S(\phi_k, \psi_k) \cong \bigoplus_{j=1}^{t}N_j^{ \oplus n_{(e,k,j)}}$ where $N_j= \Cok_S(\alpha_j,\beta_j)$ for all $j \in \{1,\dots,t\}$ . As a result,  it follows that
%\begin{equation}\label{E35}
%  \Cok_S(A^k,A^{q-k})=\Cok_S[(\phi_k , \psi_k) \oplus (f,1)^{r_k} \oplus (1,f)^{t_k}]=R^{r_k}\oplus\bigoplus_{j=1}^{t}N_j^{ \oplus n_{(e,k,j)}}
%\end{equation}
%for all $k \in \{1,\dots,q-1 \}$ . Since $ F_*^e(S/f^kS)=\Cok_S(A^k,A^{q-k})$ by Corollary \ref{C4.12}, we conclude that
%\begin{equation*}
%  F_*^e(S/f^kS)=R^{r_k}\oplus\bigoplus_{j=1}^{t}N_j^{ \oplus n_{(e,k,j)}}
%\end{equation*}
%for all $k \in \{1,\dots,q-1 \}$. This shows that $\{F_*^e(S/(f^k)) \, | \, 0\leq k<p^e \text{ and } e\in \mathbb{N}\}$ are direct sums with direct summands taken from a finite set of indecomposable $R$-modules. \\

Now suppose $F_*^e(S/(f^k))$  is  a direct sums with direct summands taken from indecomposable $R$-modules  $ N_1,\dots,N_t $ for every $e\in \mathbb{N}$ and  $1 \leq k<p^e $. Therefore, for each $1 \leq k \leq q-1$,  there exist  non-negative integers $n_{(e,k)},n_{(e,k,1)},\dots., n_{(e,k,t)}$ such that
\begin{equation}\label{E31}
 \Cok_S(A^k,A^{q-k}) \cong F_*^e(S/f^kS)\cong R^{\oplus n_{(e,k)}} \oplus \bigoplus\limits_{j=1}^{t}N_j^{\oplus n_{(e,k,j)}}  \text{ over }  R.
\end{equation}
Since $F_*^e(S/f^kS)$ is a MCM $R$-module  (by Corollary \ref{C4.12}), it follows that   $N_j$ is an indecomposable non-free MCM $R$-module for each $j \in \{1,\dots,t\}$ and hence
 by Proposition \ref{A} $N_j= \Cok_S(\alpha_j,\beta_j)$ for some reduced matrix factorization $(\alpha_j,\beta_j)$ for  all $j $. If  $M_j = \Cok_{S[\![u,v]\!]}(\alpha_j,\beta_j)^{\maltese}$, it follows that
 $$ \Cok_{S[\![u,v]\!]}(A^k,A^{q-k})^{\maltese}= (R^{\bigstar})^{n_{(e,k)}} \oplus \bigoplus\limits_{j=1}^{t}M_j^{n_{(e,k,j)}} $$ and hence by the  Equation \ref{Eq12}  we conclude  that the ring $R^{\bigstar}=S[\![u,v]\!]/(f+uv)$ has FFRT by $\{ R^{\bigstar},M_1,\dots, M_t\}$.
\end{proof}

 %.  Recall that $F_*^e(S/f^kS)=\Cok_S(A^k,A^{q-k})$  where $A=M_S(f,s)$ and $(A^k,A^{q-k})$ is a matrix factorization of $f$ for all $k \in \{1,\dots, q-1 \}$. By Proposition \ref{P3.15} and Remark \ref{R3.8}, we get that
%\begin{equation*}
%\Cok_{S[[u,v]]}(\begin{bmatrix}  A^k  & -vI \\  uI & A^{q-k} \end{bmatrix}) =\Cok_S(A^k,A^{q-k})^{\maltese}= (R^{\bigstar})^{n_{(e,k)}} \bigoplus\limits_{j=1}^{t}M_j^{n_{(e,k,j)}}
%\end{equation*}
%Since
%\begin{equation*}
%  F_*^e(S[[u,v]]/(f+uv)) \cong( R^{\bigstar})^{q^n} \bigoplus \bigoplus \limits_{k=1}^{q-1}\Cok_{S[[u,v]]}(\begin{bmatrix}  A^k  & -vI \\  uI & A^{q-k} \end{bmatrix} )
 %\end{equation*}
% we obtain that $R^{\bigstar}=S[[u,v]]/(f+uv)$ has FFRT by $M_1,\dots, M_t$ and $R^{\bigstar}$

The following result is a direct application of the above proposition.
\begin{corollary}
Let $K$ be an algebraically closed field of prime characteristic $p > 2$ and $q=p^e$. Let  $S:= K[\![x_1,\dots,x_{n}]\!]$ and let $ \mathfrak{m}$ be the maximal ideal of $S$ and $f \in \mathfrak{m}^2 \setminus \{0\}$. Let $R=S/(f)$ and $R^{\bigstar}=S[\![u,v]\!]/(f+uv)$. If $R^{\bigstar}=S[\![u,v]\!]/(f+uv)$ has FFRT over $R^{\bigstar}$, then  $S/f^kS$ has  FFRT over $S/f^kS$ for every positive integer $k$.

\end{corollary}

\begin{proof}
Let $k$ be a positive integer and let $e_0$ be a positive integer such that $k < p^{e_0}$. If $R^{\bigstar}=S[\![u,v]\!]/(f+uv)$ has FFRT over $R^{\bigstar}$, from Theorem \ref{P30}, there exist finitely generated $R$-modules $N_1,\dots,N_t$ such that $F_*^e(S/(f^k))$ for each $e \geq e_0$  is a  direct sums with direct summands taken from the finite  set $\{R,N_1,\dots,N_t\}$. Notice from the proof of Theorem \ref{P30} that $N_j=\Cok_S(\alpha_j,\beta_j)$ for some reduced matrix factorization $(\alpha_j,\beta_j)$ of $f$ for  all $1 \leq j \leq t $. As a result, $f\Cok_S(\alpha_j,\beta_j)=0$ and hence $f^k\Cok_S(\alpha_j,\beta_j)=0$ for every  positive integer $k$. This makes $N_j$ a  module over $S/f^kS$. Therefore,$F_*^e(S/(f^k))$ for each $e \geq e_0$  is a  direct sums with direct summands taken from the  finite   set $\{R,N_1,\dots,N_t\}$ of the $S/f^kS$-modules. This is enough to show that $S/f^kS$ has  FFRT over $S/f^kS$ for every positive integer $k$.
\end{proof}
%\begin{proof}
%Let  $e_0$ be a positive integer such that $ p^{e_0} > k$. By the above proposition, there exist finitely generated $R$-modules $N_1,\dots,N_t$ such that $F_*^e(S/(f^k))$ for each $e \geq e_0$  is a  direct sums with direct summands taken from the finite  set $\{R,N_1,\dots,N_t\}$. Notice from the proof of the above proposition that $N_j=\Cok_S(\alpha_j,\beta_j)$ for some reduced matrix factorization $(\alpha_j,\beta_j)$ for  all $j \in \{1,\dots,t\}$. As a result, $f\Cok_S(\alpha_j,\beta_j)=0$ and hence $f^k\Cok_S(\alpha_j,\beta_j)=0$ for every  positive integer $k$. This makes $N_j$ a $S/f^kS$-module for each $k$. Therefore,$F_*^e(S/(f^k))$ for each $e \geq e_0$  is a  direct sums with direct summands taken from the  finite   set $\{R,N_1,\dots,N_t\}$ of the $S/f^kS$-modules. This is enough to show that $S/f^kS$ has  FFRT over $S/f^kS$ for every positive integer $k$.
%\end{proof}

The above corollary implies evidently the following.
\begin{corollary}

Let $K$ be an algebraically closed field of prime characteristic $p > 2$ and $q=p^e$. Let  $S:= K[\![x_1,\dots,x_{n}]\!]$ and let $ \mathfrak{m}$ be the maximal ideal of $S$ and $f \in \mathfrak{m}^2 \setminus \{0\}$. Let $R=S/(f)$ and $R^{\bigstar}=S[\![u,v]\!]/(f+uv)$. If $S/f^kS$ does not have  FFRT over $S/f^kS$ for some positive integer $k$, then $R^{\bigstar}$ does not have FFRT. In particular, if $R$ does not have  FFRT , then $R^{\bigstar}$ does not have FFRT.

\end{corollary}
An easy induction gives the following result.

\begin{corollary}
Let $K$ be an algebraically closed field of prime characteristic $p > 2$ and $q=p^e$. Let  $S:= K[\![x_1,\dots,x_{n}]\!]$ and let $ \mathfrak{m}$ be the maximal ideal of $S$,  $f \in \mathfrak{m}^2 \setminus \{0\}$ and let $R=S/(f)$. If $R$ does not have  FFRT, then the ring  $$S[\![u_1,v_1,u_2,v_2,\dots,u_t,v_t]\!]/(f+u_1v_1+u_2v_2+\dots+u_tv_t)$$ does not have FFRT for all $t\in \mathbb{N}$.
\end{corollary}

%%%%%%%%%%%%%%%%%%%%%%%%%%%%%%%%%%%%%%%%%%%%%%%%%%%%%%%%%%%%%%%%%%%%%
\section{Class of rings that have FFRT but not finite CM type}
\label{section:Class of rings that have FFRT but not finite CM type}

We  keep the same notation as in Notation \ref{N4.1}.  Recall that every F-finite  local ring $(R,\mathfrak{m})$ of prime characteristic that has finite CM representation type has also FFRT (section \ref{section:Definition and examples}).   The main result of this section is to provide a class of rings that have FFRT but not finite CM representation type Theorem \ref{P7.10}.
\begin{lemma}\label{L31}
If $ 0 \leq d,k \leq q-1  $,  then
\begin{enumerate}
  \item [(a)]$ d k = nq + s$ where $0 \leq n \leq  d-1 $ and $  0 \leq s \leq q-1 $.
  \item [(b)]For any $   0 \leq \beta \leq  q-1  $, $ d k  + \beta = cq + t $ where $ 0 \leq c \leq d $ and $  0 \leq t \leq q-1  $.
  \item [(c)] Fix $c,d,k \in \mathbb{Z}_{+}$ such that $0 \leq d,k \leq q-1  $ and $ 0 \leq c \leq d $. Then there exists an $ \alpha \in \{0,\ldots, q-1\}$ such that $\alpha = qc-dk+ s$ for some $ s \in \{0,\ldots, q-1\}$ if and only if $|qc-dk|< q$. Furthermore, If $|qc-dk|< q$, there exist $q-|qc-dk|$ values of $\alpha $ such that $\alpha = qc-dk+ s$ for some $ s \in \{0,\dots, q-1\}$.
\end{enumerate}

\end{lemma}
\begin{proof}
(a) By Division Algorithm,  $ d k = nq + s $ for some  $ n \in \mathbb{N} $ and $ 0 \leq s \leq q-1 $. If $ n \geq d $ , we get  $ nq \geq dq > dk = nq + s \geq nq $ which is a contradiction. This shows that $ 0 \leq n \leq d-1 $. \\
(b) From the above result, we get that  $ d k = nq + s$ where $  0 \leq n \leq d-1  $ and $  0 \leq s \leq q-1 $. Let $  0 \leq \beta \leq q-1  $. If $ \beta + s  < q $, we get $ d k  + \beta = cq + t $ where $ c =n \in \{ 0 , ... , d-1 \} $ and $ t = \beta +s   \in \{ 0 , ... , q-1 \} $. Now, suppose that  $ \beta + s  \geq q $. We notice that $ 0 \leq \beta + s -q   \leq q -1$. Therefore $ d k  + \beta = cq + t $ where $ c = n+1  \in \{ 1 , ... , d \} $ and $ t = \beta + s -q   \in \{ 0 , ... , q-1 \} $.\\
(c) First, if there is $ \alpha \in \{0,\ldots, q-1\}$ such that $\alpha = qc-dk+ s$ for some $ s \in \{0,\ldots, q-1\}$, then   $|qc-dk|=|\alpha -s|< q$ (as $ \alpha,s \in \{0,\ldots, q-1\}$). Now let $u = qc-dk$ and suppose that $|u| < q$. If $ 0 \leq u < q$, we can choose $\alpha \in \{u, u+1,\dots,q-1 \}$. On the other hand, if    $ -q < u  < 0$, then $\alpha$ can be taken from $\{q-1+u , q-2+u ,\dots,0\}$. In both cases,   $\alpha$ can be chosen by $q-|c q - kd |$ ways.
\end{proof}

If $\alpha = (\alpha_1, \ldots , \alpha_n) \in \mathbb{Z_{+}}^n $, we write $ x^\alpha = x_1^{\alpha_1}\dots x_n^{\alpha_n}$ where $x_1,\dots,x_n$ are different variables.
\begin{proposition}\label{L4.25}
 Let  $f = x_1^{d_1}x_2^{d_2} \dots x_n^{d_n} $ be a monomial in $S$    where  $d_j \in \mathbb{Z_{+}}$ for each $j$. Let $\Gamma = \{(\alpha_1, \ldots , \alpha_n) \in \mathbb{Z_{+}}^n \, | \, 0 \leq \alpha_j\leq d_j \text{ for all  } 1\leq j \leq n \}$ , $d=(d_1,\ldots,d_n)$, and let $e$ be a positive integer such that $q=p^e > \max \{ d_1 , \dots , d_n \}+1$. If $A=M_S(f,e)$, then for each $1 \leq k \leq q-1$ the matrix $A^k=M_S(f^k,e)$ is equivalent to diagonal matrix, $D$,  of size   $r_e \times r_e$ in which the diagonal entries are of the form $x^c$  where $c \in \Gamma$. Furthermore, if $c=(c_1,\dots,c_n) \in \Gamma $ and
\begin{equation*}
   \eta_k(c_j)= \begin{cases}q- | c_jq - kd_j |  \text { if } | c_jq - kd_j | < q \\
0 \text{ otherwise }
\end{cases}
\end{equation*}
 then
\begin{equation*}
  \Cok_S(A^q,A^{q-k}) = \bigoplus_{c \in \Gamma } \left[ \Cok_S(x^c, x^{d-c})\right]^{\oplus \eta_k(c)}
\end{equation*}
where $ \eta_k(c) = [K: K^q]\prod_{j=1}^n \eta_k(c_j) $ and $(x^c, x^{d-c})$ is the $1 \times 1$ matrix factorization of $f$ with the convention that $M^{\oplus0}=\{0\}$ for any module $M$.
\end{proposition}

\begin{proof}
Choose $e \in \mathbb{N}$ such that $q=p^e > \max \{ d_1 , \dots , d_n \}+1$ and let  $1\leq k \leq q-1 $. If $j=\lambda x_1^{\beta_1}\dots x_n^{\beta_n}\in \Delta_e$, we get $F_*^e(jf^k)=F_*^e(\lambda x_1^{kd_1+\beta_1}\dots x_n^{kd_n+\beta_n})$. Since $d_j , k \in \{ 0 , \dots , q-1 \} $, Lemma \ref{L31} implies that  there exist $ 0 \leq c_{i} \leq d_i  $ and $ 0 \leq u_{i}\leq q-1 $ for each $1 \leq i \leq n$ such that   $ d_i k  + \beta_{i} = c_{i}q + u_{i} $ and hence $ F_*^e(jf^k) = x_1^{c_{1}}\dots x_n^{c_{n}}F_*^e(\lambda x_1^{u_{1}}\dots x_n^{u_{n}}) $.
Therefore each  column of   $M_S(f^k,e)$ contains only one non-zero element of the form $ x_1^{c_{1}}\dots x_n^{c_{n}} $ where $  0 \leq c_{i}\leq  d  $ for all $1 \leq i \leq n$. Notice that a row in $M_S(f^k,e)$  will contain two elements of the form $ x_1^{c_{1}}\dots x_n^{c_{n}} $ and $ x_1^{l_{1}}\dots x_n^{l_{n}} $ where $  0 \leq l_{j} ,c_{i}\leq  d  $ for all $1 \leq i \leq n$ if there exist   $\mu x_1^{\beta_1}\dots x_n^{\beta_n},\gamma x_1^{\sigma_1}\dots x_n^{\sigma_n}, \lambda x_1^{u_{1}}\dots x_n^{u_{n}}\in  \Delta_e$ such that
\begin{equation*}
  F_*^e((\mu x_1^{\beta_1}\dots x_n^{\beta_n})( x_1^{kd_1}\dots x_n^{kd_n})) = x_1^{c_{1}}\dots x_n^{c_{n}}F_*^e(\lambda x_1^{u_{1}}\dots x_n^{u_{n}}) \text{ and }
\end{equation*}
\begin{equation*}
  F_*^e((\gamma x_1^{\sigma_1}\dots x_n^{\sigma_n})( x_1^{kd_1}\dots x_n^{kd_n})) = x_1^{l_{1}}\dots x_n^{l_{n}}F_*^e(\lambda x_1^{u_{1}}\dots x_n^{u_{n}}).
\end{equation*}
This makes $\mu= \lambda= \gamma $, $\beta_i+kd_i= qc_i+u_i$, and $\sigma_i+kd_i= ql_i+u_i$ for all $1 \leq i \leq n$.
 Accordingly, $\beta_i- \sigma_i = q(c_i-l_i)$ for all $1 \leq i \leq n$. Since $0\leq \beta_i, \sigma_i\leq q-1$ and $0\leq c_i, l_i\leq d \leq  q-1$ for all $1 \leq i \leq n$, it follows that $\beta_i= \sigma_i$ and   $c_i=l_i$ and for all $1 \leq i \leq n$. This also shows  that each  row of   $M_S(f^k,e)$ contains only one non-zero element of the form $ x_1^{c_{1}}\dots x_n^{c_{n}} $ where $  0 \leq c_{i}\leq  d  $ for all $1 \leq i \leq n$. Since each  column and row  of   $M_S(f^k,e)$ contains only one non-zero element of the form $ x_1^{c_{1}}\dots x_n^{c_{n}} $ where $  0 \leq c_{i}\leq  d  $ for all $1 \leq i \leq n$,  using the row and column operations, the matrix $M_S(f^k,e)$ is equivalent to  a diagonal matrix, $D$,  of size   $r_e \times r_e$ in which the diagonal entries are of the form $ x_1^{c_1}\dots x_n^{c_n} $ where $  0 \leq c_{i}\leq  d_i  $ for all $1 \leq i \leq n$.  Now fix $c=(c_1,\dots,c_n) \in \Gamma$ and let $\eta(c)$ stand for  how many times $x^c$ appears as an element in the diagonal of $D$. It is obvious that $\eta(c)$  is exactly the same as the number of the $n$-tuples $(\alpha_1,\dots,\alpha_n)$ with $ 0 \leq \alpha_j \leq q-1 $   satisfying that
\begin{equation}\label{EEE2}
 F_*^e(\lambda x_1^{kd_1+\alpha_1}\dots x_n^{kd_n+\alpha_n}) = x_1^{c_1}\dots x_n^{c_n}F_*^e(\lambda x_1^{s_1}\dots x_n^{s_n})
\end{equation}
for some $ s_1,..,s_n \in \{0,\dots,q-1 \}$ for all $\lambda \in \Lambda_e$. Notice that  an $n$-tuple $(\alpha_1,\dots,\alpha_n)$ with $ 0 \leq \alpha_i \leq q-1 $  will satisfy (\ref{EEE2}) if and only if  $\alpha_i= c_iq - kd_i + s_i$  for some $0 \leq s_i \leq q-1 $  for all $1\leq i \leq n$. As a result, by Lemma\ref{L31},  there exists an  $n$-tuples $(\alpha_1,\dots,\alpha_n)\in \mathbb{Z}^n$ with $ 0 \leq \alpha_i \leq q-1 $ that   satisfies (\ref{EEE2}) if and only if $|c_iq - kd_i|  < q$  for all $1\leq i \leq n$. Furthermore, by Lemma\ref{L31}, for each $i\in \{1,\dots,n\}$, if $|c_iq - kd_i|  < q$, then there exist  $q-|c_iq - kd_i|$ values of $\alpha_i $ such that $\alpha_i = qc_i-d_ik+ s_i$ for some $ s_i \in \{0,\dots, q-1\}$.
Set
\begin{equation*}
   \eta_k(c_j)= \begin{cases}q- |c_iq - kd_i | \text { if } | c_iq - kd_i | < q \\
0 \text{ otherwise }
\end{cases}
\end{equation*}
 Thus we get that $ \eta_k(c) = [K: K^q]\prod_{j=1}^n \eta_k(c_i) $ and consequently we have
\begin{equation*}
 \Cok_S(A^q,A^{q-k}) = \bigoplus_{c \in \Gamma } \left[ \Cok_S(x^c, x^{d-c})\right]^{\oplus \eta_k(c)}
\end{equation*}
where $(x^c, x^{d-c})$ is the $1 \times 1$ matrix factorization of $f$ with the convention that $M^{\oplus0}=\{0\}$ for any module $M$.
\end{proof}
\begin{corollary}\label{C32}
Let $K$ be an algebraically closed   field of prime characteristic $p > 2$ and $q=p^e$. Let  $S:= K[\![x_1,\dots,x_{n}]\!]$, $f =x_1^{d_1}x_2^{d_2} \dots x_n^{d_n} $ where  $d_j \in \mathbb{N}$ for each $j$, and $d=(d_1,\ldots,d_n)$. Then $R^{\bigstar}=S[\![u,v]\!]/(f+uv)$ has FFRT over $R^{\bigstar}$. Furthermore, for every $e \in \mathbf{N}$ with $q=p^e > \max \{ d_1 , \dots , d_n \}+1$, $F_*^e(R^{\bigstar})$ has the following decomposition:
\begin{equation*}
 F_*^e(R^{\bigstar})= (R^{\bigstar})^{r_e} \bigoplus \bigoplus_{k=1}^{q-1}\left[\bigoplus_{c \in \Gamma } \left[\Cok_{S[\![u,v]\!]}(x^c,x^{d-c})^{\maltese}\right]^{\oplus \eta_k(c)} \right]
\end{equation*}
where $\eta_k(c)$ and $\Gamma$ as in the above Proposition.
\end{corollary}
\begin{proof}
Let $e \in \mathbb{N}$ with $q=p^e > \max \{ d_1 , \dots , d_n \}+1$ and let  $1 \leq k \leq q-1 $. Let $\Gamma $ and $ \eta_k(c)$ be as in the above Proposition. If $A=M_S(f,e)$,   it follows that
\begin{equation*}
  F_*^e(S/f^k) \cong \Cok_S(A^k,A^{q-k}) \cong \bigoplus_{c \in \Gamma } \left[ \Cok_S(x^c, x^{d-c})\right]^{\oplus \eta_k(c)}.
\end{equation*}
 If $ \mathfrak{M}= \{\Cok_S(x^c, x^{d-c})\, | \, c \in \Gamma \}\cup \{ F^j(S/f^i)\, |\, p^j \leq \max \{ d_1 , \dots , d_n \}  \text{ and }   0 \leq i \leq p^j\} $,
then $F_*^e(S/(f^k))$  is  a direct sums with direct summands taken from the finite set $\mathfrak{M}$ for every $e\in \mathbb{N}$ and  $1 \leq k<p^e $. By Theorem \ref{P30}  $R^{\bigstar}$ has FFRT.

Furthermore, we can describe explicitly the direct summands of $F_*^e(R^{\bigstar})$. Indeed, if $\hat{\Gamma}:= \{ c \in \Gamma \, | \, \eta_k(c)> 0 \text{ and } c \notin \{d,0\} \} $, it follows that  $$(A^k,A^{q-k}) \sim \bigoplus_{c \in \hat{\Gamma} } (x^c,x^{d-c})^{\oplus \eta_k(c)}\bigoplus (x^d,1)^{\oplus \eta_k(d)}\bigoplus (1,x^d)^{\oplus \eta_k(0)} $$ where $ \eta_k(d)$ (respectively  $ \eta_k(0)$) denotes how many times $x^d$ (respectively $1$) appears in $A^k$.
Recall  by Remark \ref{R3.8} that  $( A^k, A^{q-k} )^{\maltese}$ is a matrix factorization of $f+uv$ and
$$(A^k,A^{q-k})^{\maltese} \sim \bigoplus_{c \in \hat{\Gamma} } \left[(x^c,x^{d-c})^{\maltese}\right]^{\oplus \eta_k(c)}\bigoplus \left[(x^d,1)^{\maltese}\right]^{\oplus \eta_k(d)}\bigoplus \left[(1,x^d)^{\maltese}\right]^{\oplus \eta_k(0)} $$
Therefore
\begin{eqnarray*}
 \Cok_{S[\![u,v]\!]}(A^k,A^{q-k})^{\maltese}   &=& \bigoplus_{c \in \hat{\Gamma} } \left[\Cok_{S[\![u,v]\!]}(x^c,x^{d-c})^{\maltese}\right]^{\oplus \eta_k(c)}\\
 & & \bigoplus \left[\Cok_{S[\![u,v]\!]}(x^d,1)^{\maltese}\right]^{\oplus \eta_k(d)}  \\
& & \bigoplus \left[\Cok_{S[\![u,v]\!]}(1,x^d)^{\maltese}\right]^{\oplus \eta_k(0)}.
\end{eqnarray*}

By  Proposition  \ref{P21},  the above equation, and the convention that $M^{\oplus 0}=\{0\}$, we can write
\begin{equation*}
 F_*^e(S[\![u,v]\!]/(f+uv)) = (R^{\bigstar})^{q^n} \bigoplus \bigoplus_{k=1}^{q-1}\left[\bigoplus_{c \in \Gamma } \left[\Cok_{S[\![u,v]\!]}(x^c,x^{d-c})^{\maltese}\right]^{\oplus \eta_k(c)} \right].
\end{equation*}
\end{proof}

We   benefit from the proof of Proposition \ref{L4.25} above when we compute the $F$-signature in the next chapter.

The following theorem provides an example of rings that have FFRT but not finite CM type.

\begin{theorem}\label{P7.10}
Let $K$ be an infinite algebraically closed field with $\charact (K) > 2$,
 and let $S = K[\![x_1, \dots, x_d]\!]$ where $d > 2$.
 If  $ f \in S$ is a monomial of degree grater than $3$ and  $R^{\bigstar}=S[\![u,v]\!]/(f+uv)$,  then $R^{\bigstar}$ has FFRT but it does not have finite CM representation type.
\end{theorem}
\begin{proof}
Let $t$ be the degree of the monomial $f$ and let $\mathfrak{m}$ be the maximal ideal of $S$. Clearly, $t$ is the largest natural number satisfying $ f \in \mathfrak{m}^t - \mathfrak{m}^{t+1}$ and consequently the multiplicity $e(R)$ of the ring $R$ is  $e(R)=t$ (Proposition \ref{F1}). Since $e(R)=t > 3$,  it follows from Proposition \ref{F2} that $R$ is not a simple singularity.  Therefore, by  Proposition \ref{F3} $R$ does not have finite CM type. Consequently,  by Proposition \ref{P28}, $R^{\bigstar}$ does not have finite CM type as well. However,  Corollary \ref{C32} implies that  $R^{\bigstar}$ has FFRT.
\end{proof}

\section{$S/I$ has FFRT when $I$ is a monomial ideal}
\label{section:Qutient by   a monomial ideal}

We  keep the same notation as in in Notation \ref{N4.1}.

If $x_1,...,x_n$ are variables, a monomial
in $x_1, . . . , x_n $ is an element of the form $x_1^{r_1}...x_n^{r_n}$ where $r_1,...,r_n\in \mathbb{Z}_{+}$. If $R = A[x_1, . . . , x_n]$ (or $R = A[\![x_1, . . . , x_n]\!]$) where $A$ is a nonzero commutative ring with identity. A monomial ideal in $R$ is an ideal
of $R$ that can be generated by monomials in $x_1, . . . , x_n$.

We need the following Proposition in order to prove Proposition \ref{P5.24}.

\begin{proposition}\label{P4.19}
Let $f_1,...,f_t$ be nonzero and non unite elements in $S$ and let $R= S/( f_1,...,f_t )S$. If $[ M_S(f_1,e)... M_S(f_t,e)]$ is the $r_e \times tr_e$ matrix over $S$ whose columns are the columns of the matrices $M_S(f_1,e),..., M_S(f_t,e)$ respectively, then :\\
1) $F_*^e(R)$ is isomorphic to $\Cok_S[ M_S(f_1,e)... M_S(f_t,e)]$ as $S$ -modules.\\
2) $F_*^e(R)$ is isomorphic to $\Cok_S[ M_S(f_1,e)... M_S(f_t,e)]$ as $R$ -modules.\\

\end{proposition}
\begin{proof}
 Let $I$ be the ideal $( f_1,...,f_t )S$. Since $ \{ F_*^e(j) | j \in \Delta_e \} $ is a basis of $F_*^e(S)$ as free $S$-module, $F_*^e(R)$ is  generated  by  $ \{ F_*^e(j + I) \,| \,j \in \Delta_e \} $ as $S$-module. For every $ F_*^e(g) \in F_*^e(S)$, define $ \phi( F_*^e(g)) = F_*^e(g + I)$. Then $\phi: F_*^e(S) \longrightarrow F_*^e(R)$ is a surjective homomorphism of $S$-modules. For every $1 \leq k \leq t$ recall that  $M_S(f_k,e)=[f^{(k)}_{(i,j)}]$, where $f^{(k)}_{(i,j)}$,  indexed by $i,j \in \Delta_e$, satisfies that $F_*^e(jf_k)= \bigoplus_{i\in\Delta_e}f^{(k)}_{(i,j)}F_*^e(i)$. Now, define the $S$-module homomorphism
 $ \psi : F_*^e(S)^{\oplus t} \rightarrow F_*^e(S)$ by
\begin{equation*}
  \psi [(F_*^e(g_1) , \ldots , F_*^e(g_t) )] = F_*^e(g_1f_1)+ \ldots +F_*^e(g_tf_t)
\end{equation*}
for all $(F_*^e(g_1) , \ldots , F_*^e(g_t) )\in F_*^e(S)^{\oplus t} $. Since $\Ima \psi = \Ker \phi= F_*^e(I)$, we have an exact sequence $F_*^e(S)^{\oplus t}\xrightarrow{\psi}F_*^e(S)\xrightarrow{\phi}F_*^e(R) \xrightarrow{} 0$. For every $j \in \Delta_e $ and $1 \leq k \leq t$ define $j^{(k)}$ to be the element in $ F_*^e(S)^{\oplus t} $ whose  $k$th coordinate is $F_*^e(j)$ and zero elsewhere and let $\Omega_e^{(k)}=\{ j^{(k)} \,|\, j \in \Delta_e\}$. Since $ \{ F_*^e(j)\, |\, j \in \Delta_e \} $ is a basis of $F_*^e(S)$ as free $S$-module, it follows that $\Omega_e=\Omega_e^{(1)}\cup...\cup \Omega_e^{(t)} $ is a basis for $ F_*^e(S)^{\oplus t}$ as free $S$-module. Notice for each $j \in \Delta_e$ and  $1 \leq k \leq t$ that $\psi(j^{(k)})= F_*^e(jf_k)= \bigoplus_{i\in\Delta_e}f^{(k)}_{(i,j)}F_*^e(i)$
and hence the matrix $[ M_S(f_1,e)... M_S(f_t,e)]$
represents the map $\psi$ on the given free-bases (Remark \ref{Rem 2.9}). This proves that $F_*^e(R)$ is isomorphic $\Cok_S[ M_S(f_1,e)... M_S(f_t,e)]$ as $S$ -modules.

2) Since $IF_*^e(R) =0$ and $I \Cok_S[ M_S(f_1,e)... M_S(f_t,e)]=I\Cok \psi=0$, it follows that   $F_*^e(R)$ is isomorphic $\Cok_S[ M_S(f_1,e)... M_S(f_t,e)]$ as $R$ -modules.
\end{proof}

\begin{proposition}\label{P5.24}
 Let $S$ denote the ring $K[x_1,...,x_n]$ or the ring $K[\![x_1,...,x_n]\!]$. For each $1 \leq j \leq t$, let $ f_j = x_1^{d_{(1,j)}} \ldots x_n^{d_{(n,j)}}$  where $d_{(i,j)} \in \mathbb{Z}_{+}$ for all $i$ and $j$ and set
  $$ \mathcal{G}_j = \{x_1^{m_{(1,j)}} \ldots x_n^{m_{(n,j)}} \, | \, 0 \leq m_{(i,j)}\leq d_{(i,j)}, 1 \leq i \leq n \} .$$
  Let $I$ be the monomial ideal $I=(f_1,...,f_n)$ and let $\mathfrak{J}$ be the set of all ideals $J=(g_1,...,g_t)$ where $g_j \in \mathcal{G}_j  $ for all $1 \leq j \leq t$. If $R=S/I$, then
   $$F_*^e(R)=\bigoplus_{J \in \mathfrak{J}}[S/J]^{\oplus \alpha_e(J)} \text{ where } \alpha_e(J) \in \mathbb{Z}_{+}$$
with the convention that $ [S/J]^{\oplus \alpha_e(J)} = \{0\}$ when $ \alpha_e(J)=0$.

\end{proposition}

\begin{proof}
For every $1\leq j \leq t$, notice that  the proof of  Proposition  \ref{L4.25}  shows that each  column and each  row of  $M_S(f_j,e)$ contains only one none zero element of the set $\mathcal{G}_j $.  Therefore, if $A=[ M_S(f_1,e)... M_S(f_t,e)]$, then each  row  of $A$ contains only $t$ none zero elements and each column contains only one none zero element such that all of them belong to  $\mathcal{G}_1 \cup ... \cup \mathcal{G}_t$. Let $\Upsilon$ be the set of all $1 \times t$ matrix of the form $\left[
                                                 \begin{array}{ccc}
                                                   g_1 & \ldots & g_t \\
                                                 \end{array}
                                               \right] $ with $g_j \in \mathcal{G}_j$ for all $1\leq j \leq t$. Using the row and column operations,   the matrix $A$ is equivalent to an $r_e \times tr_e$ matrix  of the form
\begin{equation}\label{Eq4.39}
  A \sim \left[
           \begin{array}{ccccc}
             A_1 &   &   &   &   \\
               &   & \ddots &   &   \\
               &   &   &   & A_{r_e} \\
           \end{array}
         \right]
\end{equation}

where $A_i \in \Upsilon$ for all $ i \in \{1,...,r_e\}$.
Notice for every $ i \in \{1,...,r_e\}$ that $\Cok_SA_i= S/J$ for some $J \in \mathfrak{J}$  and that the Proposition \ref{P4.19} implies  that $ F_*^e(R)=\Cok_SA=\bigoplus_{i=0}^{r_e}\Cok_SA_i$ where $A_i \in \Upsilon$ .
Therefore, we can write  $$F_*^e(R)=\bigoplus_{J \in \mathfrak{J}}[S/J]^{\oplus \alpha_e(J)} \text{ where } \alpha_e(J) \in \mathbb{Z}_{+}$$
with the convention that $ [S/J]^{\oplus \alpha_e(J)} = \{0\}$ when $ \alpha_e(J)=0$. Since $\Omega$ is finite set, the set $\Upsilon$ is also a finite set.
\end{proof}

Proposition \ref{P5.24} implies the following result. However, the following result can be obtained from   \cite[Example 1.3 (v)]{TT} but the above proposition provides another proof.

 \begin{theorem}\label{P5.25}
Let $S$ denote the ring $K[x_1,...,x_n]$ or the ring $K[\![x_1,...,x_n]\!]$.  Let $I$ be a monomial ideal in $S$  generated by the monomials $f_1,...,f_t$. If $R=S/I$, then $R$ has Finite F-representation type on $R$.
\end{theorem}

%%%%%%%%%%%%%%%%%%%%%%%%%%%%%%%%%%%%%%%%%%%%%%%%%%%%%%%%%%%%%%%%%%%%%%%%%%%%%%%%%%
%%%%%%%%%%%%%%%%%%%%%%%%%%%%%%%%%%%%%%%%%%%%%%%%%%%%%%%%%%%%%%%%%%%%%%%%%%%%%%%%%

\chapter{F-signature of specific hypersurfaces}
\label{chapter:F-signature of specific hypersurfaces}

Recall from Remark \ref{C2.10} that if $R$ is an $F$-finite local ring,  for every $e \in \mathbb{N}$ there exist a unique nonnegative integer $a_e$ and an $R$-module $M_e$ that has no free direct summand  such that $F_*^e(R)=R^{a_e}\oplus M_e$ and $a_e(R)=\sharp (F_*^e(R),R)=a_e$. Now we are ready to define the $F$-signature as it appears in \cite{KT} as follows.

\begin{definition}\label{D6.1}
\emph{Let  $(R,\mathfrak{m},K)$ be a $d$-dimensional  $F$-finite Noetherian local ring of prime characteristic
$p$. If $[K:K^p] $  is the dimension of $K$ as $K^p$-vector space and  $\alpha (R) =\log_p[K:K^p]$, then the $F$-signature of $R$,  denoted  $\mathbb{S}(R)$, is defined as
\begin{equation*}
\mathbb{S}(R)=\lim_{e\rightarrow \infty} \frac{a_e(R)}{p^{e(d+\alpha (R))}}.
\end{equation*}
}
\end{definition}

\begin{proposition}\label{R6.2}
 If $(R,\mathfrak{m},K)$ is as above, then $\mathbb{S}(R)= \mathbb{S}(\hat{R})$ where $\hat{R}$ is the $\mathfrak{m}$-adic completion of $R$.
\end{proposition}
\begin{proof}
  If $L$ denotes the residue field of the local ring $(\hat{R}, \hat{\mathfrak{m}})$, then $L$ is isomorphic to $\hat{K}$ where $\hat{K}$ is the $\mathfrak{m}$-adic completion of $K$. Since  $\hat{K}$ is isomorphic to $K$, it follows that $\alpha (R) =\alpha (\hat{R})$.
     %Notice that $[K:K^p]=[F_*^1(K):K]$ and  $[L:L^p]=[F_*^1(L):L]$. If $c=[K:K^p]=[F_*^1(K):K]$, we get that $F_*^1(K)=K^{\oplus c}$. According to Theorem \ref{thm 2.21} and  Proposition \ref{L2.4}, we get that
  %\begin{equation*}
  %  F_*^1(L)=F_*^1(K)\otimes_R \hat{R}=K^{\oplus c}\otimes_R \hat{R}=L^{\oplus c}.
  %\end{equation*}
  %This shows that $[F_*^1(L):L]=c= [F_*^1(K):K]$ and consequently
  It is well known that $\dim R= \dim \hat{R}$ \cite[Corollary 10.2.2]{BCA}. Now, if $a_e(R)=a_e$, we can write $F_*^e(R)=R^{a_e}\oplus M_e$ where $M_e$ is an $R$-module that has no free direct summand. As a result, we get $F_*^e(\hat{R})=\hat{R}^{a_e}\oplus \hat{M_e}$. However, if $\hat{R}$ is a direct summand of $\hat{M_e}$, it follows from Proposition \ref{Pro2.22} (a) that $R$ is a direct summand of $M_e$ which is  a contradiction.  Therefore, $a_e(\hat{R})=a_e(R)$ and consequently $\mathbb{S}(R)= \mathbb{S}(\hat{R})$.
\end{proof}
\begin{remark}\label{R6.3}
  If $R=\bigoplus_{n=0}^{\infty}R_n$ is a graded ring with $R_0$
 equals to a field $K$ of  characteristic $p> 0$, then for every $e \in \mathbb{N}$ there exist a unique nonnegative integer $a_e$ and an $R$-module $M_e$ that has no free direct summand  such that $F_*^e(R)=R^{a_e}\oplus M_e$. If $a_e(R)=a_e$ and $\alpha (R) =\log_p[K:K^p]$ where $[K:K^p] $  is the dimension of $K$ as $K^p$-vector space, then the limit $\mathbb{S}(R)=\lim_{e\rightarrow \infty} \frac{a_e(R)}{p^{e(d+\alpha (R))}}$ is well-defined  \cite[Lemma 6.6]{MVK}. Furthermore, if $ \mathfrak{m}$ is the homogenous maximal ideal of $R$, then $a_e(R)=a_e(R_{\mathfrak{m}})$ and $\mathbb{S}(R)=\mathbb{S}(R_{\mathfrak{m}})$ \cite[Lemma 6.6]{MVK}.
 According to Proposition \ref{R6.2}, $\mathbb{S}(R_{\mathfrak{m}})= \mathbb{S}(\widehat{R_{\mathfrak{m}}})$ where $\widehat{R_{\mathfrak{m}}}$ is the $\mathfrak{m}R_{\mathfrak{m}}$-adic completion of the ring $R_{\mathfrak{m}}$. Since $\widehat{R_{\mathfrak{m}}}$ is isomorphic to $\hat{R}$ where $\hat{R}$ is the $\mathfrak{m}$-adic completion of $R$ \cite[Section 22]{AK}, it follows that $\mathbb{S}(R)=\mathbb{S}(\hat{R})$.
\end{remark}

In this chapter, we will compute the $F$-signature of some hypersurfaces.

\section{The F-signature of $ \frac{S[\![u,v]\!]}{f+uv}$ when $f$ is a monomial}
\label{section:The F-signature of uv is a monomial}

We will keep the same notation as in Notation \ref{N4.1} unless otherwise stated. \\

\begin{notation}\label{N1}
\emph{Let $\Delta = \{1,\dots,n \} $ and let  $d, d_1,\dots,d_n$ be real numbers.  For every $1\leq s \leq n-1$, define
\begin{equation*}
  W^{(n)}_s = \sum_{j_1,\dots,j_s \in \Delta } [(d-d_{j_1})\dots(d- d_{j_s})(\prod_{j \in \Delta \setminus \{j_1,\dots,j_s\}}d_j)]
\end{equation*}
\begin{equation*}
W^{(n)}_n= \prod_{i=1}^{n}(d-d_{i}) \text{  and  }  W^{(n)}_0= \prod_{i=1}^{n}d_{i}.
\end{equation*}
}
\end{notation}
For example, if $d, d_1,d_2,d_3,d_4$ are real numbers, we get that $W^{(4)}_0=d_1 d_2 d_3 d_4$
\begin{eqnarray*}
  W^{(4)}_1 &=&(d-d_1)d_2d_3d_4 + (d-d_2)d_1d_3d_4+(d-d_3)d_1d_2d_4 \\
   &+& (d-d_4)d_1d_2d_3
\end{eqnarray*}

\begin{eqnarray*}
  W^{(4)}_2 &=& (d-d_1)(d-d_2)d_3d_4+(d-d_1)(d-d_3)d_2d_4+ (d-d_1)(d-d_4)d_2d_3 \\
    &+& (d-d_2)(d-d_3)d_1d_4 + (d-d_2)(d-d_4)d_1d_3
\end{eqnarray*}

\begin{eqnarray*}
  W^{(4)}_3 &=&(d-d_2)(d-d_3)(d-d_4)d_1+ (d-d_1)(d-d_3)(d-d_4)d_2 \\
    &+&(d-d_1)(d-d_2)(d-d_4)d_3+(d-d_1)(d-d_2)(d-d_3)d_4
\end{eqnarray*}
and $W^{(4)}_4=(d-d_1)(d- d_2)(d- d_3)(d- d_4).$

According to the above notation,  we can observe the following remark.
\begin{remark}\label{R5.3}
Let  $d, d_1,\dots,d_n$ be real numbers where $n\geq 1$ and let $W^{(n)}_j$ be defined on  $d, d_1,\dots,d_n$ as in \ref{N1}. If $d_{n+1}$ is a real number, then  $W^{(n+1)}_j$  is defined on  $d, d_1,\dots,d_n,d_{n+1}$  for all $ 1 \leq j \leq n $ as follows: $$W^{(n+1)}_j= d_{n+1} W^{(n)}_j+(d-d_{n+1}) W^{(n)}_{j-1}.$$
Furthermore, $W^{(n+1)}_0=d_{n+1}W^{(n)}_0$ and $W^{(n+1)}_{n+1}=(d-d_{n+1})W^{(n)}_n$.
\end{remark}

%\begin{remark}\label{R5.3}
%If $n\geq 2$, then  $W^{(n)}_j= (d-d_n)W^{(n-1)}_{j-1} + d_nW^{(n-1)}_{j}$ for every $ 1 \leq j \leq n-1 $
%\end{remark}

 The following lemma is needed to prove Proposition \ref{P8.4}.
 \begin{lemma}\label{L5.4}
  If $r$ , $q$, $d_j$ and $u_j$ are real numbers for all $1 \leq j \leq n$, then

  \begin{equation*}
    \prod_{j=1}^{n}(d_jr + \frac{q(d-d_j)}{d}+ u_j) = \sum_{j=0}^{n} \frac{q^j}{d^j}W^{(n)}_jr^{n-j}+ \sum_{c=0}^{n-1}g^{(n)}_c(q)r^{c}
  \end{equation*}

 where $g^{(n)}_c(q)$ is a polynomial in $q$ of degree $n-1-c$ for all $0\leq c \leq n-1$.
 \end{lemma}

 \begin{proof}
 By induction on $n$, we will prove this lemma. It is clear when $n=1$. The induction hypothesis implies that
% \begin{eqnarray*}
 % \prod_{j=1}^{n+1}(d_jr + \frac{q(d-d_j)}{d}+ u_j) &=& \prod_{j=1}^{n}(d_jr + \frac{q(d-d_j)}{d}+ u_j)(d_{n+1}r + \frac{q(d-d_{n+1})}{d}+ u_{n+1}) \\
 %   &=& W^{(n)}_0d_{n+1}r^{n+1} + \frac{q}{d}d_{n+1}W^{(n)}_1r^{n}+ \frac{q^2}{d^2}d_{n+1}W^{(n)}_2r^{n-2}+\ldots \\
  %  &+& \frac{q^{n-1}}{d^{n-1}}d_{n+1}W^{(n)}_{n-1}r^{2}+ \frac{q^n}{d^n}d_{n+1}W^{(n)}_nr +\sum_{c=0}^{n-1}d_{n+1}g^{(n)}_c(q)r^{c+1}\\
  %  &+&  \frac{q(d-d_{n+1})}{d}W^{(n)}_0r^n + \frac{q^{2}}{d^{2}}(d-d_{n+1})W^{(n)}_1r^{n-1}+  \\
  %  &+& \frac{q^3}{d^3}(d-d_{n+1})W^{(n)}_2r^{n-2}+\ldots + \frac{q^{n}}{d^{n}}(d-d_{n+1})W^{(n)}_{n-1}r + \\
  %  &+&  \frac{q^{n+1}}{d^{n+1}}(d-d_{n+1})W^{(n)}_n +\sum_{c=0}^{n-1}\frac{q(d-d_{n+1})}{d}g^{(n)}_c(q)r^{c} \\
  %  &+& u_{n+1}W^{(n)}_0r^n + u_{n+1}\frac{q}{d}W^{(n)}_1r^{n-1}+ u_{n+1}\frac{q^2}{d^2}W^{(n)}_2r^{n-2}+\ldots \\
  %  &+& u_{n+1}\frac{q^{n-1}}{d^{n-1}}W^{(n)}_{n-1}r+ u_{n+1}\frac{q^n}{d^n}W^{(n)}_n +\sum_{c=0}^{n-1}u_{n+1}g^{(n)}_c(q)r^{c}
 %\end{eqnarray*}

 \begin{eqnarray*}
   \prod_{j=1}^{n+1}(d_jr + \frac{q(d-d_j)}{d}+ u_j) &=& (d_{n+1}r + \frac{q(d-d_{n+1})}{d}+ u_{n+1})\prod_{j=1}^{n}(d_jr + \frac{q(d-d_j)}{d}+ u_j) \\
   &=& (d_{n+1}r + \frac{q(d-d_{n+1})}{d}+ u_{n+1})(\sum_{j=0}^{n} \frac{q^j}{d^j}W^{(n)}_jr^{n-j}+ \\
   & & \sum_{c=0}^{n-1}g^{(n)}_c(q)r^{c})\\
     &=& A+B+C
 \end{eqnarray*}
 where
 \begin{eqnarray*}
     A &=&  d_{n+1}r (\sum_{j=0}^{n} \frac{q^j}{d^j}W^{(n)}_jr^{n-j}+ \sum_{c=0}^{n-1}g^{(n)}_c(q)r^{c})\\
      &=&  \sum_{j=0}^{n}d_{n+1} \frac{q^j}{d^j}W^{(n)}_jr^{n-j+1}+\sum_{c=0}^{n-1}d_{n+1}g^{(n)}_c(q)r^{c+1}\\
    B &=& \frac{q(d-d_{n+1})}{d}(\sum_{j=0}^{n} \frac{q^j}{d^j}W^{(n)}_jr^{n-j}+ \sum_{c=0}^{n-1}g^{(n)}_c(q)r^{c})\\
      &=& \sum_{j=0}^{n}(d-d_{n+1}) \frac{q^{j+1}}{d^{j+1}}W^{(n)}_jr^{n-j} +\sum_{c=0}^{n-1}\frac{q(d-d_{n+1})}{d}g^{(n)}_c(q)r^{c}\\
    C &=&  u_{n+1}  (\sum_{j=0}^{n} \frac{q^j}{d^j}W^{(n)}_jr^{n-j}+ \sum_{c=0}^{n-1}g^{(n)}_c(q)r^{c})\\
     &=&  \sum_{j=0}^{n} u_{n+1}\frac{q^j}{d^j}W^{(n)}_jr^{n-j}+\sum_{c=0}^{n-1}u_{n+1}g^{(n)}_c(q)r^{c}.
 \end{eqnarray*}

 Write $A=A_1+A_2$ where
 \begin{eqnarray*}
   A_1 &=& \sum_{j=0}^{n}d_{n+1} \frac{q^j}{d^j}W^{(n)}_jr^{n-j+1} \text{ and } \\
   A_2 &=& \sum_{c=0}^{n-1}d_{n+1}g^{(n)}_c(q)r^{c+1}
 \end{eqnarray*}
 and write $B=B_1+B_2$ where
 \begin{eqnarray*}
   B_1 &=& \sum_{j=0}^{n}(d-d_{n+1}) \frac{q^{j+1}}{d^{j+1}}W^{(n)}_jr^{n-j} \text{ and } \\
   B_2 &=& \sum_{c=0}^{n-1}\frac{q(d-d_{n+1})}{d}g^{(n)}_c(q)r^{c}.
 \end{eqnarray*}

  Notice that
 \begin{eqnarray*}
 A_1+B_1&=& \sum_{j=0}^{n}d_{n+1} \frac{q^j}{d^j}W^{(n)}_jr^{n-j+1} + \sum_{j=0}^{n}(d-d_{n+1}) \frac{q^{j+1}}{d^{j+1}}W^{(n)}_jr^{n-j} \\
  &=& d_{n+1}W^{(n)}_0r^{n+1}+ \sum_{j=1}^{n}d_{n+1} \frac{q^j}{d^j}W^{(n)}_jr^{n-j+1}\\
   & & +\sum_{j=0}^{n-1}(d-d_{n+1}) \frac{q^{j+1}}{d^{j+1}}W^{(n)}_jr^{n-j}+ (d-d_{n+1}) \frac{q^{n+1}}{d^{n+1}}W^{(n)}_n \\
   &=& d_{n+1}W^{(n)}_0r^{n+1}+ \sum_{j=1}^{n}d_{n+1} \frac{q^j}{d^j}W^{(n)}_jr^{n-j+1}\\
   & & +\sum_{j=1}^{n}(d-d_{n+1}) \frac{q^{j}}{d^{j}}W^{(n)}_{j-1}r^{n-j+1}+ (d-d_{n+1}) \frac{q^{n+1}}{d^{n+1}}W^{(n)}_n \\
   &=& d_{n+1}W^{(n)}_0r^{n+1}+ \sum_{j=1}^{n}\frac{q^j}{d^j}[d_{n+1} W^{(n)}_j+(d-d_{n+1}) W^{(n)}_{j-1}]r^{n-j+1} \\
   & & +(d-d_{n+1}) \frac{q^{n+1}}{d^{n+1}}W^{(n)}_n.
\end{eqnarray*}

 Now apply Remark \ref{R5.3} to get that
\begin{equation}\label{E.8.14}
 A_1+B_1= \sum_{j=0}^{n+1} \frac{q^j}{d^j}W^{(n+1)}_jr^{n+1-j}.
\end{equation}

Now define
 $$g^{(n+1)}_n(q)=d_{n+1}g^{(n)}_{n-1}(q)+u_{n+1}W^{(n)}_0,$$
$$g^{(n+1)}_0(q)=\frac{q}{d}(d-d_{n+1})g^{(n)}_0(q)+u_{n+1}\frac{q^n}{d^n}W^{(n)}_n+u_{n+1}g^{(n)}_0(q),$$
and
\begin{equation*}
 g^{(n+1)}_i(q)= d_{n+1}g^{(n)}_{i-1}(q)+ \frac{q(d-d_{n+1})}{d}g^{(n)}_{i}(q)+ u_{n+1}\frac{q^{n-i}}{d^{n-i}}W^{(n)}_{n-i}+u_{n+1}g^{(n)}_{i}(q)
\end{equation*}
 for every $1\leq i \leq n-1 $.

  Since $g^{(n)}_i(q)$ is a polynomial in $q$ of degree $n-1-i$ for all $0\leq i \leq n-1$, it follows from the above definitions that $g^{(n+1)}_i(q)$ is a polynomial
   in $q$ of degree $n-i$ for all $0\leq i \leq n $.

Notice that
\begin{eqnarray*}
   A_2+ B_2 + C  &=& \sum_{c=0}^{n-1}d_{n+1}g^{(n)}_c(q)r^{c+1} + \sum_{c=0}^{n-1}\frac{q(d-d_{n+1})}{d}g^{(n)}_c(q)r^{c} \\
    & &  + \sum_{j=0}^{n} u_{n+1}\frac{q^j}{d^j}W^{(n)}_jr^{n-j}+\sum_{c=0}^{n-1}u_{n+1}g^{(n)}_c(q)r^{c} \\
    &=& \sum_{c=0}^{n-2}d_{n+1}g^{(n)}_c(q)r^{c+1}+d_{n+1}g^{(n)}_{n-1}(q)r^{n}\\
    & & +\frac{q(d-d_{n+1})}{d}g^{(n)}_0(q)+ \sum_{c=1}^{n-1}\frac{q(d-d_{n+1})}{d}g^{(n)}_c(q)r^{c} \\
    & & + u_{n+1}W^{(n)}_0r^{n}+ \sum_{j=1}^{n-1} u_{n+1}\frac{q^j}{d^j}W^{(n)}_jr^{n-j}+ u_{n+1}\frac{q^n}{d^n}W^{(n)}_n\\
    & & + u_{n+1}g^{(n)}_0(q)+\sum_{c=1}^{n-1}u_{n+1}g^{(n)}_c(q)r^{c}\\
    &=&[\frac{q}{d}(d-d_{n+1})g^{(n)}_0(q)+u_{n+1}\frac{q^n}{d^n}W^{(n)}_n+u_{n+1}g^{(n)}_0(q)] \\
    & & +[\sum_{c=0}^{n-2}d_{n+1}g^{(n)}_c(q)r^{c+1} + \sum_{c=1}^{n-1}\frac{q(d-d_{n+1})}{d}g^{(n)}_c(q)r^{c} \\
    & &  + \sum_{j=1}^{n-1} u_{n+1}\frac{q^j}{d^j}W^{(n)}_jr^{n-j}+\sum_{c=1}^{n-1}u_{n+1}g^{(n)}_c(q)r^{c}] \\
    & & + [d_{n+1}g^{(n)}_{n-1}(q)r^n+u_{n+1}W^{(n)}_0r^n].
\end{eqnarray*}
 As a result, we can write
\begin{eqnarray*}
  A_2+ B_2 + C   &=&[\frac{q}{d}(d-d_{n+1})g^{(n)}_0(q)+u_{n+1}\frac{q^n}{d^n}W^{(n)}_n+u_{n+1}g^{(n)}_0(q)] \\
    & & +[\sum_{i=1}^{n-1}d_{n+1}g^{(n)}_{i-1}(q)r^{i} + \sum_{i=1}^{n-1}\frac{q(d-d_{n+1})}{d}g^{(n)}_i(q)r^{i} \\
    & &  + \sum_{i=1}^{n-1} u_{n+1}\frac{q^{n-i}}{d^{n-i}}W^{(n)}_{n-i}r^{i}+\sum_{i=1}^{n-1}u_{n+1}g^{(n)}_i(q)r^{i}] \\
    & & + [d_{n+1}g^{(n)}_{n-1}(q)r^n+u_{n+1}W^{(n)}_0r^n]\\
    &=&[\frac{q}{d}(d-d_{n+1})g^{(n)}_0(q)+u_{n+1}\frac{q^n}{d^n}W^{(n)}_n+u_{n+1}g^{(n)}_0(q)] \\
    & & +\sum_{i=1}^{n-1}[d_{n+1}g^{(n)}_{i-1}(q)  +   \frac{q(d-d_{n+1})}{d}g^{(n)}_i(q)  \\
    & &  +    u_{n+1}\frac{q^{n-i}}{d^{n-i}}W^{(n)}_{n-i} + u_{n+1}g^{(n)}_i(q)]r^{i} \\
    & & + [d_{n+1}g^{(n)}_{n-1}(q)+u_{n+1}W^{(n)}_0]r^n\\
   &=& \sum_{i=0}^{n}g^{(n+1)}_i(q)r^{i}
\end{eqnarray*}

 and hence
 \begin{eqnarray*}
   \prod_{j=1}^{n+1}(d_jr + \frac{q(d-d_j)}{d}+ u_j) &=& A+B+C =  (A_1+B_1)+ (A_2+B_2+C) \\
   &=&   \sum_{j=0}^{n+1} \frac{q^j}{d^j}W^{(n+1)}_jr^{n+1-j}+ \sum_{c=0}^{n}g^{(n+1)}_c(q)r^{c}.
 \end{eqnarray*}
 \end{proof}

%\begin{notation}\label{N1}
%Let $\Delta = \{1,\dots,n \} $ and let  $d_1,\dots,d_n$ be positive integers  For every $s \in \{1,\dots,n-1\}$, let $\Delta(j_1,\dots,j_s)= \Delta - \{j_1,\dots,j_s\}$.

%Define
%\begin{equation*}
 % W_s = \sum_{j_1,\dots,j_s \in \Delta } [(d-d_{j_1})\dots(d- d_{j_s})(\prod_{j \in \Delta(j_1,\dots,j_s)}d_j)]
%\end{equation*}
%For example, if $d_1,d_2,d_3,d_4$ are positive integers and $d= \max \{d_1,d_2,d_3,d_4\}$, we get that
%\begin{eqnarray*}
 % W_1 &=&(d-d_1)d_2d_3d_4 + (d-d_2)d_1d_3d_4+(d-d_3)d_1d_2d_4 \\
 %  &+& (d-d_4)d_1d_2d_4
%\end{eqnarray*}

%\begin{eqnarray*}
 % W_2 &=& (d-d_1)(d-d_2)d_3d_4+(d-d_1)(d-d_3)d_2d_4+ (d-d_1)(d-d_4)d_2d_3 \\
 %   &+& (d-d_2)(d-d_3)d_1d_4 + (d-d_2)(d-d_4)d_1d_3
%\end{eqnarray*}

%\begin{eqnarray*}
 % W_3 &=&(d-d_2)(d-d_3)(d-d_4)d_1+ (d-d_1)(d-d_3)(d-d_4)d_2 \\
  %  &+&(d-d_1)(d-d_2)(d-d_4)d_3+(d-d_1)(d-d_2)(d-d_3)d_4
%\end{eqnarray*}

%\end{notation}

\begin{theorem}\label{P8.4}
 Let  $f= x_1^{d_1}\dots x_n^{d_n}$ be a monomial in $S=K[\![x_1,\dots,x_n]\!]$ where $d_j  $ is a positive  integer for each $1\leq j \leq n $. If $d = \max \{d_1,\dots, d_n \}$ and $R^{\bigstar}=S[\![u,v]\!]/(f+uv)$, then  the F-signature of $R^{\bigstar}$ is given by  \\
 \begin{equation}
 \mathbb{S}(R^{\bigstar})=\frac{2}{d^{n+1}}\left[\frac{d_1d_2\dots d_n}{n+1} + \frac{W^{(n)}_1}{n}+\dots+\frac{W^{(n)}_s}{n-s+1}+\dots+\frac{W^{(n)}_{n-1}}{2}\right]
\end{equation}
 where $W^{(n)}_1,\dots,W^{(n)}_{n-1}$ are defined as in Notation \ref{N1}.

 % If $d = \max \{d_1,\dots, d_n \}$ and  $\Delta = \{1,\dots,n \} $.  For every $s \in \{1,\dots,n-1\}$, set  $\Delta(j_1,\dots,j_s)= \Delta \diagdown \{j_1,\dots,j_s\}$ and define
%\begin{equation*}
%  W^{(n)}_s = \sum_{j_1,\dots,j_s \in \Delta } [(d-d_{j_1})\dots(d-d_{j_s})(\prod_{j \in \Delta(j_1,\dots,j_s)}d_j)]
%\end{equation*} If  $R^{\bigstar}=S[[u,v]]/(f+uv)$, then  the F-signature of $R^{\bigstar}$ is given by  \\
%\begin{equation}
 %S(R^{\bigstar})=\frac{2}{d^{n+1}}\left[\frac{d_1d_2\dots d_n}{n+1} + \frac{W^{(n)}_1}{n}+\dots+\frac{W^{(n)}_s}{n-s+1}+\dots+\frac{W^{(n)}_{n-1}}{2}\right]
%\end{equation}

\end{theorem}

\begin{proof}
 Let $R=S/fS$ and  $R^{\bigstar}=S[\![u,v]\!]/(f+uv)$. Set $[K:K^p]=b$ and recall from Notation \ref{N4.1} that $\Lambda_e$ is the basis of $K$ as $K^q$-vector space where $q=p^e$.  We know from Proposition \ref{C4.19} (d)  that
 \begin{equation}\label{E1}
  \sharp ( F_*^e(R^{\bigstar}),R^{\bigstar})= r_e + 2 \sum_{k=1}^{q-1} \sharp ( \Cok_S(A^k),R)
 \end{equation}
 where $r_e=b^eq^n$ ,  $A=M_S(f,e)$ and $A^k=M_S(f^k,e)$. Since $f^k$ is a monomial, it follows from Proposition \ref{L4.25} that  the matrix $A^k=M_S(f^k,e)$ is equivalent to a diagonal matrix $D$ whose diagonal entries are taken from the set $\{x_1^{u_1}\dots x_n^{u_n} |0\leq u_j\leq d_j  \text{  for   all } 1\leq j \leq n  \} $.
This makes $\Cok_S(A^k)= \Cok_S(D)$ and consequently the number $\sharp ( \Cok_S(A^k),R) $ is exactly the same as the number of the $n$-tuples $(\alpha_1,\dots,\alpha_n)$ with $0 \leq \alpha_j \leq q-1 $  satisfying that
\begin{equation}\label{EE2}
 F_*^e(\lambda x_1^{kd_1+\alpha_1}\dots x_n^{kd_n+\alpha_n}) = x_1^{d_1}\dots x_n^{d_n}F_*^e(\lambda x_1^{s_1}\dots x_n^{s_n})
\end{equation}
  where $ s_1,..,s_n \in \{0,\dots,q-1 \}$ for all $\lambda \in \Lambda_e$. However, an $n$-tuple $(\alpha_1,\dots,\alpha_n)$ with $ 0\leq \alpha_j \leq q-1 $ will satisfy  (\ref{EE2}) if and only if  $\alpha_j= d_j(q-k) + s_j$  for some $ 0\leq s_j \leq q-1 $  for each $ 1 \leq j \leq n $. As a result,  there exists  $n$-tuples $(\alpha_1,\dots,\alpha_n)\in \mathbb{Z}^n$ with $ 0\leq \alpha_j \leq q-1 $ satisfying  (\ref{EE2}) if and only if $d_j(q-k) < q$  for all $ 1 \leq j \leq n $.
Set  $ N_j(k):= \{ \alpha_j \in \mathbb{Z} \, | \, d_j(q-k)\leq \alpha_j < q \}$ for all $ 1 \leq j \leq n $.  Therefore,
\begin{equation}\label{E18}
  \sharp ( \Cok_S(A^k),R)= b^e|N_1(k)||N_2(k)|\dots|N_n(k)|
\end{equation}
where
\begin{equation*}
  |N_j(k)|=\left\{
             \begin{array}{ll}
               q-d_jq+d_jk, & \hbox{ if } d_j(q-k) < q \\
               0, & \hbox{ otherwise}
             \end{array}
           \right.
\end{equation*}

Let  $d = \max \{d_1,\dots, d_n \}$. Notice that
%, then $\sharp ( \Cok_S(A^k),R)\neq 0 $ if and only if $ N_j(k)\neq 0$ for all $ 1 \leq j \leq n $ if and only if $\frac{q(d-1)}{d} < k$.
\begin{eqnarray*}
 %\nonumber to remove numbering (before each equation)
\sharp ( \Cok_S(A^k),R)\neq 0 & \Leftrightarrow &  |N_j(k)|\neq 0 \text{ for all } j \in \{ 1,\dots,n \}   \\
   & \Leftrightarrow & d_j(q-k) < q \text{ for all }j \in \{ 1,\dots,n \}\\
  & \Leftrightarrow & d(q-k) < q \\
   & \Leftrightarrow & \frac{q(d-1)}{d} < k.\\
\end{eqnarray*}
Therefore,
\begin{equation}\label{EE18}
  \sharp ( \Cok_S(A^k),R) = b^e\prod_{j=1}^{n}(q-d_jq+d_jk) \text{ whenever } k> \frac{q(d-1)}{d}.
\end{equation}

Let $q=du+t$ where $t\in \{0,..,d-1 \}$. If $t \neq 0$, then one can  verify that
\begin{equation}\label{EE19}
 \frac{q(d-1)}{d} < q-\frac{q-t}{d} < \frac{q(d-1)}{d}+1.
\end{equation}
Therefore, \\ $\sharp ( \Cok_S(A^k),R)\neq 0$ if and only if $ k \in \{q-\frac{q-t}{d}+r \, | \, r \in \{0,\dots,\frac{q-t}{d}-1\} \}.$

 %it follows by \ref{EE19} that
%\begin{eqnarray*}
%\sharp ( Cok_S(A^k),R)\neq 0&\Leftrightarrow& \frac{q(d-1)}{d} < k \\
%&\Leftrightarrow& k \in \{ q-\frac{q-t}{d},\dots,q-1 \}\\
%&\Leftrightarrow& k \in \{q-\frac{q-t}{d}+r \, | \, r \in \{0,\dots,\frac{q-t}{d}-1\} \} \\
%\end{eqnarray*}
However, if $t=0$,  it follows that $\frac{q(d-1)}{d}= q- \frac{q}{d}\in \mathbb{Z}$ and consequently

\begin{equation}\label{EEE19}
 \sharp ( \Cok_S(A^k),R)\neq 0 \Leftrightarrow k \in \{q-\frac{q}{d}+r \, | \, r \in \{1,\dots,\frac{q}{d}-1\} \}.
\end{equation}

%$\sharp ( \Cok_S(A^k),R)\neq 0$ if and only if $k \in \{q-\frac{q}{d}+r \, | \, r \in \{1,\dots,\frac{q}{d}-1\} \}$
%\begin{equation*}
%\sharp ( Cok_S(A^k),R)\neq 0 \Leftrightarrow   k \in \{q-\frac{q}{d}+r \, | \, r \in \{1,\dots,\frac{q}{d}-1\} \}
%\end{equation*}
%First assume that $t \neq 0$ and consequently $k \in \{q-\frac{q-t}{d}+r \, | \, r \in \{0,\dots,\frac{q-t}{d}-1\} \}$.

Assume now that  $t\neq 0$. This implies that
 \begin{eqnarray*}
   \sum_{k=1}^{q-1} \sharp ( \Cok_S(A^k),R) &=& b^e \sum_{k=q-\frac{q-t}{d}}^{q-1}\prod_{j=1}^{n}(q-d_jq+d_jk) \text{ (Use \ref{EE18} and \ref{EEE19} })\\
   &=& b^e \sum_{r=0}^{\frac{q-t}{d}-1}\prod_{j=1}^{n}(q-d_jq+d_j(r+q-\frac{q-t}{d})) \\
   &=& b^e\sum_{r=0}^{\frac{q-t}{d}-1}\prod_{j=1}^{n}(d_jr + \frac{q(d-d_j)}{d}+ \frac{d_jt}{d}).\\
 \end{eqnarray*}

%Our task now is to expand $\prod_{j=1}^{n}(d_jr + \frac{q(d-d_j)}{d}+ \frac{d_jt}{d})$ in terms of $r$ so that we can apply Faulhaber's formula \cite{JR} below   and take the limit  in order to compute    the F-signature.

%Let $\Delta = \{1,\dots,n \} $.  For every $s \in \{1,\dots,n-1\}$, let $\Delta(j_1,\dots,j_s)= \Delta \diagdown \{j_1,\dots,j_s\}$.

%Define
%\begin{equation*}
%  W_s = \sum_{j_1,\dots,j_s \in \Delta } [(d-d_{j_1})\dots(d-d_{j_s})(\prod_{j \in \Delta(j_1,\dots,j_s)}d_j)]
%\end{equation*}

Recall from Lemma \ref{L5.4} that

\begin{equation}\label{E61}
 \prod_{j=1}^{n}(d_jr + \frac{q(d-d_j)}{d}+ \frac{d_jt}{d})= \sum_{j=0}^{n} \frac{q^j}{d^j}W^{(n)}_jr^{n-j} +\sum_{c=0}^{n-1}g^{(n)}_c(q)r^{c}
\end{equation}

where  $g^{(n)}_c(q)$ is a polynomial in $q$ of degree $n-1-c$ for all $0 \leq c \leq n-1$.
Set $ \delta =\frac{q-t}{d}-1$.  By Faulhaber's formula \cite{CG}, if $s$ is a positive integer, we get the following polynomial in $\delta$ of degree $s+1$
\begin{equation}\label{E4}
 \sum_{r=1}^{\delta} r^s = \frac{1}{s+1}\sum_{j=0}^s(-1)^j\binom{s+1}{j}B_j\delta^{s+1-j}
\end{equation}
where $B_j$ are  Bernoulli numbers, $B_0=1$ and  $B_1= \frac{-1}{2}$. This makes
\begin{equation}\label{E62}
 \sum_{r=0}^{\delta} r^s=\frac{q^{s+1} }{(s+1)d^{s+1}}+V_s(q)
\end{equation}
where $V_s(q)$ is a polynomial of degree $s$ in $q$.
%Therefore,
%\begin{equation}\label{E5}
%\lim_{e\rightarrow \infty }\frac{1}{p^{e(n+1)}}\sum_{r=1}^{\delta} r^s = \begin{cases} \frac{1}{(n+1)d^{n+1}} & \text{if  } s=n \\ 0 & \text{if  } s<n  \end{cases}
%\end{equation}
 From Faulhaber's formula and the  equations (\ref{E61}), and  (\ref{E62}),   we get that
 \begin{eqnarray*}
   \sum_{k=1}^{q-1} \sharp ( \Cok_S(A^k),R) &=& b^e\sum_{r=0}^{\delta}[\sum_{j=0}^{n} \frac{q^j}{d^j}W^{(n)}_jr^{n-j} +\sum_{c=0}^{n-1}g^{(n)}_c(q)r^{c}] \\
    &=& b^e[\sum_{j=0}^{n} \frac{q^j}{d^j}W^{(n)}_j\sum_{r=0}^{\delta}r^{n-j} +\sum_{c=0}^{n-1}g^{(n)}_c(q)\sum_{r=0}^{\delta}r^{c}]\\
    &=&b^e[\sum_{j=0}^{n} \frac{q^j}{d^j}W^{(n)}_j(\frac{q^{n-j+1} }{(n-j+1)d^{n-j+1}}+V_{n-j}(q)) \\
     & & + \sum_{c=0}^{n-1}g^{(n)}_c(q)(\frac{q^{c+1} }{(c+1)d^{c+1}}+V_c(q))]  \\
     &=& \frac{b^eq^{n+1}}{d^{n+1}}\sum_{j=0}^{n}\frac{W^{(n)}_j}{n-j+1}+ b^e[\sum_{j=0}^{n}\frac{q^j}{d^j}W^{(n)}_jV_{n-j}(q) \\
     & & +\sum_{c=0}^{n-1}g^{(n)}_c(q)\frac{q^{c+1} }{(c+1)d^{c+1}}+\sum_{c=0}^{n-1}g^{(n)}_c(q)V_c(q)].
 \end{eqnarray*}
Since $ \sum_{j=0}^{n}\frac{q^j}{d^j}W^{(n)}_jV_{n-j}(q)$ and $\sum_{c=0}^{n-1}g^{(n)}_c(q)\frac{q^{c+1} }{(c+1)d^{c+1}}+\sum_{c=0}^{n-1}g^{(n)}_c(q)V_c(q)$ are polynomials in $q=p^e$ of degree $n$ and $n-1$ respectively, it follows that

$$ \lim_{e\rightarrow \infty }\frac{1}{b^ep^{e(n+1)}}b^e[\sum_{j=0}^{n}V_{n-j}(q)+\sum_{c=0}^{n-1}g^{(n)}_c(q)\frac{q^{c+1} }{(c+1)d^{c+1}}+\sum_{c=0}^{n-1}g^{(n)}_c(q)V_c(q)]=0.$$
Therefore
\begin{eqnarray*}
\lim_{e\rightarrow \infty }\frac{1}{b^ep^{e(n+1)}}\sum_{k=1}^{q-1} \sharp ( \Cok_S(A^k),R) &=& \frac{1}{d^{n+1}}\sum_{j=0}^{n}\frac{W^{(n)}_j}{n-j+1}.
\end{eqnarray*}
By  the equation (\ref{E1}) and the above equation we conclude that the F-signature of the ring $R^{\bigstar}$ is given by
\begin{equation}\label{E64}
 \mathbb{S}(R^{\bigstar})=\frac{2}{d^{n+1}}\Big[\frac{d_1d_2\dots d_n}{n+1} + \frac{W^{(n)}_1}{n}+\dots+\frac{W^{(n)}_s}{n-s+1}+\dots+\frac{W^{(n)}_{n-1}}{2} \Big].
\end{equation}
Second if  $q=du $, then  $\frac{q(d-1)}{d}= q- \frac{q}{d}\in \mathbb{Z}$ and consequently
\begin{equation*}
\sharp ( \Cok_S(A^k),R)\neq 0 \Leftrightarrow   k \in \{q-\frac{q}{d}+r \, | \, r \in \{1,\dots,\frac{q}{d}-1\} \}.
\end{equation*}
 Therefore
\begin{eqnarray*}
   \sum_{k=1}^{q-1} \sharp ( \Cok_S(A^k),R) &=& b^e\sum_{k=q-\frac{q}{d}+1}^{q-1}\prod_{j=1}^{n}(q-d_jq+d_jk)\\
   &=& b^e \sum_{r=0}^{\frac{q}{d}-2}\prod_{j=1}^{n}(q-d_jq+d_j(r+q-\frac{q}{d}+1)) \\
   &=& b^e \sum_{r=0}^{\frac{q}{d}-1}\prod_{j=1}^{n}(d_jr + \frac{q(d-d_j)}{d}+d_j).\\
 \end{eqnarray*}

By an argument similar to the above argument, we conclude the same result that

\begin{equation}\label{E64}
 \mathbb{S}(R^{\bigstar})=\frac{2}{d^{n+1}}\Big[\frac{d_1d_2\dots d_n}{n+1} + \frac{W^{(n)}_1}{n}+\dots+\frac{W^{(n)}_s}{n-s+1}+\dots+\frac{W^{(n)}_{n-1}}{2}\Big].
\end{equation}
\end{proof}

\begin{remark}
Let $K$ be a perfect field of positive characteristic $p$ and let $R=\frac{K[\![x_1,x_2,u,v]\!]}{(x_1x_2-uv)}$.
Applying Theorem \ref{P8.4} gives  that $\mathbb{S}(R)=\frac{2}{3}$.

 Let  $r,s\geq 2$ be integers. If $A$ is  the Segre product of the polynomial rings $K[x_1,...,x_r]$ and $K[y_1,...,y_s]$, i.e.,
$R$ is the subring of $ K[x_1,...,x_r, y_1,...,y_s]$ generated over $K$ be the monomials $x_iy_j$
for $1 \leq i \leq r$ and $1 \leq j \leq s$, it is well-known that $A$ is isomorphic to the determinantal ring
obtained by killing the size two minors of an $r \times s$ matrix of indeterminates, and that the
dimension of the ring $A$ is $d =r+s-1$. A.Singh in \cite[Example 7]{AS} shows that
$$ \mathbb{S}(A)= \frac{1}{d!}\sum_{i=0}^{s}(-1)^{i}\binom{d+1}{i}(s-i)^d $$
As a result, the $F$-signature of the determinantal ring $\frac{K[x_1,x_2,u,v]}{(x_1x_2-uv)}$   is  $\mathbb{S}(\frac{K[x_1,x_2,u,v]}{(x_1x_2-uv)})=\frac{2}{3}$ and consequently by Remark \ref{R6.3} we get also that $\mathbb{S}(\frac{K[\![x_1,x_2,u,v]\!]}{(x_1x_2-uv)})=\frac{2}{3}$.
\end{remark}

\begin{remark}
Let $K$ be a perfect field of positive characteristic $p$ and let $R=\frac{K[\![x,u,v]\!]}{(x^d-uv)}$.
According to  Theorem \ref{P8.4}, we get  that $\mathbb{S}(R)=\frac{1}{d}$.
However, we can conclude  that $\mathbb{S}(R)=\frac{1}{d}$ from the first case of  \cite[Example 8]{HL}.
\end{remark}

\section{The $F$-signature of $S/fS$ and $S[\![y]\!]/fS[\![y]\!]$ are the same}
\label{section:The Fsignatures are the same}
We  keep the same notation as in Notation \ref{N4.1}.
\begin{proposition}
Let  $S=K[\![x_1,...,x_n]\!]$. If  $f$ is a nonunit nonzero element of $S$ and $y$ is a new variable on $S$, then
the $F$-signature of $S/fS$ and $S[\![y]\!]/fS[\![y]\!]$ are the same and consequently the $F$-signature of $S/fS$ and $S[\![y_1,....,y_m]\!]/fS[\![y_1,...,y_m]\!]$ are the same for every positive integer $m$.
\end{proposition}
\begin{proof}
Let $R=S/fS$ and $B=S[\![y]\!]/fS[\![y]\!]$. For any $e$ and $q=p^e$,  recall that if $A=M_S(f,e)$, then  $F_*^e(S/fS)= \Cok_S(A)$ (Theorem \ref{P5.11}) and $M_{S[[y]]}(f,e)$ is a $q \times q$ matrix over the ring $M_{r_e}(S[\![y]\!])$ (Proposition \ref{EP}) that is given by

\begin{equation}\label{E10}
M_{S[[y]]}(f,e)=\left[
                     \begin{array}{ccc}
                       A &   &   \\
                         &\ddots &  \\
                        &  & A \\
                     \end{array}
                   \right].
\end{equation}

Recall from Proposition \ref{P5.5} that $(A,A^{q-1})$ is a matrix factorization of $f$. If $(A,A^{q-1})$ is a nontrivial matrix factorization of $f$,  it follows   from Proposition \ref{P.24}  that $(A,A^{q-1})$ can be represented uniquely up to equivalence as
$$ (A,A^{q-1})\sim (\phi,\psi)\oplus (f,1)^u \oplus (1,f)^v$$
where $(\phi, \psi)$ is a reduced matrix factorization of $f$ over $S$ (and hence over $S[\![y]\!]$), $u = \sharp(\Cok_S(A),R)$ and $v = \sharp(\Cok_S(A^{q-1}),R)$ (see Corollary \ref{C3.8}).
 In another words, $A$ is equivalent to the matrix $\left[
  \begin{array}{cc}
  \phi &   \\
    & fI_u \\
    \end{array}
    \right]$ where $I_u$ is the $u \times u$ identity matrix. Since $(\phi, \psi)$ is also a reduced matrix factorization  of $f$ in $S[\![y]\!]$, it follows from Proposition \ref{A} that $\Cok_{S[\![y]\!]}(\phi)$ is stable $B$-module and $\Cok_{S[\![y]\!]}\left[\begin{array}{cc}
  \phi &   \\
    & fI_u \\
    \end{array}
    \right] = B^u \oplus \Cok_{S[\![y]\!]}(\phi)$. Using this result and the relation \ref{E10} we conclude that $ F_*^e(B)=B^{qu}\oplus [\Cok_{S[\![y]\!]}(\phi)]^{\oplus q}$ where $[\Cok_{S[\![y]\!]}(\phi)]^{\oplus q}$ is stable $B$-module. This shows that
\begin{equation}
\sharp(F_*^e(B),B)=q \sharp(F_*^e(R),R).
\end{equation}
However, if $ (A,A^{q-1})\sim  (f,1)^{r_e} $ (or $ (A,A^{q-1})\sim  (1,f)^{r_e} $) then we obtain  that $(F_*^e( \frac{S}{f^{q-1}S})=\Cok_SA^{q-1}= \{0\}$ (or $(F_*^e( \frac{S}{fS})=\Cok_SA= \{0\}$) which is impossible. As a result, if $(A,A^{q-1})$ is a trivial matrix factorization of f, then the only possible
case is that $ (A,A^{q-1})\sim  (f,1)^u \oplus (1,f)^v$ where $0 < u,v < r_e$ with $u+v = r_e$ and consequently  $\sharp(F_*^e(R),R)= \sharp(\Cok_S(A),R)= u$ . In this case, $A$ is equivalent to the matrix $\left[
  \begin{array}{cc}
  I_v &   \\
    & fI_u \\
    \end{array}
    \right]$ where $I_u$ (respectively $I_v$) is the $u \times u$ (respectively $v \times v$) identity matrix of the ring $M_{r_e}(S)$. It follows from the relation \ref{E10} that

\begin{equation}\label{E11}
\sharp(F_*^e(B),B)=q \sharp(F_*^e(R),R).
\end{equation}

Notice that $\alpha (B) =\log_p[K:K^p]=\alpha (R)$. Therefore,

\begin{equation*}
  \mathbb{S}(B)=\lim_{q\rightarrow \infty} \frac{\sharp(F_*^e(B),B)}{p^{e(n+\alpha (B))}}=\lim_{q\rightarrow \infty}\frac{p^e \sharp(F_*^e(R),R)}{p^{e(n+\alpha (B))}}=\lim_{q\rightarrow \infty}\frac{ \sharp(F_*^e(R),R)}{p^{e(n-1+\alpha (R))}}=\mathbb{S}(R).
\end{equation*}

\end{proof}

\section{The F-signature of $\frac{S[\![z]\!]}{(f+z^2)}$ when $f$ is a monomial}
\label{section:The Fsignature of z when  is a monomial}

We will keep the same notation as in Notation \ref{N4.1} unless otherwise stated. \\

\begin{theorem} \label{Thm6.9}
 Let $f= x_1^{d_1}\dots x_n^{d_n}$ be a monomial in $S=K[\![x_1,\dots,x_n]\!]$ where $d_j $ is a positive  integer for each $1 \leq j \leq n$ and $K$ is a field of prime characteristic $p>2$ with $[K:K^p] < \infty $. Let $R=S/fS$ and $R^{\sharp}=S[\![z]\!]/(f+z^2)$.  It follows that:  \\
1) If $ d_j=1 $ for each $1 \leq j \leq n$, then $ \mathbb{S}(S[\![z]\!]/(f+z^2))= \frac{1}{2^{n-1}}$. \\
2) If $d = \max \{d_1,\dots, d_n \}\geq 2$, then $ \mathbb{S}(S[\![z]\!]/(f+z^2))= 0$.
\end{theorem}

\begin{proof}
Set $[K:K^p]=b$ and recall from Notation \ref{N4.1} that $\Lambda_e$ is the basis of $K$ as $K^q$-vector space. We know  by  Proposition \ref{L4.13} that
\begin{equation*}
 F_*^e(R^{\sharp})= \Cok_{S[\![z]\!]}\left[
                                              \begin{array}{cc}
                                                A^{\frac{q-1}{2}} & -zI \\
                                                zI &  A^{\frac{q+1}{2}} \\
                                              \end{array}
                                            \right]  \text{  where } A=M_S(f,e).
\end{equation*}
Recall  from Proposition \ref{C4.19} (b) that
 \begin{equation}\label{E12}
  \sharp(F_*^e(R^{\sharp}),R^{\sharp})= \sharp(\Cok_S(A^{\frac{q-1}{2}}),R) + \sharp(\Cok_S(A^{\frac{q+1}{2}}),R).
 \end{equation}
Now, let $k\in \{\frac{q-1}{2}, \frac{q+1}{2}\}$ and
%Since $f^k$ is a monomial, Proposition \ref{P4.3} and  Lemma \ref{L4.25} imply that  the matrix $A^k=M_S(f^k,e)$ is equivalent to a diagonal matrix $D$ whose diagonal entries are taken from the set $\{x_1^{u_1}\dots x_n^{u_n} | u_j \in \{0,\dots, d_j \} \text{  for   all } 1 \leq j \leq n.$
%This makes $\Cok_S(A^k)= \Cok_S(D)$ and consequently the number $\sharp ( \Cok_S(A^k),R) $ is exactly the same as the number of the $n$-tuples $(\alpha_1,\dots,\alpha_n)$ with $ 0\leq \alpha_j\leq q-1 $  satisfying that
%\begin{equation}\label{E2}
% F_*^e(\lambda x_1^{kd_1+\alpha_1}\dots x_n^{kd_n+\alpha_n}) = x_1^{d_1}\dots x_n^{d_n}F_*^e(\lambda x_1^{s_1}\dots x_n^{s_n})
%\end{equation}
%where $0 \leq s_j \leq q-1 $ for all $\lambda \in \Lambda_e$.
%However, an $n$-tuple $(\alpha_1,\dots,\alpha_n)$ with $ 0\leq \alpha_j\leq q-1 $ will satisfy the equation \ref{E2} if and only if  $\alpha_j= d_j(q-k) + s_j$  for some $s_j \in \{0,\dots,q-1 \}$  for each $1 \leq j \leq n$. As a result,  an  $n$-tuples $(\alpha_1,\dots,\alpha_n)\in \mathbb{Z}^n$ will satisfy the equation \ref{E2} if and only if
%\begin{equation*}
% d_j(q-k)\leq \alpha_j < q \text{ for all } 1\leq j \leq n
%\end{equation*}
set  $ N_j(k):= \{ \alpha_j \in \mathbb{Z} \, | \, d_j(q-k)\leq \alpha_j < q \}$ for all $1\leq j \leq n $. Using the same argument that was previously used in the proof of Proposition \ref{P8.4}, it follows that
\begin{equation}\label{Equ17}
  \sharp ( \Cok_S(A^k),R)= b^e|N_1(k)||N_2(k)|\dots|N_n(k)|
\end{equation}
and
\begin{equation}\label{Equ18}
  \sharp ( \Cok_S(A^k),R)= b^e\prod_{j=1}^{n}(q-d_jq+d_jk) \text{ whenever } k> \frac{q(d-1)}{d}.
\end{equation}
%Since  $|N_j(k)|= q-d_jq+d_jk$  for all $1\leq j \leq n $, it follows that
%\begin{equation}\label{EE19}
 % \sharp ( \Cok_S(A^k),R) = b^e\prod_{j=1}^{n}(q-d_jq+d_jk)
%\end{equation}
Now if $d_1=d_2=\dots=d_n=1$, it follows from equation  (\ref{Equ18}) that $ \sharp ( \Cok_S(A^k),R)= b^ek^n$ for $k\in \{\frac{q-1}{2}, \frac{q+1}{2}\}.$  Therefore, Equation (\ref{E12}) implies that
\begin{equation}\label{EE15}
\sharp(F_*^e(R^{\sharp}),R^{\sharp})=b^e[ ( \frac{q-1}{2})^n + ( \frac{q+1}{2})^n]
\end{equation}
and consequently
\begin{equation*}
  \mathbb{S}(R^{\sharp})= \lim_{e\rightarrow \infty } \sharp(F_*^e(R^{\sharp}),R^{\sharp})/b^ep^{en} = \frac{1}{2^{n-1}}.
\end{equation*}
Now let  $d_i= \max \{d_1,\dots,d_n \}$  for some $1\leq i \leq n $.
First assume that $d_i=2 $.  If $k=\frac{q-1}{2}$ , it follows that $d_i(q-k) > q $ and consequently $|N_i(k)|=0 .$ The equation (\ref{Equ17}) implies $\sharp ( \Cok_S(A^k),R)=0$. When $k=\frac{q+1}{2}$ , we get that  $d_i(q-k)=q-1$ and consequently $N_i(k)= \{q-1 \} $ which makes $|N_i(k)|=1 $. Notice that when $k=\frac{q+1}{2}$ and $d_j=1$ , it follows that $|N_j(k)|= \frac{q+1}{2} $. As a result, if $k=\frac{q+1}{2}$, we conclude that
\begin{equation*}
 \sharp ( \Cok_S(A^k),R)=b^e |N_1(k)||N_2(k)|\dots|N_n(k)| \leq b^e (\frac{q+1}{2})^{n-1}.
\end{equation*}
Therefore,

\begin{equation*}
  \sharp(F_*^e(R^{\sharp}),R^{\sharp})= \sharp(\Cok_S(A^{\frac{q-1}{2}}),R) + \sharp(\Cok_S(A^{\frac{q+1}{2}}),R) \leq b^e(\frac{q+1}{2})^{n-1}.
\end{equation*}

As a result,

\begin{equation*}
  \mathbb{S}(R^{\sharp})= \lim_{e\rightarrow \infty } \sharp(F_*^e(R^{\sharp}),R^{\sharp})/b^ep^{en} = 0.
\end{equation*}

Second  assume that $d_i > 2 $. In this case, for every $k\in \{\frac{q-1}{2}, \frac{q+1}{2}\}$, it follows that $d_i(q-k)> q$ and consequently $|N_i(k)|= 0 $. Therefore

\begin{equation*}
  \sharp(F_*^e(R^{\sharp}),R^{\sharp})= \sharp(\Cok_S(A^{\frac{q-1}{2}}),R) + \sharp(\Cok_S(A^{\frac{q+1}{2}}),R) = 0
\end{equation*}
 and consequently

 \begin{equation*}
  \mathbb{S}(R^{\sharp})= \lim_{e\rightarrow \infty } \sharp(F_*^e(R^{\sharp}),R^{\sharp})/b^ep^{en} = 0.
\end{equation*}
\end{proof}
\begin{remark}\label{R6.10}
  Let $A=K[u_1,...,u_n]$ be the  polynomial ring in the indeterminates $u_1,...,u_n$ on the field $K$  and let  $GL(n,K)$ be  the  group of all invertible $n \times n$ matrices over $K$. Suppose that  $G$ is a finite subgroup of  $GL(n,K)$ and   let $|G|$ denote the order of $G$ which is the number of elements of $G$. An element $g \in GL(n,K)$ is called a pseudo-reflection if the rank of the matrix $g-I_n$ is one where $I_n$ is the $n\times n$ identity matrix over $K$.  For every $g=[g_{ij}] \in G$ and $f=f(u_1,...,u_n)\in A$, let $g(f)=f(v_1,...,v_n)$ where $v_i=\sum_{j=1}^{n}g_{ij}u_j$. A polynomial  $f \in A$ is invariant under $G$ if $g(f)=f$ for all $g\in G$. Notice that a polynomial $f$ is invariant under $G$ if and only if its homogeneous components are invariant under $G$ \cite[Chapter 7]{IVA}.  The set of all invariant polynomials is denoted $A^G$, this means that
 \begin{equation*}
    A^G=\{f \in S \,|\, g(f)=f \text{ for all } g \in G \}
  \end{equation*}
 It is well known  that $A^G$ is a graded subring of $A$ that is called the invariant subring of $A$ by $G$.

  If $K$ is a field of characteristic $p>0$,  $A=K[u_1,...,u_n]$  and $G$ is a finite subgroup of  $GL(n,K)$ such that $|G|$ is a unit in $K$ and $G$ has no pseudo-reflection, then the $F$-signature of the invariant subring  $A^G$ is given by $ \mathbb{S}(A^G)=\frac{1}{|G|}$ \cite[Remark 2.3]{Y2} and \cite[Theorem 4.2]{WY}.
\end{remark}
I would like to thank H.Brenner and T.Bridgland who suggested using invariant theory for providing another proof of the first result of    Theorem \ref{Thm6.9}.
\begin{remark}
We can use Remark \ref{R6.10} to   prove the first result of    Theorem \ref{Thm6.9} as follows.

 Let $K$ be a field of  characteristic $p>2$, and let $A=K[u_1,...,u_n]$. Suppose that  $G$ is the subgroup of $GL(n,K)$ consisting of all diagonal matrices whose diagonal entries are all taken from $\{1,-1\}$  with determinant equal to one. If $g \in G$, then $g$ has an even number of diagonal elements that are $-1$. This makes  $G$ have  no pseudo-reflection.  Furthermore, if $H$ is the subgroup of $GL(n,K)$ consisting of all diagonal matrices whose diagonal entries are all taken from $\{1,-1\}$, then $|H|=2^{n}$ and $G$ is a finite subgroup of $H$ such that  $H\setminus G$ is the subset of $H$ consisting of all diagonal matrices whose diagonal entries are all taken from $\{1,-1\}$  with determinant equal to $-1$.  We can define a bijection between $G$ and $H\setminus G$ by sending $g\in G$ to $\tilde{g}\in H \setminus G$  where $\tilde{g}$ is obtained from $g$ by changing the sign of the first diagonal element of $g$. This makes $|G|=|H \setminus G|$. Since $|G|+|H \setminus G|= |H| = 2^n$, we get that $|G|=2^{n-1}$. Now, clearly the monomials $u_1^2,...,u_n^2$ and $u_1...u_n$ are invariant under $G$ and hence $K[u_1^2,...,u_n^2,u_1...u_n]\subseteq A^G$. If $f=u_1^{t_1}...u_n^{t_n}$ is a monomial, then $f=(u_1^{2d_1}...u_n^{2d_n})(u_1^{e_1}...u_n^{e_n})$ where $t_j=2d_j+e_j$ and $e_j \in \{0,1\}$  for all $j=1,...,n$. Therefore, $f$ is invariant under $G$ if and only if  $e_j=1$  for all $j=1,...,n$ or $e_j=0$  for all $j=1,...,n$. This shows that a monomial $f=u_1^{t_1}...u_n^{t_n}$ is invariant under $G$ if and only if $f \in K[u_1^2,...,u_n^2,u_1...u_n]$.   Now, if $F\in A$ is a homogenous polynomial of degree $d$, then $F=f_1+...+f_r$ where $f_j$ is a monomial of degree $d$ for every $j=1,...,r$. Since $g(F)= b_1f_1+...+b_rf_r$ where $b_j\in \{-1,1\}$ for all $j \in \{1,...,r\}$ and $g\in G$, we get that $F$ is invariant under $G$ if and only if $f_j$ is invariant under $G$ for all $j=1,...,r$.  Therefore, $F$ is invariant under $G$ if and only if $F \in K[u_1^2,...,u_n^2,u_1...u_n]$. This shows that  $K[u_1^2,...,u_n^2,u_1...u_n]=A^G$ and consequently  Remark \ref{R6.10} gives that   $$ \mathbb{S}(K[u_1^2,...,u_n^2,u_1...u_n])=\frac{1}{2^{n-1}}.$$ Now, if we define a ring homomorphism $\phi:K[x_1,...,x_n,z]\rightarrow K[u_1,...,u_n]$ by $\phi(x_j)=u_j^2$ for all $j=1,...,n$ and $\phi(z)=u_1...u_n$, we get that
 \begin{equation*}
   \frac{K[x_1,...,x_n,z]}{(x_1...x_n-z^2)} \cong K[u_1^2,...,u_n^2,u_1...u_n] \text{ as rings. }
 \end{equation*}
 Therefore,
 \begin{equation*}
   \mathbb{S}(\frac{K[x_1,...,x_n,z]}{(x_1...x_n-z^2)})=\frac{1}{2^{n-1}}.
 \end{equation*}
 Now, we can use Remark \ref{R6.3} to conclude that
 \begin{equation*}
   \mathbb{S}(\frac{K[\![x_1,...,x_n,z]\!]}{(x_1...x_n-z^2)})=\frac{1}{2^{n-1}}.
 \end{equation*}
 \end{remark}

\begin{remark}
Notice that when $d = \max \{d_1,\dots, d_n \} > 2$, we can prove that $\mathbb{S}(R^{\sharp})=0$ using Fedder's Criteria (Proposition \ref{P2.13}). Indeed, let $\mathfrak{m}$ be the maximal ideal of $S[\![z]\!]$ and let $R^{\sharp}=S[\![z]\!]/(f+z^2)$. If $d = \max \{d_1,\dots, d_n \} > 2$, then $(f+z^2)^{q-1}\in \mathfrak{m}^{[q]}$ which makes, by Fedder's Criteria,   $\sharp(F_*^e(R^{\sharp}),R^{\sharp})=0$ for all $e \in \mathbb{Z}^{+}$. This means clearly that
\begin{equation*}
 \mathbb{S}(R^{\sharp})= \lim_{e\rightarrow \infty } \sharp(F_*^e(R^{\sharp}),R^{\sharp})/b^ep^{en} = 0.
\end{equation*}
\end{remark}

\section{The F-signature of the  ring $S[\![y]\!]/(y^{p^d} +f)$ }
\label{section:The Fsignature of pd}

We will keep the same notation as in Notation \ref{N4.1} unless otherwise stated. \\
\begin{proposition}
 Let $\mathfrak{m}$ be the maximal ideal of  the ring $S=K[\![x_1,...,x_n]\!]$ where $K$ is a field of positive  prime characteristic $p $ with $[K:K^p] = b<\infty $  and let $f$ be a nonzero  element in $\mathfrak{m}$ .
 If $d \in \mathbb{Z}^{+}$ and $\Re=S[\![y]\!]/(y^{p^d}+f)$ ,
   then $\mathbb{S}(\Re)=\frac{\sharp(F_*^d(\Re),\Re)}{b^dp^{nd}}$  .

\end{proposition}
\begin{proof}
Let $[K:K^p] = b $ and let $e \in \mathbb{Z}$ such that $e > d$  . If $A=M_S(f,e)$,   equations (\ref{Eq5.14}) and (\ref{Eq5.15}) in   the proof of Theorem \ref{P5.1} show that
 \begin{equation}\label{E16}
 F_*^e(\Re)= \Cok_{S[\![y]\!]}(M_{S[\![y]\!]}(y^{p^d}+f,e))=(\Cok_{S[\![y]\!]}(yI+A^{p^{e-d}}))^{\oplus p^d}
 \end{equation}
 and
 \begin{equation}\label{E17}
   [ \Cok_{S[\![y]\!]}(yI+A^{p^{e-d}})]^{\oplus p^e}= [F_*^d( \Re)]^ {\oplus ^{b^{e-d}p^{(n+1)(e-d)}}}.
 \end{equation}
If $W=\Cok_{S[\![y]\!]}(yI+A^{p^{e-d}})$,  by  equations \ref{E16}, \ref{E17} we get that
\begin{equation*}
  \sharp(F_*^e(\Re),\Re)=p^d \sharp(W,\Re) \text{ and }   p^e \sharp(W,\Re)=b^{e-d}p^{(n+1)(e-d)}\sharp(F_*^d(\Re),\Re).
\end{equation*}
This makes
\begin{equation*}
\sharp(F_*^e(\Re),\Re)=b^{e-d}p^{n(e-d)}\sharp(F_*^d(\Re),\Re)
\end{equation*}
and consequently
\begin{equation*}
   \mathbb{S}(\Re)=\lim_{e\rightarrow \infty } \frac{\sharp(F_*^e(\Re),\Re)}{b^ep^{en}} =\lim_{e\rightarrow \infty } \frac{b^{e-d}p^{n(e-d)}\sharp(F_*^d(\Re),\Re)}{b^ep^{en}}=\frac{\sharp(F_*^d(\Re),\Re)}{b^dp^{nd}}.
 \end{equation*}
\end{proof}

%%%%%%%%%%%%%%%%%%%%%%%%%%%%%%%%%%%%%%%%%%%%%%%%%%%%%%%%%%%References%%%%%%%%%%%%%%%%%%%%%%%%%%%%%%%%%%%%%%%%%%%%%%%
%%%%%%%%%%%%%%%%%%%%%%%%%%%%%%%%%%%%%%%%%%%%%%%%%%%%%%%%%%%%
%%%%%%%%%%%%%%%%%%%%%%%%%%%%%%%%%%%%%%%%%%%%%%%%%%%%%%%%%%%%

%%%%%%%%%%%%%%%%%%%%%%%%%%%%%%%%%%%%%%%%%%%%%%%%%%%%%%%%%%%%

\begin{thebibliography}{99}
%%%%%%%%%%%%%%%%%%%%%%%%%%%%%%%%%%%%%%%%%%%%%%%%%%%%
\addcontentsline{toc}{section}{\protect\numberline{}{Bibliography}}
%%%%%%%%%%%%%%%%%%%%%%%%%%%%%%%%%%%%%%%%%%%%%%%%%%%%

\bibitem{AL}{\sc I. M.~Aberbach and G.~J.~Leuschke}: \emph{The F-signature and strong F-regularity}, Math. Res. Lett. 10 (2003), 51–-56.

\bibitem{AK}{\sc A.~Altman and  S.~Kleiman}: \emph{A Term of Commutative Algebra}, Worldwide Center of Mathematics, LLC,2013.

\bibitem{AM}{\sc M.~F.~Atyah and I.~G.~Macdonald}: \emph{Introduction to commutative algebra},Addition-Wesley~Publshing~Company,
1969.

\bibitem{NW}{\sc N.R.~Baeth and R.~Wiegand}: \emph{Factorization theory and decomposition of modules},
Amer. Math. Monthly 120 (2013), no.~1, 3-–34.

\bibitem{BH}{\sc W. Bruns and J. Herzog}: \emph{Cohen-Macaulay Rings}, vol. 39, Cambridge
studies in advanced mathematics, 1993.

\bibitem{CG}{\sc J.H.~Conway and R.K.~Guy}: \emph{The Book of Numbers. New York}, NY: Copernicus, 1996. pp.~106-110.

\bibitem{IVA}{\sc  D.~Cox ,J.~Little and D. ~O'Shea}: \emph{ Ideals, Varieties, and Algorithms}. Third edition. Undergraduate Texts in Mathematics. Springer, 2007.

\bibitem{ED}{\sc D.~Eisenbud}: \emph{Homological algebra on a complete intersection, with an application to group representations},
Trans. Amer. Math. Soc. 260 (1980), no. 1, 35–-64.

\bibitem{E}{\sc D.~Eisenbud}: \emph{Commutative Algebra With a View Toward Algebraic Geometry},  GTM, Vol. 150, Springer-Verlag, New York,
1995.

\bibitem{DI}{\sc D.~Eisenbud and I.~Peeva}: \emph{ Matrix factorizations for complete intersections and minimal
free resolutions}. arXiv:1306.2615

\bibitem{RF}{\sc R.~Fedder}: \emph{ F-purity and rational singularity}, Trans. Amer. Math. Soc. 278 (1983) 461–-480

\bibitem{HL}{\sc C.~Huneke and G.~Leuschke}: \emph{Two theorems about maximal Cohen-Macaulay modules},
Math. Ann. 324 (2002), no. 2, 391–-404.

\bibitem{Kunz}{\sc E.~Kunz}: \emph{Intoduction to Commutative Algebra and Algebraic Geometry}, Birkh¨auser, Boston, Basel, Berlin, 1985.

\bibitem{CMR}{\sc G.~Leuschke and R. Wiegand}: \emph{Cohen-Macaulay Representations}, Mathematical Surveys and Monographs,  181.Providence, R.I. : AMS, c2012.

\bibitem{LR}{\sc L. S.~Levy and J. C.~Robson}:  \emph{Matrices and pairs of modules}, J. Algebra 29 (1974),
427-–454.

\bibitem{Mat}{\sc H.~Matsumura}: \emph{Commutative Ring Theory}, Cambridge Studies in Adv. Math. 8, Cambridge
University Press, Cambridge, 1986.

\bibitem{NV}{\sc C.~Nastasescu and F.~Van Oystaeyen}:\emph{Methods of graded rings}, volume 1836 of Lecture Notes
in Mathematics. Springer-Verlag, Berlin, 2004.

\bibitem{JR}{\sc  J.~Rotman}: \emph{ An introduction to homological algebra}. Second edition. Universitext. Springer, New York, 2009.

\bibitem{SZ}{\sc  K.~Schwede, and W.~Zhang}:  \emph{Bertini theorems for F-singularities},Proceedings of the London Mathematical Society, October 2013, Vol.107(4), pp.851--874

\bibitem{S}{\sc  G.~Seibert}:  \emph{The Hilbert-Kunz function of rings of finite Cohen-Macaulay type}, Arch. Math.
69 (1997), 286–296.

\bibitem{TS}{\sc  T.~Shibuta}:  \emph{One-dimensional rings of finite F-representation type},  J. Algebra 332 (2011), 434–-441.

\bibitem{AS}{\sc A.~Singh}: \emph{The F-signature of an affine semigroup ring}, J. Pure Appl. Algebra 196(2–3), 313–321 (2005).

\bibitem{BCA}{\sc B.~Singh}: \emph{Basic Commutative Algebra}, World Scientific Publishing Co. Pte. Ltd., Hackensack, NJ, 2011.

\bibitem{SV}  {\sc K.~Smith,  and M.~Van den Bergh}:   \emph{Simplicity of rings of differential operators in prime
characteristic},Proc. London Math. Soc. (3) 75 (1997), no. 1, 32–-62.

\bibitem{TT} {\sc  S.~Takagi,and R.~Takahashi}: \emph{D-modules over rings with finite F-representation type }.
Res. Lett. 15 (2008), no.~3. 563-–581.

\bibitem{KT}{\sc K.~Tucker}: \emph{F-signature exists}, Invent. Math. 190 (2012), no.~3, 743–-765.

\bibitem{MVK}{\sc M.~Von Korff}: \emph{F-Signature of affine toric varieties},  arXiv:1110.0552.

\bibitem {WY}{\sc K.-i.~Watanabe, and K.~Yoshida}:  \emph{ Minimal relative Hilbert-Kunz multiplicity },Illinois J. Math. 48 (1) (2004) 273--294

\bibitem {RW}{\sc R. Wiegand}:  \emph{Local rings of finite Cohen–Macaulay type}, J. Algebra 203 (1998), 156–-168.

\bibitem{Y}{\sc Y.~Yao}:  \emph{Modules with finite F-representation type}, J. London Math. Soc. (2) 72 (2005), no.
1, 53–-72

\bibitem{Y2}{\sc Y.~Yao}: \emph{Observations on the F-signature of local rings of characteristic p}, J.~Algebra 299 (2006),
no. 1, 198–-218.

\bibitem{YY}{\sc Y.~Yoshino}: \emph{Cohen-Macaulay modules over Cohen-Macaulay rings},
London Mathematical Society Lecture Note Series, vol. 146,
Cambridge University Press, Cambridge, 1990.


\end{thebibliography}
\end{document}